\documentclass[leqno]{imsart}
\RequirePackage[OT1]{fontenc}
\RequirePackage{amsthm,amsmath,amsfonts}

\usepackage{color}
\usepackage{amsmath}
\usepackage{amssymb}
\usepackage{amsthm}
\usepackage{amsfonts}
\usepackage[utf8]{inputenc}
\numberwithin{equation}{section}

\usepackage{geometry}
\newtheorem{theorem}{Theorem}[section]
\newtheorem{lemma}[theorem]{Lemma}
\newtheorem{proposition}[theorem]{Proposition}
\newtheorem{corollary}[theorem]{Corollary}
\newtheorem{remark*}[theorem]{Remark}

\newcommand{\N}{\mathbb{N}}
\newcommand{\R}{\mathbb{R}}

\newcommand{\eps}{\epsilon}

\newcommand{\cL}{\mathcal{L}}
\newcommand{\cN}{\mathcal{N}}
\newcommand{\cM}{\mathcal{M}}

\newcommand{\cB}{\mathcal{B}}

\newcommand\restrict[1]{\raisebox{-.2ex}{$|$}_{#1}}

\begin{document}

\begin{frontmatter}

\title{A Boundary Local Time For One-Dimensional Super-Brownian Motion And Applications}
\runtitle{The boundary local time of super-Brownian motion}

\begin{aug}
\author{\fnms{\hspace{4 mm} Thomas}
  \snm{Hughes}\thanksref{t1}\corref{}\ead[label=e1]{}} 
\ead[label=e3]{hughes@math.ubc.ca}
\thankstext{t1}{Supported by an NSERC CGS-D Scholarship.} 
\affiliation{
The University of British Columbia\thanksmark{t1} \thanksmark{t2}}

\address{Department of Mathematics\\
The University of British Columbia\\
1984 Mathematics Road\\
Vancouver, British Columbia V6T 1Z2\\
\printead{e3}}

\runauthor{T. Hughes}
\end{aug}
\begin{abstract}
For a one-dimensional super-Brownian motion with density $X(t,x)$, we construct a random measure $L_t$ called the boundary local time which is supported on $\partial \{x:X(t,x) = 0\} =: BZ_t$, thus confirming a conjecture of Mueller, Mytnik and Perkins \cite{MMP2017}. $L_t$ is analogous to the local time at $0$ of solutions to an SDE.  We establish first and second moment formulas for $L_t$, some basic properties, and a representation in terms of a cluster decomposition. Via the moment measures and the energy method we give a more direct proof that $\text{dim}(BZ_t) = 2-2\lambda_0> 0$ with positive probability, a recent result of Mueller, Mytnik and Perkins \cite{MMP2017}, where $-\lambda_0$ is the lead eigenvalue of a killed Ornstein-Uhlenbeck operator that characterizes the left tail of $X(t,x)$. In a companion work \cite{HP2018}, the author and Perkins use the boundary local time and some of its properties proved here to show that $\text{dim}(BZ_t) = 2-2\lambda_0$ a.s. on $\{X_t(\R) > 0 \}$.
\end{abstract}

\begin{keyword}[class=MSC]
\kwd[Primary ]{60J68}
\kwd[; Secondary ]{60J55, 60H15, 28A78}
\end{keyword}
 
\begin{keyword}
\kwd{super-Brownian motion}
\kwd{local time}
\kwd{stochastic pde}
\kwd{Hausdorff dimension}
\end{keyword}


\end{frontmatter}

\footskip=25pt

\section{Introduction \& Statement Of Main Results} 
Super-Brownian motion is a Markov process taking values in the space of finite measures on $\R^d$, $\cM_F(\R^d)$, equipped with the topology of weak convergence. We denote this process by $X=(X_t : t \geq 0)$ and denote by $P_{X_0}^X$ and $E_{X_0}^X$, respectively, a probability and its expectation under which $X$ is a super-Brownian motion with initial data $X_0 \in \cM_F(\R^d)$. In one dimension, $X_t$ is almost surely an absolutely continuous random measure and thus has a density we denote by $X(t,x)$. The density is jointly continuous (and will exist) for $t>0$, and is continuous with H\"older index $\frac 1 2 - \epsilon$ in the spatial variable for all $\epsilon > 0$ (see \cite{P2002}, for example, where this is implicit in the proof of Theorem III.4.2). It was shown by Konno and Shiga in \cite{KS1988} and independently by Reimers in \cite{R1989} that $X(t,x)$ satisfies the following stochastic partial differential equation (SPDE):
\begin{equation} \label{e_SPDEdensity}
\frac{\partial X(t,x)}{\partial t} =  \frac{\Delta X(t,x)}{2} + \sqrt{X(t,x)} \dot{W}(t,x),
\end{equation}  
where $\dot{W}(t,x)$ is a space-time white noise. For a complete discussion of such equations, including the precise definition of a solution, see \cite{Walsh} and \cite{KS1988}. \\

Before we discuss our results, we briefly introduce the canonical measure of super-Brownian motion. The canonical measure $\N_0$ is a $\sigma$-finite measure on $C([0,\infty),\cM_F(\R)) \backslash \{0\}$ defined as the weak limit
\begin{equation}
\N_0( X \in \cdot) := \lim_{N\to \infty} N P^X_{ \delta_0 / N}(X \in \cdot).  \nonumber
\end{equation}
When restricted to $\{X_t>0\}$ for $t>0$, $\N_0$ is a finite measure; in particular we have $\N_0 ( \{ X_t > 0 \}) = 2/t$ (see Theorem II.7.2 of \cite{P2002}). $\N_0$ is a fundamental object; it describes the behaviour of a single cluster, that is, the descendants of a single ancestor at the origin, of super-Brownian motion. (Likewise $\N_x$ is a cluster started from $x$ and is just a shift of $\N_0$.) An important fact, which we will describe more precisely later on, is that super-Brownian motion under $P^X_{X_0}$ can be understood as a superposition of canonical clusters. We will use the notation $X_t$ and $X(t,x)$ to denote the superprocess and its density, respectively, under both $P^X_{X_0}$ and $\N_0$. The law of the process will always be clear from context. For a complete overview of the canonical measure, including proofs of the properties just stated, see Section II.7 of \cite{P2002}.  \\

In a recent work by Mueller, Mytnik and Perkins \cite{MMP2017}, the authors studied the small-scale asymptotic behaviour of $X(t,x)$, as well as the boundary of its zero set. We define the random set $Z_t = \{x \in \R : X(t,x) = 0\}$. The boundary of the zero set $BZ_t$ is then defined as
\begin{equation} \label{e_BZt}
BZ_t := \partial Z_t = \{ x \in Z_t : (x-\epsilon,x+\epsilon) \cap Z_t^c \neq \emptyset \,\,\, \forall \epsilon > 0 \}, \nonumber
\end{equation}
where the second equality holds by continuity of the density. The results in \cite{MMP2017} involve an eigenvalue $\lambda_0 \in (\frac 1 2,1)$ which we describe in greater detail shortly. The authors of \cite{MMP2017} show that the left tail of the distribution of $X(t,x)$ behaves like
\begin{equation} \label{e_lefttail}
P^X_{X_0} ( 0< X(t,x) < a ) \asymp t^{-1/2 -\lambda_0}\, a^{2\lambda_0 - 1}
\end{equation}
as $a \downarrow 0$, where $f(a) \asymp g(a)$ means that $f(a)$ is bounded above and below by $cg(a)$ for different constants $c$. The upper bound is uniform in $x$ and the lower bound required a localizing assumption. For details, see Section 4 and in particular Theorem 4.8 of \cite{MMP2017}. Let $\text{dim}(B)$ denote the Hausdorff dimension of a set $B \subseteq \R$. \\

\textbf{Theorem A.} \emph{(Mueller, Mytnik, Perkins \cite{MMP2017}.) Under $P^X_{X_0}$, $\text{dim}(BZ_t) \leq 2 - 2\lambda_0$ almost surely on $\{X_t > 0 \}$ and $\text{dim}(BZ_t) \geq 2 - 2\lambda_0$ with positive probability.}\\

Because $\lambda_0 \in(1/2,1)$, the dimension satisfies $2-2\lambda_0 \in (0,1)$. The lower bound was conjectured to hold with full probability on $\{X_t>0\}$, implying that $\text{dim}(BZ_t) = 2-2\lambda_0$ almost surely on $\{X_t > 0\}$. The difficulty in proving that the lower bound for the dimension holds with probability one on $\{X_t>0\}$ is owing to the delicate nature of the $BZ_t$. It is not monotone in the initial conditions nor in the measure $X_t$ itself.\\

We will construct a random measure $L_t$, which we call the boundary local time of $X_t$, supported on $BZ_t$. (See Theorems~\ref{thm_Lt} and~\ref{thm_Ltprop}.) The existence of $L_t$ was conjectured in Section~5.1 of \cite{MMP2017}. Once we have constructed $L_t$, we use it to give a simpler alternative proof of the lower bound in Theorem~A. Our method is to show that $L_t$ has finite $p$-energy for all $p < 2-2\lambda_0$; in particular, see Theorem~\ref{thm_Ltdim} below. In a future work \cite{HP2018}, $L_t$ and several of its properties derived here, including Theorem~\ref{thm_Ltprop}(a), Proposition~\ref{prop_Ltcanonmomentsbd} and Theorem~\ref{thm_clustergeneral}, will be used to resolve the problem left open in Theorem~A and Theorem~\ref{thm_Ltdim}, showing that $\text{dim}(BZ_t) = 2-2\lambda_0$ almost surely on $\{X_t > 0 \}$.\\

We now give a description of $\lambda_0$. Define a function $F(x)$ by
\begin{equation} \label{e_Fdefintro}
F(x) := -\log P_{\delta_0}^X \big( \{X(1,x) = 0 \}\big) = \N_0\big(\{X(1,x) > 0 \}\big) > 0. 
\end{equation} 
The second equality is standard and is a consequence of \eqref{e_canonPPP} below. Section \ref{s_duality}, from (\ref{e_dualPDElambda}) to (\ref{e_FODE}), provides a thorough overview of $F$ as the limit as $\lambda \to \infty$ of the family of functions $\{V^\lambda_1\}_{\lambda>0}$ which characterize the Laplace transform of the density $X(t,x)$. Let $A f(x) = \frac 1 2 f''(x) - \frac x 2 f'(x)$ denote the infinitesimal generator of a standard, one-dimensional Ornstein-Uhlenbeck process $Y$. For a bounded, continuous function $\phi$ with limits at infinity ($F$ is such a function), $A^\phi f = Af - \phi f$ is the generator of an Ornstein-Uhlenbeck process with Markovian killing corresponding to $\phi$; that is, for a sample path $\{Y_s : s \in[0,\infty)\} \in  C([0,\infty); \R)$, we define the lifetime of the process as $\rho^\phi$, after which it is ``killed," or put into an inert cemetery state. The distribution of $\rho^\phi$ is given by 
\begin{equation} \label{e_survivalprob}
P(\rho > t\, | \, Y) = \exp \left( -\int_0^t \phi(Y_s) \, ds\right) \,\,\,\, \text{for } t > 0.
\end{equation}
Section \ref{s_OU} develops the relevant theory for these processes and their generators. In particular, Theorem \ref{thm_killedOU} states that $A^\phi$, taken as an operator on the appropriate Hilbert space, has countable orthonormal family of eigenfunctions $\{\psi_n^\phi\}_{n=0}^\infty$ with corresponding discrete spectrum $0 \geq -\lambda_0^\phi \geq -\lambda_1^\phi \geq \cdots \to -\infty$. We define $\lambda_0 = \lambda_0^F > 0$. As we have noted, it was shown in \cite{MMP2017} that $\lambda_0 \in (1/2,1)$. Numerical estimates by Zhu \cite{Z2017}, for which the stated digits are expected to be accurate, suggest that $\lambda_0 \approx 0.8882$. This implies that the value of $\text{dim}(BZ_t)$ from Theorem A, $2-2\lambda_0$, is approximately $0.224$. A more detailed discussion of the numerics can be found in the introduction of \cite{HP2018}. \\

The method the authors of \cite{MMP2017} used to show \eqref{e_lefttail} involved computing the asymptotic behaviour of the Laplace transform of the density. In particular (see Proposition 4.5 of that work),
\begin{align}\label{e_densityLaplace}
\lim_{\lambda \to \infty}t^{\lambda_0} \,& \lambda^{2\lambda_0} \,E^X_{X_0} \left(\int \phi(x)\, X(t,x) \,e^{-\lambda X(t,x)} dx\right)  \nonumber
\\ &= c_0 \iint \phi(w_0 + \sqrt t z) \exp\left(- \frac 1 t \int F(z+ t^{-1/2}(w_0 - x_0) \, dX_0(w_0) \right) \psi_0^F(z) \, dm(z)\,dX_0(x_0)
\end{align}
for every bounded Borel function $\phi$, where $m(dz)$ denotes the unit variance Gaussian measure in one dimension, $c_0$ is a positive constant and $\psi_0^F$ is the lead eigenfunction of $A^F$. For a super-Brownian motion with density $X(t,x)$, for $\lambda>0$ we define the measure $L^\lambda_t \in \cM_F(\R)$ by $dL^\lambda_t(x) = \lambda^{2\lambda_0} e^{-\lambda X(t,x)} X(t,x) \, dx$. That is, for a bounded measurable function $\phi: \R \to \R$, we define
\begin{equation} \label{e_Llambdadef}
L^\lambda_t(\phi) = \lambda^{2\lambda_0} \int \phi(x) \, X(t,x)\, e^{-\lambda X(t,x)} dx.
\end{equation}
$L^\lambda_t$ is defined the same way under $P^X_{X_0}$ and $\N_0$. The scaling factor of $\lambda^{2\lambda_0}$ can be deduced from \eqref{e_densityLaplace}. The convergence of $E^X_{X_0} (L^\lambda_t(\phi))$ as $\lambda \to \infty$, noted in \eqref{e_densityLaplace}, led the authors of \cite{MMP2017} to conjecture (Section 5.1 of that reference) that there is a random measure $L_t$ on $\R$ such that $L^\lambda_t \to L_t$ in $\cM_F(\R)$ in probability. Our main result is the verification of this conjecture. In all that follows, $X_0 \in \cM_F(\R)$.
\begin{theorem} \label{thm_Lt} 
Let $t>0$. Under both $P^X_{X_0}$ and $\N_0$ there is a random measure $L_t(dx) \in \cM_F(\R)$, supported on $BZ_t$, such that $L^\lambda_t \to L_t$ in measure as $\lambda \to \infty$, and there is a sequence $\lambda_n \to \infty$ such that $L_t^{\lambda_n} \to L_t$ a.s. as $n \to \infty$. Moreover, under $P^X_{X_0}$ or $\N_0$, for all bounded and continuous functions $\phi$, $L^\lambda_t(\phi) \to L_t(\phi)$ in $\cL^2$ as $\lambda \to \infty$.
\end{theorem}
\begin{theorem} \label{thm_Ltprop} 
(a) $P^X_{X_0}(L_t > 0 \,| \, X_t > 0 ) > 0$ and $\N_0 (L_t > 0 \,| \, X_t > 0 ) \geq \frac{1-\lambda_0}{2}$ for all $t>0$.\\
(b) $L_t$ is atomless almost surely under $P^X_{X_0}$ and $\N_0$. 
\end{theorem}
\textbf{Definition.} $L_t$ is the boundary local time of $X_t$. \\

We note that $Z_t$ will contain intervals, unlike the zero set of a Brownian motion (which is equal to its boundary). It is easy to see that $L_t$ is supported on $BZ_t$ from the fact that as $\lambda$ gets large, $L^\lambda_t$ concentrates on $\{x : 0 < X(t, x) = O(\lambda^{-1})\}$, and properties of the weak topology on $\cM_F(\R)$ (see the proof of Theorem~\ref{thm_Lt} in Section~\ref{s_main}). For fixed $t>0$, $x \to X(t,x)$ is a continuous path taking values in $\R^+ = [0,\infty)$. $BZ_t$ is the set of points where this path begins and ends its excursions from $0$. As $L_t$ is supported on $BZ_t$, in this sense $L_t$ is a local time of $x\to X(t,x)$ on these excursion endpoints, and hence the boundary local time of $X(t,\cdot)$.\\

The existence of a measure supported on $BZ_t$ allows us to use the energy method to study its dimension. We will provide a second moment formula for $L_t$, with which we compute the expectation of energy integrals of the form
\begin{equation} \label{e_energyint}
\iint |x-y|^{-p} \, dL_t(x) \,dL_t(y).
\end{equation}
If $L_t >0$ and the above energy is finite, then $\text{dim}(\text{supp}(L_t)) \geq p$ by Frostman's connection between energy integrals and Hausdorff dimension (see Theorem 4.27 of M\"orters and Peres \cite{MP10}). We introduce some notation. For $h:\R^2 \to \R$, define $(L_t \times L_t)(h)$ by
\[ (L_t \times L_t)(h) = \iint h(x,y) \, dL_t(x)\, dL_t(y).\]
For $p>0$, we define $h_p(x,y) = |x-y|^{-p}$. The second moment formula for $L_t$ allows us to establish the following.
\begin{theorem} \label{thm_Ltdim} 
$E^X_{X_0}((L_t \times L_t)(h_p))$ and $\N_0((L_t \times L_t)(h_p))$ are finite for all $p< 2-2\lambda_0$. Moreover, $\text{dim}(BZ_t) = 2-2\lambda_0$ almost surely on $\{L_t> 0 \}$ under both measures.
\end{theorem}
The fact that $\text{dim}(BZ_t) \leq 2-2\lambda_0$ $P^X_{X_0}$-a.s. is already known from Theorem A, and from this it follows easily under $\N_0$, as we point out in the proof of Theorem~\ref{thm_Ltdim}. By the above, the lower bound, ie. $\text{dim}(BZ_t) \geq 2-2\lambda_0$, holds with at least the probability that $L_t >0$, as in Theorem~\ref{thm_Ltprop}(a). This plays an important role in Hughes-Perkins \cite{HP2018}; in Theorem 1.2 of \cite{HP2018} we show that with respect to both $P^X_{X_0}$ and $\N_0$, $L_t > 0$ almost surely on $\{X_t > 0 \}$, thus improving part (a) of Theorem~\ref{thm_Ltprop} above and establishing almost sure non-degeneracy of $L_t$. Combined with Theorem~\ref{thm_Ltdim}, this will show that $\text{dim}(BZ_t) = 2-2\lambda_0$ almost surely on $\{X_t > 0 \}$. \\

There are a number of other potential uses for such a local time. We now discuss some possibilities. By sampling a point from $L_t$, we are able to ``view $X_t$ from the perspective of a typical point in $BZ_t$." More precisely, one can define $Q_{X_0}((Z,X_t) \in A) = E^X_{X_0}(\int 1_A(z,X_t) \, dL_t(z))$ and study properties of the Palm measure $Q_{X_0}(X_t \in \cdot \, | \, Z = z)$. The behaviour of $X_t$ near $BZ_t$ is complex and there is still much that is not understood about it. For example, the density has an improved modulus of continuity and is nearly Lipzschitz (ie. H\"older $1-\eta$ for all $\eta > 0$) at points in $BZ_t$ (see Theorem 2.3 of \cite{MP2011}). This suggests that $BZ_t$ would be small, but despite this $BZ_t$ has positive dimension. Constructing and studying the Palm measure described above would give a more structured approach for investigating this phenomenon. \\

As a local time, $L_t$ has the potential to study pathwise uniqueness in the SPDE \eqref{e_SPDEdensity}, a problem which remains open, assuming a similar role as that of the semi-martingale local time in the Yamada-Watanabe Theorem for one-dimensional SDEs (see Theorem V.40 of Rogers and Williams \cite{RW}). It may also provide insight in the behaviour of some discrete processes; super-Brownian motion in high dimensions is the scaling limit of a number of lattice models and interacting particle systems. In dimension one, it is still the scaling limit of branching random walk (for example see \cite{W1968} or Theorem II.5.1(iii) of \cite{P2002}). One could obtain information about the boundaries of such approximating processes by proving a limit theorem establishing weak convergence of the laws of their discrete local times to that of $L_t$. Of course, $L_t$ allows for us to study $BZ_t$ more directly, as we have done in Theorem~\ref{thm_Ltdim}. In fact, with $L_t$ it may be possible to determine the exact Hausdorff measure function of $BZ_t$.\\

We now discuss the method of our proof. Upper bounds on second moments of $L^\lambda_t$ were obtained in Section 5.1 of \cite{MMP2017}, but in order to establish the existence of $L_t$ we require exact asymptotics, which are more delicate. The main ingredient is the following convergence result. In order to state it we need to introduce some notation. Recall that $m(dx)$ denotes the centred unit variance Gaussian measure. Let $\psi_0 = \psi_0^F$ (the eigenfunction of $A^F$ corresponding to eigenvalue $-\lambda_0$). The constant $C_{\ref{thm_l2limit}}$ is given explicitly in (\ref{e_Cdef}), and the function $\rho$ is defined in (\ref{e_rhodef}). The function $V^{\infty,\infty}_t$ is defined in Section~\ref{s_duality} as $V^{\infty,\infty}_t(x_1,x_2) = \N_0 \left(\{ X(t,x_1) >  0 \} \cup \{X(t,x_2) > 0\} \right)$ (see \eqref{V2infprob}).

\begin{theorem} \label{thm_l2limit} 
There exists a constant $C_{\ref{thm_l2limit}} > 0$ and continuous function $\rho:\R \times \R \to (0,1]$ such that for bounded Borel $h : \R^2 \to \R$,
\begin{align}
&\lim_{\lambda, \lambda' \to \infty} \N_0 ((L^{\lambda}_t \times L^{\lambda'}_t)(h) ) \nonumber
\\ &\hspace{ 6mm}=  C_{\ref{thm_l2limit}}^2 \int_0^t (t-s)^{-2\lambda_0} \bigg[ \iint E^{B}_0 \bigg(  \exp \left( - \int_0^s V^{\infty, \infty}_{t-u} (\sqrt{t-s} \, z_1 + B_s - B_u, \sqrt{t-s} \, z_2 + B_s - B_u) \, du \right) \nonumber
\\ & \hspace{10 mm}\times  h(\sqrt{t-s} \, z_1 + B_s,\sqrt{t-s} \, z_2 + B_s) \bigg)\, \rho(z_1, z_2) \,\psi_0(z_1) \, \psi_0(z_2) \, dm(z_1) \, dm(z_2)  \bigg] ds. \nonumber
\end{align}
Moreover, the limit is finite for all bounded $h$.
\end{theorem}
That the formula above is finite is not obvious, as $\lambda_0 > 1/2$; we discuss this in more detail shortly. From the above we can deduce that $\{L^\lambda_t(\phi)\}_{\lambda>0}$ is Cauchy in $\cL^2(\N_0)$ and therefore has a limit by completeness; in particular see Corollary~\ref{cor_l2limit} and its proof. We then argue that the limit is in fact the integral with respect to a unique measure, which is $L_t$. The proof of Theorem~\ref{thm_l2limit} is long and technical; Section~\ref{s_momentsconvergence} is entirely devoted to it. We use the Laplace functional to obtain a Feynman-Kac type representation for $\N_0 ( L_t^\lambda(\phi)\, L_t^{\lambda'}(\phi) )$ and then establish its convergence. The reason we do so under $\N_0$ is because the Feynman-Kac formulas are simpler in this setting. We now present first and second moment formulas for $L_t$ under $\N_0$; as one would expect, the second moment formula in part (b) agrees with the limit of $\N_0((L_t^\lambda \times L_t^{\lambda'})(h))$ given in Theorem~\ref{thm_l2limit}. The terms $C_{\ref{thm_l2limit}}$ and $\rho$ are the same that appeared in that result.

\begin{theorem} \label{thm_Ltcanonmoments} (a) For a bounded or non-negative Borel function $\phi:\R \to \R$,
\begin{equation} \label{e_Ltcanon_firstmoment}
\N_0 ( L_t(\phi) ) = C_{\ref{thm_l2limit}}\, t^{-\lambda_0} \int \phi(\sqrt t z)\, \psi_0(z) \, dm(z).
\end{equation}
(b) For measurable $h: \R^2 \to \R$, either bounded or non-negative,
\begin{align} \label{e_Ltcanon_secondmoment}
\N_0 ( (L_t &\times  L_t)(h) )  \nonumber
\\ = \,&C_{\ref{thm_l2limit}}^2 \int_0^t (t-s)^{-2\lambda_0} \bigg[ \iint E^{B}_0 \bigg(  \exp \left( - \int_0^s V^{\infty, \infty}_{t-u} (\sqrt{t-s} \, z_1 + B_s - B_u, \sqrt{t-s} \, z_2 + B_s - B_u) \, du \right) \nonumber
\\ \hspace{5 mm} &\times  h(\sqrt{t-s} \, z_1 + B_s, \sqrt{t-s} \, z_2 + B_s) \bigg)\, \rho(z_1, z_2) \,\psi_0(z_1) \, \psi_0(z_2) \, dm(z_1) \, dm(z_2)  \bigg] ds. 
\end{align}
Moreover, \eqref{e_Ltcanon_secondmoment} is finite for all bounded $h$.
\end{theorem}
As we noted earlier, finiteness of \eqref{e_Ltcanon_secondmoment} is not obvious since $\lambda_0 > 1/2$ (although it is implicit in the proof of Theorem~\ref{thm_l2limit}), which can make \eqref{e_Ltcanon_secondmoment} hard to use; for applications, the following upper bound for second moments is easier to apply than the exact formula. The value $\theta$ is defined as $\theta = \int \psi_0 \, dm$. $Y$ is an Ornstein-Uhlenbeck process started at $z_1$ with corresponding expectation $E^Y_{z_1}$. The exponential term in the first bound of the following proposition can be interpreted as a survival probability of $Y$, producing a $w^{\lambda_0}$ term which makes the integral finite. (The proofs of Theorem~\ref{thm_Ltdim} and Theorem~\ref{thm_Ltprop}(b) in Section~\ref{s_main} both use this technique.)
\begin{proposition} \label{prop_Ltcanonmomentsbd}
For a non-negative Borel function $h: \R^2 \to \R$,
\begin{align} \label{e_Ltcanon_secondmomentbd}
\N_0 ( (L_t \times  L_t)(h) )  \leq \,&C_{\ref{thm_l2limit}}^2 \int_0^t w^{-2\lambda_0} \bigg[ \iint E^Y_{z_1} \bigg(  \exp \bigg( - \int_{0}^{\log (t/w)} F(Y_u) \, du \bigg)  \nonumber
\\ & \times h(\sqrt t Y_{\log (t/w)}, \sqrt t Y_{\log (t/w)}+ \sqrt{w}(z_2 - z_1)) \bigg) \psi_0(z_1) \, \psi_0(z_2) \, dm(z_1) \, dm(z_2)  \bigg] dw.
\end{align} 
Moreover,
\begin{equation} \label{e_Ltcanon_secondmomentmassbd}
\N_0 ( L_t(1)^2 ) \leq  \frac{C_{\ref{thm_l2limit}}^2 \theta^2}{1-\lambda_0} t^{1-2\lambda_0}.
\end{equation}
\end{proposition}
As we have alluded to, applying \eqref{e_Ltcanon_secondmomentbd} with $h(x,y) = |x-y|^{-p}$ gives an upper bound for the expectation of energy integrals of the form \eqref{e_energyint}, which is how we prove Theorem~\ref{thm_Ltdim}.\\

Thus far, we have not commented on the proofs of existence and properties of $L_t$ under $P^X_{X_0}$. The proofs rely on the conditional representation in terms of canonical clusters, which we will discuss shortly. First, in order to keep  the moment results together, we state our results regarding the moments of $L_t$ under $P^X_{X_0}$.
\begin{theorem} \label{thm_Ltmoments}
For a bounded or non-negative Borel function $\phi:\R \to \R$,
\begin{equation} \label{e_Ltfirstmoment} 
E^X_{X_0}( L_t(\phi) ) = C_{\ref{thm_l2limit}}\, t^{-\lambda_0} \iint \phi(x_0 + \sqrt t z) \exp \left( -\frac 1 t \int F(z + t^{-1/2} (x_0 - y_0) \, dX_0(y_0) \right) \psi_0 (z)\, dm(z) \,dX_0(x_0).
\end{equation}
(b) There is a constant $C_{\ref{thm_Ltmoments}}$ such that
\begin{equation} \label{e_Ltsecondmomentbd}
E^X_{X_0} (L_t(1)^2) \leq  C_{\ref{thm_Ltmoments}} \left( X_0(1)\, t^{1-2\lambda_0} + X_0(1)^2\, t^{-2\lambda_0} \right).
\end{equation}
\end{theorem}
We note that the right hand side of \eqref{e_Ltfirstmoment} is equal to that of \eqref{e_densityLaplace}, and so was originally computed in Proposition 1.5 of \cite{MMP2017} as $\lim_{\lambda \to \infty} E^X_{X_0}(L^\lambda_t(\phi))$. The fact that the same formula gives the mean measure of $L_t$ then follows from the $\cL^2$ convergence of $L_t^\lambda(\phi)$, as in Theorem~\ref{thm_Lt}.\\

We first establish the existence of $L_t$, as well as its properties, under the measure $\N_0$, owing to the fact that the second moments of $L^\lambda_t$ admit simpler formulas in this case. In order to prove the same for super-Brownian motion, we need to use the relationship between super-Brownian motion under $P^X_{X_0}$ and the canonical measure, which we now describe. We recall that $\N_x$ is a $\sigma$-finite measure such that $\N_x (\{X_t>0\}) = 2/t$ which describes the ``law" of a single cluster of super-Brownian motion started at $x$; that is, the descendants of a single ancestor at $x$. More precisely, super-Brownian motion is a superposition of canonical clusters; for a bounded, non-negative Borel function $\phi:\R \to \R$, 
\begin{equation} \label{e_canonPPP}
E^X_{X_0} \left( \exp \left(-X_t(\phi)\right)\right) = \exp \left( -\iint 1-e^{-\mu_t(\phi)}d\N_{x_0}(\mu)\, dX_0(x_0) \right).
\end{equation}
This expression for the Laplace functional is in fact a consequence of a distributional equality between super-Brownian motion under $P^X_{X_0}$ and a Poisson point process of canonical clusters. For $X_0 \in \cM_F(\R)$, let $\N_{X_0}(\cdot) = \int \N_x(\cdot) \, dX_0(x)$ and let $\Theta_{X_0}$ be a Poisson point process on $C([0,\infty),\cM_F(\R))$ with intensity $\N_{X_0}$. We define a $\cM_F(\R)$-valued process $(X_t : t\geq 0)$ by
\begin{equation} \label{e_XtPPP}
X_t(\cdot) = 
\begin{cases} 
\int \mu_t(\cdot) \, d\Theta_{X_0}(\mu)& \text{  if } t>0,
\\ X_0(\cdot) &\text{  if } t = 0.
\end{cases}
\end{equation}
By Theorem 4 of Section IV.3 of \cite{LG1999}, $(X_t : t\geq 0)$ is a super-Brownian motion with initial measure $X_0$. The ``points" of the point process $\Theta_{X_0}$ are the clusters of $X$. For fixed $t>0$, \eqref{e_XtPPP} leads to 
\begin{equation}
X_t = \sum_{j \in I_t} \mu^j_t, \nonumber
\end{equation}
where $\{\mu^j_t : j \in I_t\}$ are the points of a Poisson point process with finite intensity $\N_{X_0}(\mu_t \in \cdot \, | \, \mu_t > 0 \}$. Let $\bar{X_0}(\cdot) = X_0(\cdot) / X_0(1)$. Assuming our probability space is rich enough to allow us to choose random relabellings of these points, by the above we can write
\begin{equation} \label{e_clusterrep}
X_t = \sum_{i=1}^N X_t^i,
\end{equation}
where $N$ is Poisson$(2X_0(1)/t)$ and, given $N$, $\{X^i_t : i = 1,\ldots,N \}$ are iid with distribution $\N_{\bar{X}_0}(X_t \in \cdot \,| \, X_t > 0 )$. We can and do condition on the values of the initial points of the clusters, denoted by $x_1, \ldots, x_N$, which are iid points with distribution $\bar{X}_0$, in which case $X^i_t$ has conditional distribution $\N_{x_i}(X_t \in \cdot \, | \, X_t > 0)$. In order to prove the existence and properties of $L_t$ with respect to a super-Brownian motion $X_t$, we realize the super-Brownian motion as a point process and express $X_t$ as above. Conditioning on $N$ and applying \eqref{e_clusterrep}, we can write $L^\lambda_t(\phi)$ as
\begin{equation}
L^\lambda_t(\phi) = \lambda^{2\lambda_0} \int \left[ \sum_{i=1}^{N} X^i(t,x) \right] e^{-\lambda \sum_{i=1}^{N} X^i(t,x)}\, \phi(x) \, dx. \nonumber 
\end{equation}
The almost sure existence of boundary local times corresponding to the canonical clusters allows us to take this limit quite easily and so establish that $L_t$ exists under $P^X_{X_0}$ (ie. Theorem~\ref{thm_Lt}). Furthermore, we obtain a conditional representation for $L_t$ in terms of its clusters; this allows us to transfer the properties of $L_t$ under $\N_0$ to $L_t$ under $P^X_{X_0}$. Let $L^{i}_t$ denote the boundary local time of $X_t^i$. In the statement that follows, we assume that we have realized $X_t$ using \eqref{e_clusterrep}.
\begin{theorem} \label{thm_clusterLt} Let $X_t$ be super-Brownian motion under $P^X_{X_0}$ and $L_t$ be its boundary local time. Conditional on $N$, we have
\begin{align} \label{e_clusterdecomp}
dL_t(x) &= \sum_{i=1}^{N} 1\big(\sum_{j \neq i} X^j(t,x) = 0\big) \, dL^i_t(x) \nonumber
\\ &=  1(X(t,x) = 0) \sum_{j=1}^{N} dL^i_t(x). 
\end{align}
\end{theorem}
\textbf{Remark.} Given the nature of $BZ_t$, we expect this behaviour. In the cluster decomposition, each cluster has a boundary local time of its own. Since each is supported on the boundary of its respective zero set, the local time $L_t$ of $X_t$ will be equal to the sum of cluster local times, except the boundary of the zero set of one cluster may be ``swallowed" by the support of another, hence the indicator functions. \\

The idea of representing the boundary local time of $X_t$ in terms of the boundary local time of its clusters is not restricted to a super-Brownian motion and its comprising canonical clusters. The following formulation of the same principle will be useful in Hughes-Perkins \cite{HP2018}. Recall that a sum of independent super-Brownian motions is a super-Brownian motion.

\begin{theorem} \label{thm_clustergeneral}
Suppose $X^1, \hdots, X^n$ are independent super-Brownian motions with corresponding boundary local times $L^i_t$ at time $t>0$, for $i=1,\hdots,n$. Let $X = \sum_{i=1}^n X^i$ and let $L_t$ be the boundary local time of $X_t$. Then 
\begin{align}
dL_t(x) &=  \sum_{i=1}^{n} 1\big(\sum_{j\neq i} X^j(t,x)= 0\big) \,  dL^i_t(x) \nonumber
\\  &= 1(X(t,x) = 0\big) \sum_{i=1}^{n} dL^i_t(x). \nonumber 
\end{align} 
\end{theorem}

One example of superprocesses satisfying the above conditions follows from (III.1.3) of \cite{P2002}. Let $X_0 \in \cM_F(\R)$ and suppose that $\{A_1, \hdots, A_n\}$ is a Borel partition of $\R$. Define $X^i$ as the contribution to $X$ from ancestors at time $0$ which are in $A_i$. (This makes $X^i$ a super-Brownian motion with initial measure $X_0( \cdot \cap A_i)$; a precise definition of $X^i$ may be given in terms of the historical process as in the above reference.) Then $X = \sum_{i=1}^N X^i$ satisfies the conditions of the above theorem.\\

\textbf{Notations.} We will make use of the common convention that $C$ denotes any positive constant whose value is not important. The value of $C$ may change line to line in a derivation; to bring attention to the fact that the constant has changed, we will sometimes label the new constant $C'$. We write $f \sim g$ if $\lim_x f(x)/g(x) = 1$, where the limit will be clear from context. As the reader has probably inferred, we will write $\mu > 0$ when a measure has positive mass (that is, to indicate that $\mu(1) > 0$). For an interval $I \subseteq \R$, let $C(I,\R)$ denote the space of continuous maps from $I$ to $\R$.\\

Let $S_t$ denote the semi-group of Brownian motion and $p_t$ the associated heat kernel (the Gaussian density of variance $t$). Let $\cN(x_0,\sigma^2)$ denote the law of a one-dimensional Gaussian with mean $x_0$ and variance $\sigma^2$. \\

\textbf{Organization of Paper.} The paper is organized as follows. \textbf{ Section \ref{s_OU}} gives a brief overview of the theory of one-dimensional Ornstein-Uhlenbeck processes with Markovian killing. Our method relies on a change of variables which allow us to express certain quantities in terms of eigenvalue problems involving these processes' generators. \textbf{Section \ref{s_duality}} describes fundamental background connecting the Laplace functional of super-Brownian motion to a family of semi-linear PDEs. We also introduce the families $V^\lambda$ and $V^{\lambda, \lambda'}$, which play a key role in our analysis. \textbf{Section~\ref{s_main}} contains the proofs of all our main results, including existence and properties of $L_t$ and the cluster representations, with the exception of Theorem~\ref{thm_l2limit}. The proof of this result is reserved for \textbf{Section \ref{s_momentsconvergence}}.\\

\textbf{Acknowledgements.} The author gratefully acknowledges the assistance of Ed Perkins, his thesis supervisor, who introduced him to the problem and provided many useful insights and suggestions during its resolution, and gave close readings of the manuscript at several stages during its preparation. Any remaining inconsistencies are the sole responsibility of the author.

\section{Killed Ornstein-Uhlenbeck Processes} \label{s_OU}
As above, we define the operator $A$ by $Af(x) = \frac{f''(x)}{2} - \frac{xf'(x)}{2}$. The Markov process generated by $A$ is a one-dimensional Ornstein-Uhlenbeck process with mean zero. We denote this process by $Y$, denote its law when started at $x$ by $P^Y_x$ with corresponding expectation $E^Y_x$. For general initial conditions $Y_0 \sim \mu \in \cM_1(\R)$ (the space of probability measures on $\R$), we write its law as $P^Y_\mu$. $Y$ has a stationary measure, the unit variance Gaussian measure, $m$. When $Y_0 \sim m$, the process is reversible and can be defined for time values in $\R$. We will denote the law of this stationary process on $\R$ by $P^Y$.\\

We now introduce the notions of killing and lifetime for the process $(Y_t: t\geq 0)$. Let $\phi \in C^+([-\infty,\infty],\R)$, the space of non-negative continuous functions with limits at $\pm \infty$. We will call such functions killing functions. Let $A^\phi f(x) = Af(x) - f(x)\phi(x)$. $A^\phi$ is the generator of an Ornstein-Uhlenbeck process subjected to Markovian killing at rate $\phi(Y_t)$. The lifetime of the killed process is $\rho^\phi = \inf \{t> 0: \int_0^t \phi(Y_s) \, ds > e \}$, where $e$ is an independent Exp($1$) random variable. We recall that the distribution of $\rho^\phi$ is given by \eqref{e_survivalprob}.\\

The generators $A$ and $A^\phi$ correspond to strongly continuous contraction semigroups on $\cL^2(m)$. The following theorem is proved in \cite{MMP2017}, where it is stated as Theorem 2.3. We note that the statement of the result in that paper had a misprint when describing the convergence of of the transition densities, which appeared in part (c). We have corrected the statement, which is in part (b) of the following. 
\begin{theorem} \label{thm_killedOU} For $\phi \in C^+([-\infty, \infty])$, the following statements hold.\\
(a) $A^\phi$ has complete orthonormal family of $C^2$ eigenfunctions $\{\psi_n : n \geq 1 \}$ in $\cL^2(m)$ satisfying $A^\phi \psi_n = -\lambda_n \psi_n$, where $0\leq \lambda_0 \leq \lambda_1 \leq \cdots \to \infty$. Furthermore, $-\lambda_0$ is a simple eigenvalue and $\psi_0 > 0$. \newpage
(b) For $t>0$, the diffusion $Y$ generated by $A^\phi$ has a jointly continuous transition density $q_t(x,y)$ with respect to $m$, given by
\begin{equation} \label{OU_eigexp}
q_t(x,y) = \sum_{n=0}^\infty e^{-\lambda_n t} \psi_n(x) \psi_n(y),
\end{equation} 
where the series converges in $\cL^2(m \times m)$ and uniformly absolutely on sets of the form $[\epsilon,\infty) \times [-\epsilon^{-1}, \epsilon^{-1}]^2$ for all $\epsilon > 0$.\\

(c) For $0< \delta<\frac 1 2$, there exists a constant $c_\delta>0$ such that 
\begin{equation} \label{OU_qbound}
q_t(x,y) \leq c_\delta e^{-\lambda_0 t} e^{\delta(x^2 + y^2)} \hspace{ 4mm} \text{ for all } t \geq s^*(\delta),
\end{equation}
where $s^*(\delta) > 0$ is the solution of
\begin{equation} \label{OU_s_star}
2\delta = \frac{e^{-s^*/2} - e^{-s^*}}{1 - e^{-s^*}}.
\end{equation}
(d) Denote $\theta = \int \psi_0 \, dm$. For all $t\geq 0$ and $x \in \R$,
\begin{equation} \label{OU_survivalprob}
e^{\lambda_0 t} P_x (\rho^\phi > t) = \theta \psi_0(x) + r(t,x),
\end{equation}
where, for any $\delta > 0$, there is a constant $c_\delta >0$ such that 
\begin{align}
\psi_0(x) &\leq c_\delta e^{\delta x^2}, \label{OU_psi0bd}
\\ |r(t,x)| &\leq c_\delta e^{\delta x^2} e^{-(\lambda_1 - \lambda_0)t}. \label{OU_rbound}
\end{align}
(e) As $T \to \infty$, $P_x(Y \in \cdot | \, \rho_\phi > T) \to P^{Y,\infty}_x$ weakly on $C([0,\infty),\R)$, where $P^{Y,\infty}_x$ is the law of the diffusion with the transition density
\begin{equation} \label{OU_immortaldensity}
\tilde{q}_t(x,y) \equiv q_t(x,y) \frac{\psi_0(y)}{\psi_0(x)}e^{\lambda_0 t}
\end{equation}
with respect to $m$.
\end{theorem}
The bounds in part (c) of the above easily imply the following estimates, which we will often use. For $0< \delta < 1/2$, there is a constant $C_\delta > 0$ such that
\begin{equation} \label{OU_survprobbd}
P^Y_x (\rho_\phi > t) \leq C_\delta e^{\delta x^2} e^{-\lambda_0 t} \,\,\,\, \forall \, x \in \R, t > 0.
\end{equation}
This implies that there is a constant $C>0$ such that
\begin{equation} \label{OU_survivalprobGaus}
P^Y_m(\rho_\phi > t) \leq C e^{-\lambda_0 t} \,\,\, \forall \, t>0.
\end{equation}
The following limit result is a simple consequence of the eigenfunction expansion for $q_t(x,y)$.
\begin{lemma} \label{lemma_densitylimit}
For all $x,y \in \R$, 
\begin{equation}
\lim_{t \to \infty} e^{\lambda_0 t} q_t(x,y) = \psi_0 (x)  \psi_0 (y). \nonumber
\end{equation}
The convergence is uniform on compact sets.
\end{lemma}
\begin{proof}
For all $t>0$ and $x,y \in \R$, from \eqref{OU_eigexp}, we have
\begin{equation} \label{e_transdenslim}
e^{\lambda_0 t} q_{t}(x,y) = \psi_0(x) \psi_0(y) + \sum_{n=1}^\infty e^{-(\lambda_n - \lambda_0)t} \psi_n(x) \psi_n(y).
\end{equation}
The absolute value of the sum above is bounded above by
\[e^{-(\lambda_1 - \lambda_0) (t-1)} \sum_{n=1}^\infty e^{-(\lambda_n - \lambda_0)} |\psi_n(x)  \psi_n(y)|.\]
By Theorem~\ref{thm_killedOU}(b) with $t=1$, the series in the above is convergent, and the convergence is uniform on compact sets. Part (a) of the same theorem states that $-\lambda_0$ is a simple eigenvalue. Hence $\lambda_1 - \lambda_0 > 0$ and the above vanishes as $t \to \infty$; in fact, because the series converges uniformly on compacts to a continuous limit, the above vanishes uniformly on compacts as $t\to \infty$, so \eqref{e_transdenslim} gives the result.
\end{proof}

It will be useful for us to study the distribution of the process $Y$ when conditioned on survival and its endpoint. Hereafter we assume that $Y$ has killing function $\phi \in C^+([-\infty,\infty],\R)$ and we denote its lifetime by $\rho$. For fixed $T>0$ and $z \in \R$, consider the $[0,T]$-indexed inhomogeneous Markov process taking values in $\R$ with transition density (with respect to $dm(y_2)$)
\begin{equation} \label{e_bridgetransden1}
\hat{q}_{s,t}(y_1, y_2) = \frac{q_{t-s}(y_1, y_2) \,q_{T-t}(y_2,z)}{q_{T-s}(y_1,z)}
\end{equation} 
for $0 \leq s < t < T$. (The kernels are degenerate when $t=T$, since $Y_T = z$.) Below we verify that the finite dimensional distributions defined by this transition kernel have an extension to a (necessarily) unique law on $C([0,T],\R)$, which we denote by $P^Y_x( \cdot \, | \, \rho> T, Y_T = z)$ when the initial point is $x \in \R$, and show that it gives an explicit version of the suggested regular conditional distribution for all $z \in \R$. We then establish that for fixed $S>0$, $P^Y_x(Y\restrict{[0,S]} \in \cdot \, | \, \rho> T, Y_T = z)$ converges weakly to $P^{Y,\infty}_x(Y\restrict{[0,S]} \in \cdot )$ as $T \to \infty$ for all $z \in \R$.
\begin{lemma} \label{lemma_endpoint_indep}
(a) Let $x \in \R$ and $T>0$. For all $z \in \R$, the finite dimensional distributions described in \eqref{e_bridgetransden1}, with initial value $x$, have a unique extension to $C([0,T],\R)$. The resulting laws $P^Y_x (\cdot \, | \,\rho>T, Y_T = z)$ are continuous in $z$ and define a regular conditional probability for $Y\restrict{[0,T]}$ under $P^Y_x$ conditioned on $Y_T$. \\
(b) Let $x,z \in \R$, $S>0$ be fixed. Then $P_x(Y\restrict{[0,S]} \in \cdot \, | \, \rho >T, Y_T = z)$ converges weakly on $C([0,S], \R)$ to $P_x^{Y,\infty}(Y\restrict{[0,S]} \in \cdot)$ as $T \to \infty$. \vspace{1 mm}\\ 
(c) For all $S, K>0$, $\big\{ P_x(Y\restrict{[0,S]} \in \cdot \, | \, \rho >T, Y_T = z) : |x|,|z| \leq K, T\geq S  \big\}$ is tight on $C([0,S],\R)$.
\end{lemma}

Before proving the lemma, we make an observation concerning time reversals of $Y$ under $P^Y_x (\cdot \, | \,\rho>T, Y_T = z)$. For $T>0$ and $t\in [0,T]$, define $\hat{Y}_t = Y_{T-t}$. Let $x ,z \in \R$. For $0 < t_1 < t_2 < T$ and $\phi_1, \phi_2$ bounded Borel functions, we have
\begin{align}
&E_x^Y ( \phi_1(\hat{Y}_{t_1}) \, \phi_2(\hat{Y}_{t_2}) \, \big| \, \rho > T, Y_T = z ) \nonumber
\\ &\hspace{ 6 mm }= \frac{1}{q_T(x,z)} \iint \phi_1(y_1) \, \phi_2(y_2) \, q_{T-t_2}(x,y_2) \, q_{t_2 - t_1}(y_2, y_1) \, q_{t_1}(y_1,z) \,dm(y_1) \, dm(y_2) \nonumber
\\ &\hspace{ 6 mm }= E_z^Y (\phi_1(Y_{t_1})\, \phi_2(Y_{t_2}) \, \big| \,\rho > T, Y_T = x ), \nonumber
\end{align}
where the last equality uses $q_t(x,y) = q_t(y,x)$. The above equality of distributions can be extended to general finite dimensional distributions. Because the extension of the finite dimensional distributions to a law on $C([0,T],\R)$ (ie. from Lemma~\ref{lemma_endpoint_indep}(a)) is unique, we therefore have that for all $x,z \in \R$, 
\begin{align} \label{reversible}
P^Y_x( \hat{Y}\restrict{[0,T]} \in \cdot \, | \, \rho > T, Y_T = z ) = P^Y_z( Y\restrict{[0,T]} \in \cdot \, | \, \rho > T, Y_T = x).
\end{align}
As a last note, we will sometimes denote the law $P^Y_x(\cdot \, | \, \rho>T, Y_T = z )$ simply by $P^Y_x(\cdot \, | \, Y_T = z )$ when it is clear from context that we are working with the killed process. \\

\textit{Proof of Lemma~\ref{lemma_endpoint_indep}.}
Let $x,z \in \R$ and $T>0$. We define a distribution $P^{Y}_x( \cdot \,| \,\rho > T, Y_T = z)$ on finite (time indexed) collections of random variables which describes the finite dimensional distributions (FDDs) of the inhomogeneous Markov process with transition density \eqref{e_bridgetransden1}. For $0= t_0 < t_1  < \hdots < t_n < T$ and bounded, continuous functions $\phi_1,\hdots,\phi_n$, the $n$-dimensional FDD of $(Y_{t_1}, \ldots, Y_{t_n})$ under $P^{Y}_x( \cdot \,| \,\rho> T, Y_T = z)$ is defined as
\begin{equation} \label{e_bridgefdd}
E^{Y}_x \bigg( \prod_{i=1}^n \phi_i(Y_{t_i})  \, \bigg| \,\rho>T, Y_T = z \bigg) = \frac{1}{q_T(x,z)} \int \bigg[ \prod_{i=1}^n \phi_i(y_i) q_{t_{n} - t_{n-1}}(y_{n-1}, y_n) \bigg] q_{T-t_n}(y_n,z) \prod_{i=1}^n dm(y_i)
\end{equation}
where we use the convention $y_0 = x$. We note that \eqref{e_bridgefdd} also defines the FDDs of a regular conditional distribution of $(Y_t : t \in [0,T])$ under $P_x^Y$ conditioned on $Y_T = z$ (which is why we have used this notation). Thus when we have established that these laws extend to a probability on $C([0,T],\R)$, we will have explicitly constructed a version of the regular conditional distribution.\\

To prove that $P^{Y}_x( \cdot \,| \,\rho > T, Y_T = z)$ extends to a probability on $C([0,T],\R)$, we will establish a tightness criterion. We consider the fourth moments of increments of $Y$. Let $0< s<t < T$. Expanding using \eqref{e_bridgefdd}, we have
\begin{align} \label{tightness1}
& E^{Y}_x ((Y_t-Y_s)^4 \,  | \,\rho > T, Y_T = z) \nonumber
\\ & \hspace{5 mm}= \frac{1}{q_T(x,z)} \iint (y_2-y_1)^4 q_{s}(x,y_1) \,q_{t-s}(y_1, y_2) \, q_{T-t}(y_2,z)  \,dm(y_1) \,dm(y_2) .
\end{align}
We now collect some elementary bounds and inequalities which will allow us to obtain a useful upper bound for the above. First, we note that while $q_t(x,y)$ is a transition density with respect to $m$, it will sometimes be useful to express it as a density with respect to the Lebesgue measure. Since $p_1(\cdot)$ is the density of $m$, we have
\begin{equation}\label{transLeb}
q_t(x,y)\, dm(y) = q_t(x,y) \,p_1(y) \, dy.
\end{equation}
We will use a comparison with an un-killed Ornstein-Uhlenbeck process. The transition kernel of a standard Ornstein-Uhlenbeck process is described by, for $0 \leq s < t$,
\[(Y_t - Y_s \, | \, Y_s = y) \sim \cN(e^{-(t-s)/2}y, 1- e^{-(t-s)}).\]
Let $k_t(x,y)$ denote the a transition density of an un-killed Ornstein-Uhlenbeck process with respect to Lebesgue measure. Then for $x,y\in \R$ and $t>0$, 
\begin{equation}\label{unkilledtrans}
k_t(x,y)  = \frac{(2\pi)^{-1/2}}{\sqrt{1-e^{-t}}} \exp \left( -(y-e^{-t/2}x)^2 \big /\,2(1-e^{-t}) \right).
\end{equation}
The transition densities of the killed Ornstein-Uhlenbeck process are bounded above by those of the un-killed process. This implies that
\begin{equation} \label{unkilledtransbd}
\text{(i)}\,\, q_t(x,y) \, dm(y) \leq k_t(x,y)\, dy, \hspace{8 mm} \text{(ii)}\,\, q_{t}(x,y) \,p_1(y) \leq k_t(x,y)
\end{equation}
It is easy to establish from \eqref{unkilledtrans} that there is a constant $c>0$ such that 
\begin{equation} \label{unkilledgaussbd}
k_t(x,y) \leq c p_t(y-xe^{-t/2})\,\, \text{ for all }  t\leq 2 \text{ and } x,y \in \R.
\end{equation} 
where we recall that $p_t(\cdot)$ is the Gaussian density of variance $t$. Let $K>0$. From  \eqref{unkilledtransbd}(ii) and \eqref{unkilledgaussbd} it follows that there is a constant $C_1(K)$ such that
\begin{equation} \label{tightaux1}
q_{T-t}(y_2,z) \leq k_{T-t}(y_2,z) p_1(z)^{-1} \leq \frac{C_1(K)}{\sqrt{T-T'}} \,\,\, \forall \, y_2 \in \R, z \in [-K,K],\text{ and }  t\leq T' < T.
\end{equation}
Next, we note that it holds by elementary formulas for moments of Gaussians that there is a constant $c>0$ such that
\begin{equation} \label{tightaux3}
\int (y_2-y_1)^4 p_t(y_2) \, dy_2 \leq c\,(t^2 + |y_1|^4) \,\,\,\, \forall \, y_1 \in \R, \, t>0.
\end{equation}
Finally, observe that $q_T(\cdot,\cdot)$ is bounded below by the transition density of $Y$ with constant killing function $\|\phi\|_\infty$
. Thus for all $K >0$ and $M\geq 1$, from \eqref{transLeb} we have
\begin{equation} \label{tightaux2}
q_T(x,z) \geq e^{-\|\phi\|_\infty T} k_T(x,z) p_1(z)^{-1} \geq \delta(K,M) \,\,\, \forall \, x,z \in[-K,K], \,\, T \in [M^{-1},M]
\end{equation}
for a sufficiently small constant $\delta(K,M) > 0$. \\

Let $0<T'<T$ and suppose that $0<s<t\leq T'$ such that $t-s\leq 1$. Let $K>0$ and suppose that $x,z \in [-K,K]$. Using \eqref{unkilledtransbd}(i) to bound $q_s(x,y_1) dm(y_1)$ and $q_{t-s}(y_1,y_2) dm(y_2)$, and \eqref{tightaux1} to bound $q_{T-t}(y_2,z)$, from \eqref{tightness1} we obtain that
\begin{align} \label{tightness2}
& E^{Y}_x ((Y_t-Y_s)^4 \,  | \,\rho > T, Y_T = z) \nonumber
\\ & \hspace{5 mm}\leq \frac{C_1(K)\,(T-T')^{-1/2}}{q_T(x,z)}  \int k_s(x,y_1) \bigg[ \int (y_2- y_1)^4 k_{t-s}(y_1,y_2) dy_2 \bigg] dy_1 \nonumber
\\ &\hspace{5 mm}\leq \frac{C_1(K)\,(T-T')^{-1/2}}{q_T(x,z)}  \int k_s(x,y_1) \bigg[ \int c(y_2- y_1)^4 p_{t-s}(y_2-e^{-(t-s)/2}y_1) dy_2 \bigg] dy_1,
\end{align}
where the second inequality uses \eqref{unkilledgaussbd}. Changing variables and applying \eqref{tightaux3}, we obtain that
\begin{equation} \label{tightaux666}
 \int c(y_2- y_1)^4 p_{t-s}(y_2-e^{-(t-s)/2}y_1) dy_2 \leq C (|y_1|^4(1-e^{-(t-s)/2})^4 + (t-s)^2) \leq C(1+|y_1|^4)(t-s)^2, 
\end{equation}
where in the second inequality we have $1-e^{-x} \leq x$ for $x \geq 0$ and $(t-s)^4 \leq (t-s)^2$ (since $t-s \leq 1$). Substituting this into \eqref{tightness2}, we obtain that
\begin{equation} 
E^{Y}_x ((Y_t-Y_s)^4 \,  | \, Y_T = z) \leq\frac{C_1(K)\,(T-T')^{-1/2} }{q_T(x,z)} \,(t-s)^2 \int C \, k_s(x,y_1)\, (|y_1|^4+1) \, dy_1.
\end{equation}
Recall that we have assumed $x,z \in [-K,K]$. By \eqref{unkilledtrans} it is clear that for $K>0$, the integral is bounded above by some constant $C_2(K)>0$ for all $x \in [-K,K]$ and $s>0$. Using this along with \eqref{tightaux2}, with a choice of $M\geq 1$ for which $T \in [M^{-1},M]$, from the above we deduce the following:
\begin{align} \label{tightness3}
\text{For all } \,x,z \in &[-K,K], \,  0 < s < t \leq T' \text{ such that } t-s \leq 1,  \nonumber
\\ &E^{Y}_x ((Y_t-Y_s)^4 \,  | \, Y_T = z) = C_2(K)\,\delta(K,M)^{-1}\,C_1(K)\,(T-T')^{-1/2} \times (t-s)^2 .
\end{align}
Let $T' = 2T/3$. Hereafter we consider increments of size at most $1 \,\wedge \,T/3$. We have that \eqref{tightness3} holds for all $0 < s <t \leq 2T/3$ such that $t-s \leq 1 \,\wedge \, T/3$ with constant $C_2(K) \delta(K,M)^{-1} C_1(K) (T/3)^{-1/2}$. It remains to show that it also holds on $[2T/3,T]$. To do so, we make use of reversibility. Suppose $T/3 \leq s < t < T$. Then
\begin{align} \label{increverse}
&E_x^Y ( (Y_t - Y_s)^4 \, \big| \, Y_T = z ) \nonumber
\\ &\hspace{ 6 mm }= \frac{1}{q_T(x,z)} \iint (y_2- y_1)^4\, q_{s}(x,y_1) \, q_{t-s}(y_1, y_2) \, q_{T-t}(y_2,z) \,dm(y_1) \, dm(y_2) \nonumber
\\ &\hspace{ 6 mm }= E_z^Y ((Y_{T-s}- Y_{T-t})^4 \, \big| \,Y_T = x ), 
\end{align}
where the last equality uses $q_t(x,y) = q_t(y,x)$ (a consequence of \eqref{OU_eigexp}) and \eqref{e_bridgefdd}. Since $0 < T-t < T-s \leq 2T/3$, by \eqref{tightness3} and \eqref{increverse} we have that for all $x,z \in [-K,K]$ and $T/3 \leq s < t < T$ such that $t-s \leq 1 \wedge T/3$,
\begin{equation}
E^{Y}_x ((Y_t-Y_s)^4 \,  | \,\rho > T, Y_T = z) = C_2(K)\,\delta(K,T)^{-1}\, C_1(K) \,(T/3)^{-1/2} \times (t-s)^2. \nonumber
\end{equation}
Combined with the previous statement that this holds for all $0<s<t \leq 2T/3$, we have that
\begin{align} \label{tightnessF}
\text{For all } \,x,z \in &[-K,K], \,  0 < s < t <T  \text{ such that } t-s \leq 1 \wedge T/3,  \nonumber
\\ &E^{Y}_x ((Y_t-Y_s)^4 \,  | \, Y_T = z) \leq C_2(K)\,\delta(K,M)^{-1}\, C_1(K) \,(T/3)^{-1/2} \times (t-s)^2 .
\end{align}
The above proof can be easily modified to obtain the same bound (with a potential change to the constant) for increments in which $s=0$ or $t = T$, and we omit it. Thus by \eqref{tightnessF} and the Kolmogorov Continuity Theorem, $P^Y_x ( \cdot \, | \,\rho>T, Y_T = z)$ has a unique extension to a probability on $C([0,T],\R)$, also denoted by $P^Y_x (\cdot \, | \,\rho>T, Y_T = z)$. As we noted earlier, this gives an explicit construction of the regular conditional distribution $(Y_t : t \in[0,T])$ under $P^Y_x$ given $\rho > T$ and $Y_T = z$,  Additionally, suppose that $z_n \to z$ and that $z_n \in [-K,K]$ for all $n\geq 1$. From \eqref{tightnessF}, $\{P^Y_x ( \cdot | \,\rho>T, Y_T = z_n) : n \geq 1\}$ is tight. It is clear from \eqref{e_bridgefdd} and continuity of $q_t(\cdot,\cdot)$ that the FDDs of $P^Y_x ( \cdot \, | \,\rho>T, Y_T = z_n)$ converge to those of $P^Y_x (\cdot \, | \,\rho>T, Y_T = z)$. Thus the aforementioned tightness proves that $P^Y_x ( \cdot \, | \,\rho>T, Y_T = z_n)$ converges to $P^Y_x ( \cdot \, | \,\rho>T, Y_T = z_n)$ as a law on $C([0,T],\R)$. Thus we have proved part (a).\\

Before proving (b), we note the following consequence of \eqref{tightnessF} and its proof. Let $S,K>0$ and fix $M > 1$ such that $S \in [M^{-1},M]$. By considering increments of $(Y_s : s \in [0,S])$ but allowing the time $T$ at which we condition $Y_T = z$ to take values in $[S,M]$, we have that
\begin{equation} \label{bonustightness}
\{P^Y_x ( Y\restrict{[0,S]} \in \cdot \, | \,\rho>T, Y_T = z) : |x|,|z| \leq K, T \in [S, M]\} \, \text{ is tight.}
\end{equation}

Next we turn to part (b). Fix $S>0$ and $x,z\in \R$. We now check that the FDDs of $(Y_s : s \in [0,S])$ under $P^{Y}_x (\cdot \,  | \,\rho>T, Y_T = z)$ converge to those of $(Y_s : s \in [0,S])$ under $P^{Y,\infty}$ as $T \to \infty$. Let $0 < t_1 < t_2 \leq S$ and let $\phi_1$ and $\phi_2$ be bounded and continuous functions. Then from \eqref{e_bridgefdd}, we have
\begin{align} \label{e_endpointindep11}
&E_x^Y (\phi_1(Y_{t_1}) \, \phi_2(Y_{t_2}) \, \big| \,\rho > T, Y_T = z )  \nonumber
\\&\hspace{8mm}= \frac{1}{q_T(x,z)} \iint \phi_1(Y_{t_1}) \, \phi_2(Y_{t_2}) \, q_{t_1}(x,y_1) \,q_{t_2-t_1}(y_1, y_2) \,q_{T-t_2}(y_2,z) \,  dm(y_1) \, dm(y_2) \nonumber 
\\&\hspace{8mm}= \frac{e^{\lambda_0 t_2}}{e^{\lambda_0T}q_T(x,z)} \iint \phi_1(Y_{t_1}) \, \phi_2(Y_{t_2}) \, q_{t_1}(x,y_1) \,q_{t_2-t_1}(y_1, y_2) \,e^{\lambda_0(T-t_2)}q_{T-t_2}(y_2,z) \,  dm(y_1) \, dm(y_2).
\end{align} 
By Lemma~\ref{lemma_densitylimit}, we have
\begin{equation} \label{end1}
\lim_{T \to \infty} e^{\lambda_0(T-t)}q_{T-t}(y_2,z) = \psi_0(y_2) \psi_0(z), \,\,\,\, \lim_{T \to \infty} e^{\lambda_0 T}q_{T}(x,z) = \psi_0(x) \psi_0(z).
\end{equation}
Moreover, applying (\ref{OU_qbound}) with $\delta = 1/8$, we have that 
\begin{equation} \label{tightaux4}
 e^{\lambda_0(T-t)}q_{T-t}(y_2,z) \leq ce^{y_2^2/8 + z^2/8} \,\,\,\,\,\, \forall \, y_2,z \in \R, t \in (0,S] \, \text{ and }  T \geq S + s^*(1/8),
\end{equation}
where $s^*(1/8)$ is as in \eqref{OU_s_star}. Using \eqref{tightaux4} (replacing $t$ with $t_2$) and \eqref{unkilledtransbd}(i) we obtain the following bound for the integrand in \eqref{e_endpointindep11}:
\begin{align}
&| \phi_1(Y_{t_1}) \, \phi_2(Y_{t_2}) \, q_{t_1}(x,y_1) \,q_{t_2-t_1}(y_1, y_2) \,e^{\lambda_0(T-t_2)}q_{T-t_2}(y_2,z)|\,  dm(y_1) \, dm(y_2) \nonumber
\\ &\hspace{6 mm} \leq e^{z^2/8}\|\phi_1\|_\infty \|\phi_2\|_\infty e^{y_2^2/8} k_{t_1}(x,y_1) k_{t_2 - t_1}(y_1, y_2) \,dy_1 \, dy_2 \nonumber
\end{align}
for all $T \geq S + s^*(1/8)$. By \eqref{unkilledtrans}, $k_{t_1}(x,y_1)$ and $k_{t_2 - t_1}(y_1, y_2)$ are Gaussians with variance at most $1$, and so a short argument shows that the above quantity is integrable. This allows us to use Dominated Convergence in \eqref{e_endpointindep11}, so by \eqref{end1} we have
\begin{align}
& \lim_{T \to \infty} E_x (\phi_1(Y_{t_1}) \, \phi_2(Y_{t_2}) \, \big| \, Y_T = z )  \nonumber
\\ &\hspace{8mm}=\frac{e^{\lambda_0 t_2}}{\psi_0(x) \psi_0(z)} \iint \phi_1(Y_{t_1}) \, \phi_2(Y_{t_2}) \, q_{t_1}(x,y_1) \,q_{t_2-t_1}(y_1, y_2) \,\psi_0(y_2)\,\psi_0(z) \,  dm(y_1) \, dm(y_2)\nonumber
\\ &\hspace{8mm}= \iint \phi_1(Y_{t_1}) \, \phi_2(Y_{t_2}) \, \left[e^{\lambda_0 t_1} q_{t_1}(x,y_1) \frac{\psi_0(y_1)}{\psi_0(x)}\right] \, \left[e^{\lambda_0(t_2 - t_1)}q_{t_2-t_1}(y_1, y_2) \frac{\psi_0(y_1)}{\psi_0(y_2)} \right] \,  dm(y_1) \, dm(y_2)\nonumber
\\ &\hspace{8mm}=\iint \phi_1(Y_{t_1}) \, \phi_2(Y_{t_2}) \, \tilde{q}_{t_1}(x,y_1) \, \tilde{q}_{t_2 - t_1}(y_1, y_2) \,  dm(y_1) \, dm(y_2) = E_x^{Y,\infty}( \phi_1(Y_{t_1}) \, \phi_2(Y_{t_2}) ). \nonumber
\end{align}
The above argument can be easily generalized to $n$-fold FDDs for all $n \geq 2$, (the $\delta=1/8$ in \eqref{tightaux4} can be reduced to handle larger $n$) and thus we have the desired convergence of the FDDs as $T \to \infty$. In order to obtain weak convergence of the laws on $C([0,S],\R)$, we need tightness of the distributions as $T \to \infty$. To prove that the distributions are tight we will analyse the fourth moments of increments, as in \eqref{tightness1}, but first we obtain one more bound. We note that by Lemma~\ref{lemma_densitylimit} and joint continuity of $(T,(x,y)) \to  q_T(x,y)$, it holds that for all $K>0$,
\begin{equation} \label{tightaux5}
 e^{\lambda_0 T}q_T(x,z) \geq \delta(K) > 0 \,\,\,\, \forall\, x,z \in [-K,K], \, T \geq 1
\end{equation}
for sufficiently small $\delta(K) > 0$. Let $K>0$ and $x,z \in [-K,K]$. In \eqref{tightness1}, we bound $q_{T-t}(y_2,z)$ above using \eqref{tightaux4} and bound the other transition densities using \eqref{unkilledtransbd}, which gives
\begin{align} 
&E^{Y}_x ((Y_t-Y_s)^4 \,  | \,\rho > T, Y_T = z) \nonumber 
\\ &\hspace{5 mm}\leq \frac{e^{\lambda_0 t + z^2/8}}{e^{\lambda_0 T}q_T(x,z)} \int k_s(x,y_1) \bigg[\int (y_2 - y_1)^4 \, k_{t-s}(y_1,y_2) \,e^{y_2^2/8} dy_2 \bigg] dy_1 \nonumber
\\ &\hspace{5 mm} \leq e^{\lambda_0 S + K^2/8} \, \delta(K)^{-1} \int k_s(x,y_1) e^{y_1^2/4} \bigg[  \int c \,(y-y_1(1-e^{-(t-s)/2}))^4 p_{t-s}(y) \,e^{y^2/4} \, dy \bigg] dy_1 \nonumber
\\ &\hspace{5 mm} \leq e^{\lambda_0 S + K^2/8} \, \delta(K)^{-1} \int c'\, k_s(x,y_1) e^{y_1^2/4} \bigg[  \int c' \,(y-y_1(1-e^{-(t-s)/2}))^4\, p_{2(t-s)}(y) \, dy \bigg] dy_1 \nonumber
\end{align} 
for all $T\geq S + s^*(1/8)$. In the second inequality we have used \eqref{tightaux5} as well as \eqref{unkilledgaussbd} and a change of variables. The third follows from a short calculation and the fact that $t-s \leq 1$. Applying \eqref{tightaux3} to the above and arguing as in \eqref{tightaux666}, we obtain that, for all $T\geq S + s^*(1/8)$,
\begin{align} \label{inftight1}
&E^{Y}_x ((Y_t-Y_s)^4 \,  | \,\rho > T, Y_T = z) \nonumber 
\\ &\hspace{5 mm} \leq e^{\lambda_0 S + K^2/8} \, \delta(K)^{-1} \, (t-s)^2 \,C \int k_s(x,y_1)\, e^{y_1^2/4} \,(1+|y_1|^4) \,dy_1, \nonumber
\\&\hspace{5 mm}\leq C_3(S,K)\, (t-s)^2 \,\,\, \forall \,x,z \in [-K,K], \, 0\leq s<t\leq S \text{ such that } t-s \leq 1,
\end{align} 
for a constant $C_3(S,K) > 0$, where to see that the integral is bounded uniformly for $|x| \leq K$, we use the fact, from \eqref{unkilledtrans}, that $k_s(x,y_1)$ is Gaussian with mean absolutely bounded by $|x|$ and variance less than $1$. The fact that \eqref{inftight1} holds for all $T \geq S + s^*(1/8)$ implies that the laws $P^Y_x (Y\restrict{[0,S]}\in \cdot \, | \,\rho>T, Y_T = z)$ are tight as $T \to \infty$. Combined with the convergence of the FDDs to those of $P^{Y,\infty}_x$, this proves (b). \\

Observe that \eqref{inftight1} proves part (c) if we restrict to $T \geq S + s^*(1/8)$. If we choose $M\geq 1$ such that $M^{-1} < S < S+s^*(1/8) < M$, then \eqref{bonustightness} gives tightness of the laws for $T \in [S,S+s^*(1/8)]$. Combining these two cases gives desired tightness and proves (c). \qed \\


\section{Some non-linear PDE} \label{s_duality}
Let $\cB_{b^+}(\R)$ denote the space of bounded, non-negative Borel functions. Recall that $S_t$ denotes the semigroup of Brownian motion. By Theorem III.5 of \cite{LG1999}, for $\phi \in \cB_{b^+}(\R)$, there exists a unique non-negative solution, denoted $V^\phi_t(x)$, to the evolution equation
\begin{equation}\label{e_integraleq}
V_t = S_t(\phi) -  \left( \int_0^t S_{t-s}(V_{s}^2/2) \, ds \right),
\end{equation}
such that
\begin{equation} \label{e_LapFun}
E_{X_0}^X \big( e^{-X_t(\phi)} \big) = e^{-X_0(V_t^\phi)}
\end{equation}
for all $X_0 \in \cM_F(\R)$. Applying the above with $X_0 = \delta_x$, \eqref{e_canonPPP} gives
\begin{equation} \label{e_LapFunCanon}
\N_x \big(1-e^{-X_t(\phi)} \big) = V_t^\phi(x).
\end{equation}
We are interested in the case when the initial data is a measure, and also in the differential form of the equation. The integral equation \eqref{e_integraleq} has a corresponding PDE, which is the following:
\begin{equation} \label{e_dualPDE}
\frac{\partial V}{\partial t} = \frac 1 2 \frac{\partial^2 V}{\partial x^2} - \frac{V^2}{2} \hspace{4 mm} \text{ for } (t,x) \in(0,\infty) \times \R, \hspace{3 mm} V_t \to \phi \, \text{  as } t \downarrow 0. 
\end{equation}
In \cite{BF1983}, this equation was shown to have a unique $C^{1,2}$ solution when $\phi \in \cM_F(\R)$, where $V_t \to \phi$ is understood as weak convergence. By Lemma 2.1 of \cite{M2002}, this solution is also the unique solution to \eqref{e_integraleq}. We denote the unique solution to \eqref{e_integraleq} and \eqref{e_dualPDE} by $V^\phi_t$. Part (d) of the same lemma establishes that if $\phi_n \to \phi$ weakly as $n \to \infty$, then $V^{\phi_n}_t(x) \to V^\phi_t(x)$ for all $t>0, x \in \R$. We note from \eqref{e_integraleq} that $V^{\phi_n}_t \leq S_t\phi_n \leq ct^{-1/2}\phi_n(\R)$. Using this and the fact that $X_t$ has a bounded, continuous density, if we approximate measures by functions in $\cB_{b^+}(\R)$, we can take bounded limits in \eqref{e_LapFun} and \eqref{e_LapFunCanon} to establish that \eqref{e_LapFun} and \eqref{e_LapFunCanon} hold for $V^\phi_t$ when $\phi \in \cM_F(\R)$.\\

\textbf{Notation:} As $X_t$ is absolutely continuous, when $\phi \in \cM_F(\R)$ we interpret $X_t(\phi)$ as $\int X(t,x)\, d\phi(x)$.\\

We now state some useful properties of solutions to \eqref{e_dualPDE}. For a proof, see Lemma~2.6 in \cite{M2002}.
\begin{proposition} \label{prop_mono_subadd}
Let $\phi, \psi \in \cM_F(\R)$. \\
(a) (Monotonicity) If $\phi \leq \psi$, then $0 \leq V_t^\phi \leq V_t^\psi$ for all $t>0$.\\
(b) (Sub-additivity) $V_t^{\phi + \psi} \leq V_t^\phi + V_t^\psi$ for all $t>0$.
\end{proposition}

We now let $\phi = \lambda \delta_x \in \cM_F(\R)$ for $\lambda>0$, so that $X_t(\phi) = \lambda X(t,x)$. Denote by $V_t^\lambda$ the unique, non-negative $C^{2,1}$ solution to the initial value problem
\begin{align} \label{e_dualPDElambda}
&\frac{\partial V}{\partial t} = \frac 1 2 \frac{\partial^2 V}{\partial x^2} - \frac{V^2}{2} \hspace{4 mm} \text{ for } (t,x) \in(0,\infty) \times \R, \hspace{3 mm} V_t \to \lambda \delta_0 \, \text{  weakly as } t \downarrow 0. 
\end{align}
This family was originally studied in \cite{KP1985}. It is an exercise to use \eqref{e_dualPDElambda} or the scaling properties of super-Brownian motion to show that $V^\lambda_t(x)$ satisfies the following space-time scaling relationship. For $\lambda,r>0$, we have
\begin{equation} \label{e_Vscale}
V^{\lambda r}_t(x) = \lambda^2 V^r_{\lambda^2 t} (\lambda x).
\end{equation}

By translation invariance in the initial conditions of (\ref{e_dualPDElambda}), and by (\ref{e_LapFun}) and (\ref{e_LapFunCanon}) we have
\begin{align}
&E_{\delta_0}^X \big( e^{-\lambda X(t,x)} \big) = e^{-V^\lambda_t(x)}, \label{e_LapFunlambda}
\\ &\N_0 \big(1-e^{-\lambda X(t,x)} \big) = V_t^\lambda(x)     \label{e_LapFunCanonlambda}
\end{align}
for all $x \in \R$ and $t>0$. It is clear from (\ref{e_LapFunlambda}) that $V^\lambda_t$ increases to a limit as $\lambda \to \infty$. In the PDE literature this was established in \cite{KP1985}, where it was shown that $V^\lambda$ converges locally uniformly as $\lambda \to \infty$ to a function $V^\infty_t$ on $(0,\infty) \times \R$. Heuristically, $V^\infty_t$ is the solution of (\ref{e_dualPDElambda}) when $\lambda = + \infty$. Rigorously, it is the unique solution to the following problem:
\begin{align} \label{e_dualPDEinfty}
&\frac{\partial V}{\partial t} = \frac 1 2 \frac{\partial^2 V}{\partial x^2} - \frac{V^2}{2} \hspace{4 mm} \text{ for } (t,x) \in(0,\infty) \times \R, \nonumber
\\ & \lim_{t\downarrow 0} V_t(x) = 0 \,\,\,\forall x \neq 0, \,\,\,\, \lim_{t\downarrow 0} \int_{B_\epsilon} V^\infty_t(x) \, dx = + \infty \,\,\, \forall \,\epsilon>0, 
\end{align}
where $B_\epsilon = B(0, \epsilon)$, the ball with radius $\epsilon$ centered at the origin. $V^\infty_t$ was introduced and shown to solve \eqref{e_dualPDEinfty} in \cite{BPT1986}; uniqueness of the solution is a consequence of Theorem 3.5 of \cite{MV1999}. Taking $\lambda \to \infty$ in \eqref{e_LapFunCanonlambda}, we see that $V^\infty_t$ satisfies
\begin{equation} \label{e_Vinf_prob}
V^\infty_t(x) = \N_0(\{X(t,x) > 0\}).
\end{equation}
We recall that (see Theorem II.7.2 of \cite{P2002})
\begin{equation} \label{e_Xtsurvive}
\N_0 (\{X_t > 0\}) = 2/t.
\end{equation}
Thus \eqref{e_Vinf_prob} implies that
\begin{equation}\label{e_Vinf_tbd}
V^\infty_t(x) \leq 2/t \,\,\, \forall \, x.
\end{equation}
Taking $\lambda^2 = 1/t$ and letting $r \to \infty$ in \eqref{e_Vscale}, one obtains that $V^\infty_t(x) = t^{-1} V^\infty_1(t^{-1/2}x)$. \\

\textbf{Definition:} $F(x) := V_1^\infty(x)$.\\

Then we have $V_t^\infty(x) = t^{-1}F(t^{-1/2} x)$. It was shown in \cite{BPT1986} that $F$ is the solution to an ODE problem. (In fact, their PDEs and ODEs have different (constant) coefficients, but Section~3 of \cite{MMP2017} shows that $F$ is a rescaled version of the function they study.) $F$ is the unique solution of
\begin{align} \label{e_FODE}
(i)&\,\,  F''(x) +  xF'(x) + F(x)(2 - F(x)) = 0 \nonumber
\\ (ii)&\, \,F>0, F\in C^2(\R)
\\ (iii)&\,\, F'(0) = 0, F(x) \sim c_1 |x| e^{-x^2/2} \, \text{ as }\, |x| \to \infty \nonumber
\end{align}
for some $c_1 >0$. We recall that $f(x) \sim h(x)$ means $f(x) / h(x) \to 1$ as $x \to \infty$. This $F$ is the function we discussed in the introduction, for which $-\lambda_0$ is the lead eigenvalue of the operator $A^F$. In particular, by evaluating \eqref{e_Vinf_prob} at $t=1$ we can recover \eqref{e_Fdefintro}, our preliminary definition.\\

As part of the proof of Theorem A, the authors of \cite{MMP2017} computed the rate of convergence of $V^\lambda_t$ to $V^\infty_t$. In particular, Proposition~4.6 of that reference states that
\begin{equation} \label{e_VlambdaConvergence}
\sup_{x\in \R} \left[V^\infty_t(x) - V^\lambda_t(x)\right] \leq C t^{-1/2 - \lambda_0} \lambda^{1-2\lambda_0}
\end{equation}
for some constant $C$. (This is closely connected to \eqref{e_densityLaplace}.) A similar lower bound with the same power of $\lambda$ is established in the same proposition. We will make frequent use of \eqref{e_VlambdaConvergence} in this work to bound error terms arising when we make approximations to obtain an eigenvalue problem. Let $Y$ be an Ornstein-Uhlenbeck process. We define $Z_T(Y)$ as 
\begin{equation} \label{e_Zdef}
Z_T(Y) = \exp \bigg( \int_0^T F(Y_s) - V_1^{e^{s/2}}(Y_s) \, ds \bigg).
\end{equation}
Since $V_1^{e^{s/2}} \uparrow V^\infty_1 = F$ as $s \to \infty$, the integrand is converging to zero. As $Z_T(Y)$ is increasing in $T$, we can define $Z_\infty(Y) := \lim_{T \to \infty} Z_T(Y)$. By (\ref{e_VlambdaConvergence}), we can easily deduce that the (monotone) limit
\begin{equation} \label{e_Zinf}
Z_\infty(Y) := \lim_{T \to \infty} Z_T(Y)
\end{equation} 
exists and is finite, and that moreover there is a constant $C_Z>0$ such that, uniformly for all $Y$,
\begin{equation} \label{e_ZTbd}
Z_T(Y)\leq Z_\infty(Y) \leq C_Z < \infty \,\,\, \forall \, T>0.
\end{equation}

Finally, we introduce another family of solutions to (\ref{e_dualPDE}), which arise when we compute second moments of $L^\lambda_t$; we will evaluate expressions that involve the density at two points $x_1, x_2 \in \R$. Let $V_t^{(\lambda, \lambda'),(x_1,x_2)}$ denote $V_t^\phi$ when $\phi = \lambda \delta_{x_1} + \lambda' \delta_{x_2} \in \cM_F(\R)$. When we evaluate this function at $0$ we simply write $V_t^{(\lambda, \lambda'),(x_1,x_2)}(0) = V_t^{\lambda, \lambda'}(x_1,x_2)$. In particular, by \eqref{e_dualPDE}, this is equivalent to
\begin{equation} \label{V2ptdef}
V^{\lambda, \lambda'}_t(x_1,x_2) = \N_0 \big(1 - e^{-\lambda X(t,x_1) - \lambda' X(t,x_2)} \big).
\end{equation}
We will make frequent use of the fact that this family of solutions is translation invariant in the initial conditions of (\ref{e_dualPDE}). This implies that $V_t^{(\lambda, \lambda'),(x_1,x_2)}(y) = V_t^{\lambda, \lambda'}(y-x_1,y-x_2)$. The family satisfies the following scaling relationship:
\begin{equation} \label{e_Vscale_2pt}
V^{r \lambda, c\lambda'}_t(x_1,x_2) = \lambda^2 V^{r,c \lambda'/\lambda }_{\lambda^2 t}(\lambda x_1, \lambda x_2) = (\lambda')^2 V^{r\lambda/\lambda' ,c }_{(\lambda')^2 t}(\lambda' x_1, \lambda' x_2),
\end{equation}
for all $\lambda, \lambda', r,c>0$. Taking limits and applying bounded convergence in (\ref{e_LapFun}), we see that $V_t^{\lambda, \lambda'}(x_1,x_2)$ has a monotone limit as $\lambda, \lambda' \to \infty$ (by Proposition~\ref{prop_mono_subadd}(a)). We denote this limit $V_t^{\infty,\infty}(x_1,x_2)$. In agreement with our previous notation we define the following.\\

\textbf{Definition.} $F_2(x_1,x_2) := V_1^{\infty,\infty}(x_1,x_2)$.\\

By taking the limit as $\lambda, \lambda' \to \infty$ in \eqref{V2ptdef} (and in \eqref{e_LapFun} with $\phi = \lambda \delta_{x_1} + \lambda' \delta_{x_2}$) we obtain that
\begin{equation} \label{V2infprob}
V_t^{\infty,\infty}(x_1,x_2) = \N_0(\{X(t,x_1)>0\} \cup \{X(t,x_2) > 0\} ) = -\log P^X_{\delta_0}(\{X(t,x_1) = X(t,x_2) = 0\} ). 
\end{equation}

We conclude by stating a version of (\ref{e_VlambdaConvergence}) for the functions $V^{\lambda, \lambda'}_t$.
\begin{lemma} \label{lemma_V2pt_ROC}
There is a positive constant $C$ such that for all $t,\lambda,\lambda' >0 $,
\begin{equation}
\sup_{x_1,x_2 \in \R} \left[V^{\infty, \infty}_t(x_1,x_2) - V^{\lambda, \lambda'}_t(x_1,x_2)\right] \leq C t^{-1/2 - \lambda_0} \left[ \lambda^{1-2\lambda_0} + \lambda'^{1-2\lambda_0} \right]. \nonumber
\end{equation}
\end{lemma}
\begin{proof}
Let $x_1, x_2 \in \R$ and $t, \lambda, \lambda' > 0$. We write
\begin{align} \label{exp_2ptVlemma1}
&V^{\infty,\infty}_t(x_1,x_2) - V^{\lambda,\lambda'}_t(x_1,x_2) \nonumber
\\ &\hspace{6 mm}= \left[V^{\infty,\infty}_t(x_1,x_2)-V^{\lambda, \infty}_t(x_1,x_2) \right] + \left[V^{\lambda,\infty}_t(x_1,x_2)-V^{\lambda, \lambda'}_t(x_1,x_2) \right].
\end{align}
By \eqref{V2infprob} and  \eqref{e_LapFunCanon}, the first term is equal to
\begin{align} 
\N_0 &\left(1 - 1( X(t,x_1) = X(t,x_2) = 0 )\right) - \N_0\left(1- e^{-\lambda X(t,x_1)} 1(X(t,x_2) = 0 ) \right) \nonumber
\\ &= \N_0 \left( 1(X(t,x_2) = 0) \left(e^{-\lambda X(t,x_1)} - 1(X(t,x_1) = 0) \right) \right) \nonumber
\\ &\leq \N_0 \left( e^{-\lambda X(t,x_1)} - 1(X(t,x_1) = 0) \right) \nonumber
\\ &= V^\infty_t(x_1) - V^\lambda_t(x_1) \nonumber
\\ &\leq Ct^{-1/2 - \lambda_0}  \lambda^{1-2\lambda_0},  \nonumber
\end{align}
where the second last line follows from \eqref{e_Vinf_prob} and \eqref{e_LapFunCanon}, and the final inequality is by \eqref{e_VlambdaConvergence}. We use similar reasoning to bound the second term of \eqref{exp_2ptVlemma1} by the same expression with $\lambda'$ replacing $\lambda$, which gives the desired result. \end{proof}

\section{Existence and Properties of $L_t$} \label{s_main}
As stated in the introduction, our method first establishes the existence and properties of $L_t$ under $\N_0$ and then uses the cluster decomposition to establish them under $P^X_{X_0}$. The main ingredient in the proof of Theorem \ref{thm_Lt} is the convergence of second moments of $L^\lambda_t(\phi)$ as $\lambda \to \infty$. For a bounded Borel function $\phi$, we show that $\N_0 ( L^\lambda_t(\phi)^2 )$ converges as $\lambda \to \infty$. In fact, we prove convergence of second moments of general functions of two variables. For $h : \R^2 \to \R$ we recall the notation
\begin{equation} 
(L_t^\lambda \times L_t^\lambda)(h) = \int h(x,y)\, dL_t^\lambda(x) \, dL_t^\lambda(y). \nonumber
\end{equation}
$L_t^\lambda(\phi)^2$ is easily recovered by taking $h(x,y) = \phi(x) \, \phi(y)$. The following result is the workhorse of this paper.\\

\textbf{Theorem~\ref{thm_l2limit}.}\emph{
There is a constant $C_{\ref{thm_l2limit}} > 0$ and continuous function $\rho:\R \times \R \to (0,1]$ such that for bounded Borel $h : \R^2 \to \R$,
\begin{align}
&\lim_{\lambda, \lambda' \to \infty} \N_0 ((L^{\lambda}_t \times L^{\lambda'}_t)(h) ) \nonumber
\\ &\hspace{ 6mm}=  C_{\ref{thm_l2limit}}^2 \int_0^t (t-s)^{-2\lambda_0} \bigg[ \iint E^{B}_0 \bigg(  \exp \left( - \int_0^s V^{\infty, \infty}_{t-u} (\sqrt{t-s} \, z_1 + B_s - B_u, \sqrt{t-s} \, z_2 + B_s - B_u) \, du \right) \nonumber
\\ & \hspace{10 mm}\times  h(\sqrt{t-s} \, z_1 + B_s,\sqrt{t-s} \, z_2 + B_s) \bigg)\, \rho(z_1, z_2) \,\psi_0(z_1) \, \psi_0(z_2) \, dm(z_1) \, dm(z_2)  \bigg] ds. \nonumber
\end{align}}
\begin{corollary} \label{cor_l2limit}
For a bounded Borel function $\phi$, $L^{\lambda}_t(\phi)$ converges in $\cL^2(\N_0)$ as $\lambda \to \infty$.
\end{corollary}
\begin{proof}
Since $\cL^2(\N_0)$ is complete, it is enough to show that $\{L^\lambda_t(\phi)\}_{\lambda > 0}$ is Cauchy in $\cL^2(\N_0)$. For $\lambda, \lambda' > 0$, we have
\begin{align*}
\N_0 ( (L^{\lambda}_t(\phi) - L^{\lambda'}_t(\phi))^2 ) &= \N_0 ( (L^{\lambda}_t(\phi)^2 ) + \N_0 ( (L^{\lambda'}_t(\phi)^2 )  -2\N_0 ( L^{\lambda}_t(\phi)L^{\lambda'}_t(\phi) ).
\end{align*}
By Theorem \ref{thm_l2limit}, this converges to $0$ as $\lambda, \lambda' \to \infty$. \end{proof}

The proof of Theorem \ref{thm_l2limit} is long and technical. We defer it to Section \ref{s_momentsconvergence}, which is devoted to its proof. For now, we take it as a given and use it to establish our other main results, the first being the existence of $L_t$ under $\N_0$. \\

\emph{Proof of Theorem \ref{thm_Lt} for $\N_0$.} Fix $t>0$. Because $X_t = 0$ implies that $L^\lambda_t = 0$ for all $\lambda > 0$, without loss of generality we can work under the finite measure $\N_0(\cdot \cap \{X_t > 0 \})$. By Corollary~\ref{cor_l2limit}, for a bounded continuous function $\phi$, there exists a random variable $l(t,\phi)$ such that $L^\lambda_t(\phi)\to l(t,\phi)$ in $\cL^2(\N_0)$ as $\lambda \to \infty$. It follows that $L^\lambda_t(\phi)\to l(t,\phi)$ in measure. We will now establish that there exists a unique random measure $L_t$ such that the random variable $l(t,\phi)$ is the integral of $\phi$ with respect to a random measure $L_t$, ie. $l(t,\phi) = L_t(\phi)$ for all continuous and bounded functions $\phi$.\\

We need to establish that the measures $\{L^\lambda_t : \lambda>0\}$ are tight $\N_0$-almost surely. To see that this is true, we recall that $X(t,\cdot)$ is compactly supported $\N_0$-a.s., (see Corollary III.1.4 of \cite{P2002} for the result under $P^X_{\delta_0}$; condition the cluster representation on $N=1$ to get it for $\N_0$) and hence the mass of $X_t$ is contained in a ball $B(0,R)$ for some $R = R(\omega) > 0$. Since $L^\lambda_t(A) = \lambda^{2\lambda_0}  \int_A X(t,x) e^{-\lambda X(t,x)}\, dx$, this implies that the mass of $L^\lambda_t$ is contained in $B(0,R)$ for all $\lambda >0$, which implies that $\{L^\lambda_t(\omega) : \lambda >0\}$ is tight.\\

Let $\{\phi_n\}_{n=1}^\infty$ be a countable determining class for $\cM_F(\R)$ consisting of bounded, continuous functions. We choose $\phi_1 = 1$. $\cL^1$-boundedness of the total mass and tightness are sufficient conditions for a family in $\cM_F(\R)$ (with the weak topology) to be relatively compact. By Corollary~\ref{cor_l2limit}, $\{L^\lambda_t(1) : \lambda>0\}$ is $\cL^2(\N_0)$-bounded, and hence $\cL^1(\N_0)$-bounded, and so from the above we see that
\[\{L^{\lambda}_t : \lambda > 0 \} \, \text{ is relatively compact } \, \N_0\text{-a.s.}  \]

As we have noted, $L^\lambda_t(\phi_n) \to l(t,\phi_n)$ in measure as $\lambda \to \infty$. Using the fact that convergence in measure implies almost sure convergence along a subsequence, we can iteratively define subsequences and take a diagonal subsequence $\{\lambda_m\}_{m=1}^\infty$ which satisfies
\begin{equation} \label{e_measureconverge1}
L_t^{\lambda_m}(\phi_n) \to l(t, \phi_n) \, \text{  as } m \to \infty 
\,\, \text{ for all } n\geq 1 \,\,\,\ \N_0\text{-a.s.}
\end{equation} 
As we have noted, $\{L^{\lambda_m}_t\}_{m=1}^\infty$ is relatively compact $\N_0$-almost surely. Combined with the above, this means that for $\N_0$-a.a. $\omega$ we have the above convergence for all $n$ and relative compactness of the measures $\{L^{\lambda_m}_t\}_{m=1}^\infty$. Choose such an $\omega$. By relative compactness of $\{L^\lambda_t\}_{\lambda>0}$, any subsequence of $\{\lambda_m\}_{m=1}^\infty$ admits a further sequence along which the measures converge in the weak topology. It remains to show that all subsequential limits coincide. Suppose $L_t(\omega)$ and $L_t'(\omega)$ are two such limit measures. Since $\omega$ has been chosen so that (\ref{e_measureconverge1}) holds, we have that $L_t(\omega)(\phi_n) = L_t'(\omega)(\phi_n)$ for all $n$. Since the family $\{\phi_n\}_{n\geq 1}$ are a determining class, this implies that $L_t(\omega) = L_t'(\omega)$. Hence all subsequences admit a further subsequence with the same limit $L_t(\omega)$ in the weak topology. Since the weak topology on $\cM_F(\R)$ is metrizable, the ``every subsequence admits a further converging subsequence" criterion for convergence applies, and we have $L^{\lambda_m}_t(\omega)$ converges to $L_t(\omega)\in \cM_F(\R)$ as $m \to \infty$. This gives the almost sure convergence along $\{\lambda_m\}_{m=1}^\infty$. As the weak topology is metrizable we also have $L^\lambda_t \to L_t$ in measure. Furthermore, we observe that for continuous and bounded $\phi$, $L_t(\phi) = l(t,\phi)$. To see this, recall that $L^{\lambda_m}_t(\phi)$ converges to $l(t,\phi)$ in $\cL^2(\N_0)$. As we have just shown that $\lim_{m \to \infty} L^{\lambda_m}_t(\phi) = L_t(\phi)$ $\N_0$-a.s, it must hold that $L_t(\phi) = l(t,\phi)$. This implies that $L^\lambda_t(\phi) \to L_t(\phi)$ in $\cL^2(\N_0)$.\\ 

Finally, we verify that $L_t$ is supported on $BZ_t$. We fix $\omega$ outside of a null set such that $L^{\lambda_m}_t \to L_t$ in $\cM_F(\R)$ as $m \to \infty$. For an open set $U$, $L_t(U) \leq \liminf_{m \to \infty}L^{\lambda_m}_t(U)$ (a consequence of the Portmanteau theorem). From \eqref{e_Llambdadef}, we have $L^{\lambda_m}_t(Z_t) = 0$ for all $m \geq 1$, which implies that $L_t(\text{int}(Z_t)) = 0$. Moreover, $X(t,x) > 0$  implies that $\lambda_m^{2\lambda_0} X(t,x)e^{-\lambda_m X(t,x)} \to 0$ as $m \to \infty$, so for $\epsilon >0$, $L_t(\{x : X(t,x) > \epsilon\}) = 0$, and hence $L_t(Z_t^c) = 0$. Since $L_t( \text{int}(Z_t) \cup Z_t^c ) = 0$, we must have $\text{supp}(L_t) \subseteq BZ_t$. \qed \\

\emph{Proof of Theorem~\ref{thm_Ltcanonmoments}.} To prove (b), by Theorem~\ref{thm_l2limit} it is enough to show that $\N_0 ((L_t \times L_t)(h)) = \lim_{n \to \infty} \N_0 ((L^{\lambda_n}_t \times L^{\lambda_n}_t)(h))$ for a sequence $\lambda_n \to \infty$, which we choose to be the sequence from Theorem~\ref{thm_Lt} on which $L^{\lambda_n}_t \to L_t$ almost surely. Because $L_t = 0$ when $X_t = 0$, we can work on the probability measure $\N_0(\cdot \, | \, X_t > 0 )$. For bounded and continuous $h:\R^2 \to \R$, $| (L_t \times L_t)(h) | \leq \|h\|_\infty L_t^{\lambda_n}(1)^2$. By Theorem~\ref{thm_Lt}, $L^{\lambda_n}_t(1)$ converges in probability and in $\cL^2(\N_0(\cdot\, | \, X_t>0))$ to $L_t(1)$, which implies that $L^{\lambda_n}_t(1)^2$ and hence $(L^{\lambda_n}_t \times L^{\lambda_n})(h)$ are uniformly integrable (see, e.g. Theorem 4.6.3 of \cite{D}). We can therefore exchange limit and expectation, giving 
\[\N_0 (\lim_{n \to \infty} (L_t^{\lambda_n} \times L^{\lambda_n}_t)(h)) = \lim_{n \to \infty} \N_0 ((L^{\lambda_n}_t \times L^{\lambda_n}_t)(h)).\] 
Since $L^{\lambda_n}_t \to L_t$ in $\cM_F(\R)$ and $h$ is bounded and continuous, the integrand on the left hand side is equal to $(L_t \times L_t)(h)$, which gives the result. By a Monotone Class Theorem (e.g. Corollary 4.4 in the Appendix of Ethier and Kurtz \cite{EK}), the same holds for all bounded and measurable $h$. \\

We now turn to part (a). Let $\phi:\R \to \R$ be bounded and Borel. We recall from the Introduction (see \eqref{e_densityLaplace}) that Proposition 4.5 of \cite{MMP2017} states that
\begin{align} \label{canonmoment1}
&\lim_{\lambda \to \infty} E^X_{X_0}(L^\lambda_t(\phi)) \nonumber 
\\&\hspace{ 6 mm} = C_{\ref{thm_l2limit}}\, t^{-\lambda_0} \iint \phi(x_0 + \sqrt t z) \exp \left( -\frac 1 t \int F(z + t^{-1/2} (x_0 - y_0) \, dX_0(y_0) \right) \psi_0 (z)\, dm(z) \,dX_0(x_0). 
\end{align}
(The fact that the constant appearing in Proposition~4.5 of \cite{MMP2017} equals $C_{\ref{thm_l2limit}}$ is implicit in the proof.) The proof uses the Palm measure formula for $X_t$ under $P^X_{X_0}$; see Theorem 4.1.3 of Dawson-Perkins \cite{DP}. The corresponding Palm measure formula for the superprocess under $\N_0$ is in fact simpler, and the same proof shows that
\begin{align}\label{canonmoment2}
&\lim_{\lambda \to \infty} \N_0(L^\lambda_t(\phi)) = C_{\ref{thm_l2limit}}\, t^{-\lambda_0} \int \phi(\sqrt t z)\, \psi_0 (z)\, dm(z). 
\end{align}
Consider now a bounded and continuous function $\phi$; we can also clearly assume that $\phi \geq 0$. By Theorem~\ref{thm_Lt} (under $\N_0$), $L^\lambda_t(\phi)$ converges in $\cL^2$ with respect to the probability measure $\N_0 ( X_t \in \cdot\,|\,X_t>0)$, which implies that it also converges in $\cL^1$, allowing us to exchange limit and expectation in \eqref{canonmoment2}, which gives part (a) for bounded and continuous $\phi$. This extends to all bounded and measurable $\phi$ by a monotone class argument (as above for part (b)). Finally, it is clear that both (a) and (b) hold for general non-negative functions by the Monotone Convergence Theorem. \qed \\

We now describe how to ascertain the existence of $L_t$ when $X_t$ is a super-Brownian motion under $P^X_{X_0}$ via the cluster representation. In particular, we recall \eqref{e_XtPPP} and \eqref{e_clusterrep}. Let $X_0 \in \cM_F(\R)$ and $t>0$. \\

\emph{Proof of Theorem~\ref{thm_Lt}~for~$P^X_{X_0}$~and~Theorem~\ref{thm_clusterLt}.} Let $N, x_1, \hdots, x_N, X^1_t, \hdots X^N_t$ be as in the cluster decomposition \eqref{e_clusterrep}. For $\lambda>0$, define the measure $L^\lambda_t$ via \eqref{e_Llambdadef} using $X_t$. For $i = 1,\hdots,N$, let $L^{i,\lambda}_t$ denote the measure defined in \eqref{e_Llambdadef} corresponding to $X^i_t$. By Theorem~\ref{thm_Lt} for $\N_0$ and translation invariance, $\N_{x_i}(X_t^i \in \cdot \, | \, X_t^i>0)$-a.s. there exists $L_t^i$ such that $L^{i,\lambda}_t \to L_t$ in $\cM_F(\R)$ in measure. Define $L_t \in \cM_F(\R)$ by \eqref{e_clusterdecomp}. That is,
\[dL_t(x) = \sum_{i=1}^{N} 1\big(\sum_{j \neq i} X^j(t,x) = 0\big) \, dL^i_t(x) .\]
Let $\phi:\R \to \R$ be bounded and continuous. We will show that
\begin{equation} \label{e_clusterprobcon}
L^\lambda_t(\phi) \to L_t(\phi) \, \text{ in probability as } \, \lambda \to \infty.
\end{equation}
Once we establish \eqref{e_clusterprobcon}, the proof of Theorem~\ref{thm_Lt} for $\N_0$ applies and shows that $L^\lambda_t \to L_t$ in probability in $\cM_F(\R)$ as $\lambda \to \infty$. With the exception of $\cL^2$ convergence, which we show afterward, this proves Theorem~\ref{thm_Lt} for $P^X_{X_0}$. Furthermore, since $L_t$ is defined by \eqref{e_clusterdecomp}, this also proves Theorem~\ref{thm_clusterLt}.\\

Turning to \eqref{e_clusterprobcon}, we will argue conditionally on $(N,x_1, \hdots, x_N)$. That is, we argue under the regular conditional distribution for $(X^1_t,\hdots, X^N_t)$ given $(N,x_1, \hdots, x_N)$. As such, we treat $N \geq 1$ and $x_1,\hdots,x_N \in \R$ as fixed, and $X^1_t,\hdots,X^N_t$ are independent random measures with respective laws $\N_{x_i}(X_t \in \cdot \, | \, X_t>0)$ for $i=1,\hdots,N$. Let $E$ denote the expectation of a probability realizing this conditional representation for $X_t$. Expanding $L^\lambda_t(\phi)$ in terms of the clusters, we have	
\begin{align} \label{e_clusteraux1}
L^\lambda_t(\phi) &=  \int \lambda^{2\lambda_0} X(t,x)e^{-\lambda X(t,x)} \, \phi(x)\, dx \nonumber
\\ &=  \int  \lambda^{2\lambda_0}\left[ \sum_{i=1}^{N} X^i(t,x) \right] e^{-\lambda \sum_{i=1}^{N} X^i(t,x)}\phi(x) \, dx \nonumber 
\\ &= \sum_{i=1}^{N} \int \lambda^{2\lambda_0}  X^i(t,x) e^{-\lambda X^i(t,x)}\,\left[ e^{-\lambda \sum_{j\neq i} X^j(t,x)} \phi(x) \right]  dx \nonumber 
\\ &= \sum_{i=1}^{N} L^{i,\lambda}_t \big(\phi \cdot e^{-\lambda Z^i_N(t,\cdot)} \big),
\end{align}
where we define $Z^i_N(t,x) = \sum_{j\neq i} X^j(t,x)$, in which the indices are understood to sum from $1$ to $N$. Using this notation, $L_t(\phi) = \sum_{i=1}^N L_t^i(\phi \cdot 1(Z^i_N(t,\cdot) = 0) )$. Thus by \eqref{e_clusteraux1}, to prove \eqref{e_clusterprobcon} it is clearly enough to show that for any $1\leq i \leq N$,
\begin{equation} \label{clusterconditionprob}
L^{i,\lambda}_t \big(\phi \cdot e^{-\lambda Z^i_N(t,\cdot)} \big) \to L_t^i(\phi \cdot 1(Z^i_N(t,\cdot) = 0) ) \, \text{ in probability as } \, \lambda \to \infty.
\end{equation}
Without loss of generality, assume that $\lambda>1$. Let $1\leq\lambda' \leq \lambda$. Then
\begin{align} \label{e_clusteraux2}
|L^{i,\lambda}_t (\phi &\cdot e^{-\lambda Z^i_N(t,\cdot)} ) - L_t^i(\phi \cdot 1\big(Z^i_N(t,\cdot) = 0) )| \nonumber
\\ \leq \,&|L^{i,\lambda}_t (\phi \cdot ( e^{-\lambda Z^i_N(t,\cdot)} - e^{-\lambda' Z^i_N(t,\cdot)} ) )| + |L^{i,\lambda}_t (\phi \cdot e^{-\lambda' Z^i_N(t,\cdot)} ) - L_t^i(\phi \cdot e^{-\lambda' Z^i_N(t,\cdot)} )|  \nonumber
\\ &\hspace{5 mm} + |L^{i}_t (\phi \cdot ( e^{-\lambda' Z^i_N(t,\cdot)} - 1(Z^i_N(t,\cdot) = 0) ) ) | \nonumber
\\ \leq \,&\|\phi\|_\infty |L^{i,\lambda}_t(e^{-\lambda' Z^i_N(t,\cdot)}1(Z^i_N(t,\cdot)>0))| +  |L^{i,\lambda}_t (\phi \cdot e^{-\lambda' Z^i_N(t,\cdot)} ) - L_t^i(\phi \cdot e^{-\lambda' Z^i_N(t,\cdot)} )|  \nonumber
\\ & \hspace{5 mm}	+\|\phi\|_\infty L^{i}_t(e^{-\lambda' Z^i_N(t,\cdot)}1(Z^i_N(t,\cdot)>0))  \nonumber
\\ =: &\|\phi\|_\infty R_1(\lambda',\lambda) + R_2(\phi,\lambda') + \|\phi\|_\infty R_3(\lambda',\lambda). 
\end{align}
We first consider $R_1$. Since $X^i_t$ and $Z^i_N(t,\cdot)$ are independent and $L_t^{\lambda}$ is a measurable function of $X_t^i$, conditional on $X^i_t$ we have, for all $\lambda > 1$ and $1 \leq \lambda' < \lambda$,
\begin{align}
E(R_1(\lambda',\lambda) \,|\, X^i_t ) &= \int E(e^{-\lambda' Z^i_N(t,x)}1(Z^i_N(t,x)>0)) \,| \,X^i_t) \,dL^{i,\lambda}_t(x) \nonumber
\\ &= \int E(e^{-\lambda' \sum_{j\neq i} X^j(t,x)}1(\sum_{j\neq i}X^j(t,x) > 0) ) \,dL^{i,\lambda}_t(x) \nonumber
\\ &\leq \sum_{j \neq i} \int \N_{x_j}(e^{-\lambda' X^j(t,x)}1(X^j(t,x) > 0) \, | \, X_t^j > 0) \,dL^{i,\lambda}_t(x) \nonumber
\\ &= \sum_{j \neq i} \N_{x_j}(X_t^j > 0 )^{-1} \int \N_{x_j}( e^{-\lambda' X^j(t,x)} -1(X^j(t,x) = 0)) \,dL^{i,\lambda}_t(x) \nonumber
\\ &= (t/2) \sum_{j \neq i} \int \N_{x_j}( 1- 1(X^j(t,x) = 0)) - \N_{x_j}(1-e^{-\lambda' X^j(t,x)} )  \,dL^{i,\lambda}_t(x) \nonumber
\\ &= (t/2) \sum_{j \neq i} \int V^\infty_t(x-x_j) - V^{\lambda'}_t(x-x_j)\, dL^{i,\lambda}_t(x),
\end{align}
where in the second last line we have used \eqref{e_Xtsurvive}, and the last follows from  \eqref{e_LapFunCanonlambda}, \eqref{e_Vinf_prob}, and translation invariance. We apply \eqref{e_VlambdaConvergence} to the integrand and take the expectation of the above to obtain that
\begin{align} \label{R1bd}
E(R_1(\lambda',\lambda)) &\leq \frac{N-1}{2} t^{1/2 - \lambda_0} \lambda'^{-(2\lambda_0-1)} \N_{x_i} (L^{i,\lambda}_t(1) \, | \, X^i_t>0) \nonumber
\\ &= \frac{N-1}{4} t^{3/2 - \lambda_0} \lambda'^{-(2\lambda_0-1)} \N_{x_i} (L^{i,\lambda}_t(1) ) \hspace{22 mm} \text{ (by \eqref{e_Xtsurvive})} \nonumber
\\ &\leq C(t,N) \lambda'^{-(2\lambda_0-1)}, 
\end{align}
for all $\lambda>1$ and $1 \leq \lambda' < \lambda$, where the last inequality is by Theorem~\ref{thm_Ltcanonmoments}(a) and the fact that $L^{i,\lambda}_t(1) \to L^i_t(1)$ in $\cL^2(\N_{x_i})$ (from Theorem~\ref{thm_Lt}). Next we consider $R_3$. Note that we can expand and bound this term in exactly the same way as we did $R_1$ in \eqref{e_clusteraux2} but with $L^i_t$ replacing $L^{i,\lambda}_t$. Taking the expectation and proceeding as above then gives
\begin{equation} \label{R3bd}
E(R_3) \leq \frac{N-1}{4} t^{3/2 - \lambda_0} \N_{x_i} (L^{i}_t(1) ) \lambda'^{-(2\lambda_0-1)} .
\end{equation}
Fix $\delta > 0$. By \eqref{R1bd} and \eqref{R3bd} and Markov's inequality there exists $\bar{\lambda}(\delta)$ such that for $\lambda' \geq \bar{\lambda}(\delta)$,
\begin{equation} \label{R1R3con}
P(R_1(\lambda',\lambda) > \delta) +  P(R_3(\lambda',\lambda) > \delta) < C'(t,N) \lambda'^{-(2\lambda_0-1)} / \delta.
\end{equation}
Now consider $R_2(\phi)$. Since $\phi \cdot e^{-\lambda' Z^i_N(t,\cdot)}$ is a bounded, continuous function for all $\lambda' \geq 1$, by Theorem~\ref{thm_Lt} for $\N_{x_i}$, $R_2(\phi,\lambda') \to 0$ in probability as $\lambda \to \infty$ for all $\lambda' \geq 1$. From this and \eqref{R1R3con} we conclude, by choosing $\lambda' \leq \lambda$ sufficiently large, that \eqref{e_clusteraux2} converges to $0$ in probability as $\lambda \to \infty$. As we noted in \eqref{clusterconditionprob}, this is sufficient to prove the result. \\

It remains to show that $L^\lambda_t(\phi) \to L_t(\phi)$ in $\cL^2(P^X_{X_0})$ for all continuous and bounded functions $\phi$. Let $\phi$ be such a function, and suppose that $X_t$ is realized as in \eqref{e_clusterrep} under a probability $P^X_{X_0}$. Under $P^X_{X_0}(\cdot \, | \, N)$, from \eqref{e_clusteraux1} and \eqref{e_clusterdecomp} we have
\begin{align} \label{e_clusterL21}
(L^\lambda_t(\phi) - L_t(\phi))^2 &= \bigg( \sum_{i=1}^N L^{i,\lambda}_t(e^{-\lambda Z^i_N(t,\cdot)} \cdot \phi) - L^i_t(1(Z^i_N(t,\cdot)=0) \cdot \phi) \bigg)^2 \nonumber
\\ &\leq N \sum_{i=1}^N \,\big[ L^{i,\lambda}_t(e^{-\lambda Z^i_N(t,\cdot)} \cdot \phi) - L^i_t(1(Z^i_N(t,\cdot)=0) \cdot \phi) \big]^2.
\end{align}
We recall that $X^1_t,\ldots,X^N_t$ are iid with distribution $\N_{\bar{X}_0}(X_t \in \cdot \, | \, X_t > 0)$, where $\bar{X}_0 = X_0(\cdot)/ X_0(1)$ and $\N_{X_0}(\cdot) = \int \N_x(\cdot) dX_0(x)$. This implies that the $N$ summands in \eqref{e_clusterL21} are identically distributed; in particular, conditional on $N$ we define identically distributed random variables $e^{N,\lambda}_i \geq 0$, for $i=1,\ldots,N$, by
\begin{equation} \label{e_edef}
e^{N,\lambda}_i = \big[ L^{i,\lambda}_t(e^{-\lambda Z^i_N(t,\cdot)} \cdot \phi) - L^i_t(1(Z^i_N(t,\cdot)=0) \cdot \phi) \big]^2.
\end{equation}
By \eqref{clusterconditionprob}, $e^{N,\lambda}_i$ converges to $0$ in probability as $\lambda \to \infty$ when conditioned on $(x_1,\ldots,x_N)$. However, one can integrate the conditional probabilities over $(x_1,\ldots,x_N) \in \R^N$ to determine that
\begin{equation} \label{e_eprob}
e^{N,\lambda}_i \to 0\, \text{ in probability under } P^X_{X_0}(\cdot \, | \, N) \text{ as } \lambda \to \infty.
\end{equation}
It is clear from \eqref{e_edef} that for all $\lambda>0$,
\begin{equation} \label{e_ebd}
e^{N,\lambda}_i \leq 2\|\phi\|_\infty^2 (L^{i,\lambda}_t(1)^2 + L^i_t(1)^2) \,\,\,\, \forall \,	 i =1,\ldots,N, \, \forall \,N \geq 1.
\end{equation}
By Theorem~\ref{thm_Lt} for $\N_0$, $L^{i,\lambda}_t(1)^2 \to L^i_t(1)^2$ in probability under $\N_{\bar{X}_0}(X_t \in \cdot\, | \, X_t > 0)$ and hence under $P^X_{X_0}(\cdot\, | \, N)$. Furthermore, since $L^{i,\lambda}_t(1) \to L^i_t(1)$ in $\cL^2(\N_{\bar{X}_0}(\cdot\, | \, X_t>0))$ (by Theorem~\ref{thm_Lt} for $\N_0$), it follows from Cauchy-Schwarz that $L^{i,\lambda}_t(1)^2 \to L^i_t(1)^2$ in $\cL^1(\N_{\bar{X}_0}(\cdot\,| \, X_t>0))$; since $X^i_t$ has distribution $\N_{\bar{X}_0}(X_t \in \cdot \, | \, X_t>0)$ under $P^X_{X_0}(\cdot\, | \, N)$, this implies $L^{i,\lambda}_t(1)^2 \to L^i_t(1)^2$ in $\cL^1(P^X_{X_0}(\cdot\, | \, N))$. 
Hence $\{2\|\phi\|_\infty(L^{i,\lambda}_t(1)^2 + L^i_t(1)^2): \lambda \geq 1 \}$ is uniformly integrable. Thus by \eqref{e_ebd}, $\{ e^{N,\lambda}_i : \lambda \geq 1\}$ is uniformly integrable, and by \eqref{e_eprob} we have $\cL^1$ convergence. That is,
\begin{equation} \label{e_eL1}
E^X_{X_0}(e^{N,\lambda}_i \, | \,N) \to 0 \, \text{ as } \lambda \to \infty.
\end{equation}
Conditioning on $N=n$ and summing over $n \in \N$, by \eqref{e_clusterL21} and Fubini's Theorem we have
\begin{align}
E^X_{X_0}((L^\lambda_t(\phi) - \L_t(\phi))^2) &\leq \sum_{n=1 }^\infty P^X_{X_0}(N=n)\, n \sum_{i=1}^n E^X_{X_0}(e^{n,\lambda}_i \, | \, N=n). \nonumber
\end{align}
Since $E^X_{X_0}(e^{N,\lambda}_i \, | \,N) \leq 2\|\phi\|_\infty E^X_{X_0}(L^{i,\lambda}_t(1)^2 + L_t(1)^2) \leq C(t,\phi)$ for all $\lambda \geq 1$, for some constant $C(t,\phi)>0$ (by uniform integrability), the $n$th term in the above is bounded above by $C(t,\phi) \,P^X_{X_0}(N=n) n^2$. Dominated Convergence therefore allows us to exchange limit and summation in the above, which by \eqref{e_eL1} gives the result. \qed \\

\emph{Proof of Theorem~\ref{thm_clustergeneral}.} This is virtually identical to the above proof of Theorem~\ref{thm_clusterLt} and is omitted.\\ 

As we have commented on, the expression in Theorem~\ref{thm_l2limit}, which is the same as \eqref{e_Ltcanon_secondmoment} in Theorem~\ref{thm_Ltcanonmoments}(b), is finite for all bounded $h$, despite the appearance of non-integrability (since $\lambda_0 > 1/2$). Proposition~\ref{prop_Ltcanonmomentsbd}, which we restate here for convenience, provides a useful upper bound on second moments which is our main tool for studying $L_t$. The bound is not difficult to obtain. Its derivation relies only on applying trivial upper bounds to several terms and making a few changes of variables. Recall that $E^Y_{z}$ denotes the expectation of a standard Ornstein-Uhlenbeck process $Y$ with $Y_0 = z$.\\

{\bf Proposition~\ref{prop_Ltcanonmomentsbd}.
\it For measurable, non-negative function $h: \R^2 \to \R$,
\begin{align*} 
\N_0 ( (L_t \times  L_t)(h) )  &\leq C_{\ref{thm_l2limit}}^2 \int_0^t w^{-2\lambda_0} \bigg[ \iint E^Y_{z_1} \bigg(  \exp \bigg( - \int_{0}^{\log (t/w)} F(Y_u) \, du \bigg)  \nonumber
\\ & \hspace{4 mm} \times h(\sqrt t Y_{\log (t/w)}, \sqrt t Y_{\log (t/w)}+ \sqrt{w}(z_2 - z_1)) \bigg) \psi_0(z_1) \, \psi_0(z_2) \, dm(z_1) \, dm(z_2)  \bigg] dw. \tag{\ref{e_Ltcanon_secondmomentbd}} 
\end{align*} 
Moreover,
\begin{equation*} 
\N_0 ( L_t(1)^2 ) \leq \frac{C_{\ref{thm_l2limit}}^2 \theta^2}{1-\lambda_0} t^{1-2\lambda_0}. \tag{\ref{e_Ltcanon_secondmomentmassbd}}
\end{equation*}}

\emph{Proof.} Let $h:\R^2 \to \R$ be Borel measurable and non-negative. We use the formula for $\N_0 ( (L_t \times L_t)(h))$ given by \eqref{e_Ltcanon_secondmoment}. We recall that $\rho(z_1,z_2) \leq 1$ and use this bound, and we bound above by using $V_u^{\infty,\infty}(x,y) \geq V_u^\infty(x)$ in the exponential. This gives
\begin{align}
\N_0 ((L_t \times L_t)(h) ) &\leq C_{\ref{thm_l2limit}}^2 \int_0^t (t-s)^{-2\lambda_0} \bigg[ \iint E^{B}_0 \bigg(  \exp \left( - \int_0^s V^{ \infty}_{t-u} (\sqrt{t-s} \, z_1 + B_s - B_u) \, du \right) \bigg)\nonumber
\\ & \hspace{4 mm}\times h(\sqrt{t-s}\, z_1 + B_s, \sqrt{t-s}\, z_2 + B_s)\, \psi_0(z_1) \, \psi_0(z_2) \, dm(z_1) \, dm(z_2)  \bigg] ds. \nonumber
\end{align}
Since $z_1 \sim m$, $\sqrt{ t- s}\, z_1$ has a normal distribution with variance $t-s$, and we interpret it as the Brownian increment $B_t - B_s$. Hence the above is equal to
\begin{align}
&C_{\ref{thm_l2limit}}^2 \int_0^t (t-s)^{-2\lambda_0} \bigg[\int E^{B}_0 \bigg(  \exp \bigg( - \int_0^s V^{\infty}_{t-u} (B_t- B_u) \, du \bigg)  \times h(B_t, \sqrt{t-s}\, z_2 + B_s)\nonumber
\\ & \hspace{20 mm} \,\psi_0 \bigg(\frac{B_t - B_s}{\sqrt{t-s}}\bigg) \bigg)   \psi_0(z_2)\, dm(z_2)  \bigg] ds \nonumber
\\ &\hspace{4 mm}= C_{\ref{thm_l2limit}}^2 \theta \int_0^t w^{-2\lambda_0} \bigg[ \int \,E^{W}_0 \bigg(  \exp \left( - \int_{w}^t V^{\infty}_{u} (W_{u}) \,du \right) h(W_t, \sqrt w \, z_2 + W_t - W_{w}) \nonumber 
\\ &\hspace{20 mm}  \times \psi_0 \left(\frac{W_w}{\sqrt{w}}\right)\,\psi_0(z_2)\, dm(z_2) \bigg) \bigg] dw, \nonumber
\end{align}
where in the second line we have used $w = t-s$ and defined $W_u = B_t - B_{t-u}$. Hence $W_u$ is a standard Brownian motion under $P^W_0$. Recall that $V_u^\infty(x) = u^{-1} F(u^{-1/2}x)$. Applying this and letting $u = e^r$ in the integral, we obtain that the above is equal to
\begin{align}
&C_{\ref{thm_l2limit}}^2 \theta \int_0^t w^{-2\lambda_0} \bigg[ \int \,E^{W}_0 \bigg(  \exp \bigg( - \int_{\log w}^{\log t} F(e^{-r/2 }W_{e^r}) \,dr \bigg) h(W_t, \sqrt w \, z_2 + W_t - W_{w}) \nonumber 
\\ &\hspace{20 mm}  \times \psi_0 \bigg(\frac{W_w}{\sqrt{w}}\bigg)\,\psi_0(z_2)\, dm(z_2) \bigg) \bigg] dw. \nonumber
\end{align}
We now define a stationary Ornstein-Uhlenbeck process $Y$ (with stationary measure $m$) by $Y_r = e^{-r/2 }W_{e^r}$ for $r \in \R$. Recall that we denote its law by $E^Y$. The above is therefore equal to
\begin{align}
&C_{\ref{thm_l2limit}}^2 \theta \int_0^t w^{-2\lambda_0} \bigg[ \int \,E^Y \bigg(  \exp \bigg( - \int_{\log w}^{\log t} F(Y_r) \,du \bigg) h(\sqrt t\, Y_{\log t}, \sqrt w \, z_2 + \sqrt t \, Y_{\log t} - \sqrt w \, Y_{\log w}) \nonumber 
\\ &\hspace{20 mm}  \times \psi_0 (Y_{\log w})\,\psi_0(z_2)\, dm(z_2) \bigg) \bigg] dw. \nonumber
\end{align}
By stationarity of $Y$, we can shift time by $\log w$ to obtain that the above is equal to
\begin{align}
&C_{\ref{thm_l2limit}}^2 \theta \int_0^t w^{-2\lambda_0} \bigg[ \int \,E^Y \bigg(  \exp \bigg( - \int_{0}^{\log (t/w)} F(Y_r) \,du \bigg) h(\sqrt t\, Y_{\log (t/w)}, \sqrt w \, z_2 + \sqrt t \, Y_{\log (t/w)} - \sqrt w \, Y_{0}) \nonumber 
\\ &\hspace{20 mm}  \times \psi_0 (Y_{0})\,\psi_0(z_2)\, dm(z_2) \bigg) \bigg] dw. \nonumber
\end{align}
$Y_0$ has distribution $m$, so we condition on the value of $Y_0$ and call it $z_1$. This gives the desired expression and proves that \eqref{e_Ltcanon_secondmomentbd} holds. The proof of \eqref{e_Ltcanon_secondmomentmassbd} is a consequence of the following lemma.
\begin{lemma} \label{lemma_integsurvivalprob} For $t>0$,
\[\int P_z^Y(\rho^F > t)\, \psi_0(z) \,dm(z) = \theta e^{-\lambda_0 t}.\]
\end{lemma}
Returning to \eqref{e_Ltcanon_secondmomentmassbd}, we apply \eqref{e_Ltcanon_secondmomentbd} with $h=1$. Separating the integrals, we obtain that
\begin{align}
\N_0(L_t(1)^2) \leq C_{\ref{thm_l2limit}}^2 \,\theta \int_0^t w^{-2\lambda_0}  \bigg( \int  P^Y_{z} \big(\rho^F > \log(t/w) \big) \,\psi_0(z) \, dm(z) \bigg) dw, \nonumber
\end{align}
where we have used $\int \psi_0 dm = \theta$. The inequality \eqref{e_Ltcanon_secondmomentmassbd} now readily follows from Lemma~\ref{lemma_integsurvivalprob}, which completes the proof of Proposition~\ref{prop_Ltcanonmomentsbd}. \qed \\

\emph{Proof of Lemma~\ref{lemma_integsurvivalprob}.}
Expanding in terms of the transition densities, we have
\begin{align} \label{integprob1}
\int  P^Y_{z} \big(\rho^F > t \big) \,\psi_0(z) \, dm(z) &=  \int \left( \int q_{t}(z,y)\, dm(y) \right) \psi_0(z) \, dm(z) \nonumber
\\ &= \langle q_t, 1 \otimes \psi_0 \rangle_{\cL^2(m \times m)},
\end{align}
where $\langle \cdot\,, \cdot \rangle_{\cL^2(m \times m)}$ denotes the inner product on $\cL^2(m \times m)$ and $\otimes$ is the tensor product of functions. Recall from that Theorem~\ref{thm_killedOU}(a) that the eigenfunction expansion \eqref{OU_eigexp} converges in $\cL^2(m \times m)$ to $q_t(\cdot,\cdot)$, and that $\|\psi_0\|_{\cL^2(m)} = 1$. Thus by the above and Fubini's theorem, \eqref{integprob1} is equal to
\begin{align}
\sum_{n=0}^\infty  e^{-\lambda_n t}  \langle\psi_n \otimes \psi_n, 1 \otimes \psi_0 \rangle_{\cL^2(m\times m)}  \nonumber
&= e^{-\lambda_0 t}  \langle\psi_0 \otimes \psi_0, 1 \otimes \psi_0 \rangle_{\cL^2(m\times m)} 
\\ &= e^{-\lambda_0 t} \int \psi_0^2 \, dm \int \psi_0 \,dm = \theta e^{-\lambda_0 t}, \nonumber
\end{align}
where the first equality follows from orthogonality of the eigenfunctions, which implies that $\int \psi_n \psi_0 \,dm = 0$ for all $n \geq 1$, and the last line has used $\int \psi_0 \,dm = \theta$ and $\int \psi_0^2\, dm = 1$. \qed \\

We now use the bounds in Proposition~\ref{prop_Ltcanonmomentsbd} to derive prove the remaining properties of $L_t$. \\

\emph{Proof of Theorem \ref{thm_Ltprop}(a).}  Via the second moment method, we have
\begin{align}
\N_0 \left( L_t(1) > 0 \right) \geq \frac{ \N_0 (L_t(1))^2}{\N_0(L_t(1)^2)} \geq  \frac{C_{\ref{thm_l2limit}}^2 \theta^2 t^{-2\lambda_0}}{C_{\ref{thm_l2limit}}^2 \theta^2 t^{1-2\lambda_0}(1-\lambda_0)^{-1}} = \frac{1-\lambda_0}{t}, \nonumber
\end{align}
where we recall that $\int \psi_0 \, dm = \theta$ and we have used Theorem~\ref{thm_Ltcanonmoments}(a) and \eqref{e_Ltcanon_secondmomentmassbd}. We recall that $\N_0 \left( X_t > 0 \right) = 2/t$, which implies that  $\N_0 \left( L_t > 0\, \big| X_t >0 \right) \geq \frac{1-\lambda_0}{2}$. \qed \\

\emph{Proof of Theorem~\ref{thm_Ltdim}.}
Recall that for $p>0$, $h_p(x,y) = |x-y|^{-p}$. We first establish that
\begin{equation} \label{e_canonenergy}
\N_0 ( (L_t \times L_t)(h_p)) < \infty
\end{equation}
for all $p< 2-2\lambda_0$. Applying \eqref{e_Ltcanon_secondmomentbd} with $h_p$, we have
\begin{align}
\N_0 ( (L_t \times  L_t)(h_p))  \leq & \, C_{\ref{thm_l2limit}}^2 \int_0^t w^{-2\lambda_0} \bigg[ \iint E^Y_{z_1} \bigg(  \exp \bigg( - \int_{0}^{\log (t/w)} F(Y_u) \, du \bigg) \bigg) \nonumber
\\ & \hspace{8 mm}\times |\sqrt{w}(z_2 - z_1)|^{-p} \, \psi_0(z_1) \, \psi_0(z_2) \, dm(z_1) \, dm(z_2)  \bigg] dw \nonumber
\\ =&\, C_{\ref{thm_l2limit}}^2 \int_0^t w^{-2\lambda_0-p/2} \bigg[ \iint E^Y_{z_1} \bigg(  \exp \bigg( - \int_{0}^{\log (t/w)} F(Y_u) \, du \bigg)\bigg)  \nonumber
\\ & \hspace{8 mm}\times |z_1 - z_1|^{-p}  \psi_0(z_1) \, \psi_0(z_2) \, dm(z_1) \, dm(z_2)  \bigg] dw. \nonumber
\end{align} 
Recalling \eqref{e_survivalprob}, the expectation is equal to the survival probability $P^Y_{z_1}(\rho^F > \log(t/w))$, so the above equals
\begin{align}
C_{\ref{thm_l2limit}}^2 \int_0^t w^{-2\lambda_0 -p/2} \bigg[ \iint P^{Y}_{z_1} \big( \rho^F > \log(t/w) \big)\,  \big| z_1 - z_2 \big|^{-p} \, \psi_0(z_1)\, \psi_0(z_2) \,dm(z_1)\, dm(z_2)  \bigg] dw. \nonumber
\end{align}
Applying \eqref{OU_psi0bd} and \eqref{OU_survivalprobGaus}, both with $\delta = 1/8$, this is bounded above by
\begin{align}
&C \int_0^t w^{-2\lambda_0 -p/2} \bigg[ \iint   \big| z_1 - z_2 \big|^{-p} \,t^{-\lambda_0} \,w^{\lambda_0}\,e^{z_1^2 / 4}\, e^{z_2^2 /8} \,dm(z_1)\, dm(z_2)  \bigg] dw \nonumber
\\ &\hspace{5 mm}=  C(p)\, t^{-\lambda_0} \int_0^t w^{-\lambda_0 -p/2}dw. \nonumber
\end{align}
The second line follows because the integrand has Gaussian tails in $z_1$ and $z_2$ and $p < 2-2\lambda_0 < 1$. Finally, the integral in the final line is finite because $-\lambda_0 - p/2 > -\lambda_0 - \lambda_0 + 1 > -1$, which proves \eqref{e_canonenergy}. In fact, we have shown that
\begin{equation} \label{e_canonenergybd}
\N_0 ( (L_t \times L_t)(h_p)) \leq C(p) t^{1-2\lambda_0 - p/2}.
\end{equation} \\
Next, we establish the same under $P^X_{X_0}$. That is, we will show that
\begin{equation} \label{e_PXenergybd}
E^X_{X_0} ( (L_t \times L_t)(h_p)) < \infty
\end{equation}
for $p<2-2\lambda_0$. We use the cluster decomposition and argue conditionally as in the proof of Theorem~\ref{thm_Lt} (for $P^X_{X_0}$) above. Suppose that $P^X_{X_0}$ is a probability under which $X_t$ is realized as in \eqref{e_clusterrep}. Conditioning on $N,x_1, \hdots, x_N$, by \eqref{e_clusterdecomp} we have
\begin{equation}
dL_t(x) \leq \sum_{i=1}^N dL_t^i(x). \nonumber
\end{equation}
Thus we obtain that
\begin{align} \label{e_PXenergy1}
&\iint |x-y|^{-p} dL_t(x)\,dL_t(y)  \nonumber
\\ &\hspace{5 mm}\leq \iint |x-y|^{-p} \bigg( \sum_{i=1}^N dL_t^i (x)\bigg) \bigg( \sum_{i=j}^N dL_t^j (y)\bigg) \nonumber
\\ &\hspace{5 mm}= \sum_{i=1}^N \iint |x-y|^{-p} dL_t^i(x)\,dL_t^i(y) + \sum_{i=1}^N \sum_{j\neq i} \iint |x-y|^{-p} dL_t^i(x)\,dL_t^j(y).
\end{align}
Recall that the $X_t^i$ are independent with distributions $\N_{x_i} (X_t \in \cdot \,|\, X_t>0 )$. By \eqref{e_Xtsurvive} and \eqref{e_canonenergybd}, we therefore have
\begin{equation}\label{e_PXenergy2}
 \N_{x_i} \left( \iint |x-y|^{-p} dL_t^i(x)\,dL_t^i(y) \, \bigg| \, X_t^i > 0\right) = C(p) t^{1-2\lambda_0 - p/2}\, (2/t)^{-1} =: C_1(p) \,t^{2-2\lambda_0 - p/2},
\end{equation}
which provides a bound for the summands in the first term of \eqref{e_PXenergy1}. We now consider the mixed integrals in \eqref{e_PXenergy1}, that is, the summands in the second term. Without loss of generality, let $i=1$ and $j=2$, and denote their (independent) laws $\N^1_{x_1}(X_t^1 \in \cdot \,|\,X^1_t>0),\N^2_{x_2}(X_t^2 \in \cdot \,|\,X^2_t>0)$. Because the integrands are non-negative, we can change the order of integration and obtain
\begin{align}\label{e_PXenergy3}
&\N^1_{x_1} \otimes \N^2_{x_2} \left(\iint |x-y|^{-p} dL_t^1(x)\,dL_t^2(y) \, \bigg| \, X_t^1 >0, X_t^2 > 0\,  \right)  \nonumber
\\ &\hspace{5 mm}= \N^1_{x_1} \left( \int \N^2_{x_2} \left( \int |x-y|^{-p} dL_t^2(y) \, \bigg| \, X_t^2 > 0 \right) dL_t^1(x) \, \bigg| \, X_t^1 >0 \right) 
\end{align}
To compute the inner expectation we apply translation invariance and \eqref{e_Xtsurvive}, which gives
\begin{align}
&\N^2_{x_2} \left( \int |y-x|^{-p} dL_t^2(y) \, \bigg| \, X^2_t>0\right)  \nonumber
\\ &\hspace{4 mm}= (t/2) \,\N_0 \left( \int |y-x|^{-p} dL_t(y-x_2) \right) \nonumber
\\ &\hspace{4 mm}= (t/2) \,\N_0 \left( \int |y-x+x_2|^{-p} dL_t(y) \right) \nonumber
\\ &\hspace{4 mm}= C_{\ref{thm_l2limit}} (t/2) t^{-\lambda_0} \int |\sqrt t z-(x-x_2)|^{-p} \psi_0(z) \, dm(z), \nonumber
\end{align}
where the last line follows from the mean measure formula \eqref{e_Ltcanon_firstmoment}. By \eqref{OU_psi0bd} with $\delta = 1/4$, we have that $\psi_0(z_2) \, dm(z_2) \leq c\, e^{-z_2^2/4} dz_2$. Thus the above is bounded above by
\begin{align}
&C t^{1-\lambda_0} \int (|\sqrt t z-(x-x_2)|^{-p}\vee 1) e^{-z^2/4} \,dz \nonumber
\\ &\hspace{5 mm}= C t^{1-\lambda_0} \int (|w-(x-x_2)|^{-p}\vee 1) t^{-1/2} e^{-w^2/4t} \,dw \nonumber
\\ &\hspace{5 mm}\leq C\,t^{1-\lambda_0} t^{-1/2} \int |w-(x-x_2)|^{-p}\, 1_{|w-(x-x_2)| \leq 1}\,dw + C\,t^{1-\lambda_0} \int t^{-1/2} e^{-w^2/4t}\,dw \nonumber
\\ &\hspace{5 mm}= C'(p)t^{1/2-\lambda_0}+ Ct^{1-\lambda_0} < \infty. \nonumber
\end{align}
By the above bound and another application of \eqref{e_Ltcanon_firstmoment}, \eqref{e_PXenergy3} is bounded above by
\begin{align} \label{e_PXenergy4}
& \left[C'(p)t^{1/2-\lambda_0}+ Ct^{1-\lambda_0} \right]\N^1_{x_1} ( L_t^1(1) \, | \, X_t^1 >0 ) =: C_2(p) \left[ t^{3/2-2\lambda_0} + t^{2-2\lambda_0}\right]. 
\end{align}
We note that both \eqref{e_PXenergy2} and \eqref{e_PXenergy4} are independent of the points $x_1, \hdots, x_N$. Therefore by these bounds and \eqref{e_PXenergy1} we have shown that
\begin{equation}
E^X_{X_0} ( (L_t \times L_t)(h_p) \,| \, N ) \leq C_1(p) N t^{2-2\lambda_0 - p/2} + C_2(p) (N^2 - N) \left[t^{3/2-2\lambda_0} +t^{2-2\lambda_0}\right].  \nonumber
\end{equation}
Taking the expectation above with respect to $N$, which we recall is Poisson with mean $2X_0(1) / t$, gives
\begin{equation} \label{e_PXenergy5}
E^X_{X_0} ( (L_t \times L_t)(h_p)) \leq C_1(p) X_0(1)\, t^{1-2\lambda_0 - p/2} + C_2(p) X_0(1)^2 \left[t^{-1/2-2\lambda_0} + t^{-2\lambda_0}\right] < \infty,
\end{equation}
which proves \eqref{e_PXenergybd}. \\

Under both $P^X_{X_0}$ and $\N_0$, we have shown that the $p$-energy of $L_t$ has finite expectation, and hence $L_t$ has finite $p$-energy almost surely, for all $p < 2-2\lambda_0$. By the energy method (see, for example, Theorem 4.27 of M{\"o}rters and Peres \cite{MP10}), this implies that $\text{dim}(BZ_t) \geq 2-2\lambda_0$ a.s. on $\{L_t > 0\}$ under $P^X_{X_0}$ and $\N_0$. Combined with Theorem A, this completes the proof of Theorem~\ref{thm_Ltdim} for $P^X_{X_0}$. To see that the upper bound on the dimension holds for $\N_0$ follows from the cluster decomposition. Consider $X_t$ under $P^X_{\delta_0}$. In the cluster decomposition of $X_t$, the probability that $N=1$ is positive. Conditioning on this event, $X_t$ is equal to $X_t^1$, which has law $\N_0( X_t^1 \in \cdot \, | \, X_t > 0 )$. Because $\text{dim}(BZ_t) \leq 2-2\lambda_0$ a.s. on this event, we therefore have $\N_0 \big( \{ \text{dim}(BZ_t) \leq 2-2\lambda_0\} \, \big| \, X_t > 0 \big) = 1$. This completes the proof. \qed \\

\emph{Proof of Theorem~\ref{thm_Ltmoments}.} To see part (a), we note that \eqref{canonmoment1} gives an expression for $\lim_{\lambda \to \infty} E^X_{X_0}(L^\lambda_t(\phi))$. On the subsequence $\{\lambda_n\}_{n=1}^\infty$ from Theorem~\ref{thm_Lt}, $L^{\lambda_n}_t(\phi) \to L_t(\phi)$ a.s. for bounded and continuous $\phi$, so it is enough to show that $
\lim_{n \to \infty} E^X_{X_0}(L^{\lambda_n}_t(\phi)) =  E^X_{X_0}(\lim_{n \to \infty} L^{\lambda_n}_t(\phi))$. By Theorem~\ref{thm_Lt}, $L_t^\lambda(\phi)$ converges in and hence is bounded in $\cL^2(P^X_{X_0})$. It is therefore uniformly integrable, which justifies the above exchange of limit and integration. This proves the result for bounded and continuous $\phi$. We extend the moment formula to bounded measurable functions by a Monotone Class Lemma and to non-negative measurable functions by Monotone convergence.\\

We now prove part (b). Suppose we realize $X_t$ under a probability $P^X_{X_0}$ such that \eqref{e_clusterrep} holds. Conditionally on $N$, by \eqref{e_clusterdecomp} we have 
\begin{align}
L_t(1)^2 &\leq  \bigg( \sum_{i=1}^N L_t^i(1) \bigg)^2  = \sum_{i=1}^N L_t^i(1)^2 + \sum_{i=1}^N \sum_{j\neq i} L_t^i(1) L_t^j(1). \nonumber
\end{align}
The clusters are independent with laws $\N_{\bar{X}_0}(X_t^i \in \cdot \, | \, X_t^i > 0) = (t/2) \N_{\bar{X}_0} ( \{X_t^i >0 , X_t^i \in \cdot\})$, the equality by \eqref{e_Xtsurvive}. Thus, applying Theorem~\ref{thm_Ltcanonmoments}(a) and Proposition~\ref{prop_Ltcanonmomentsbd}(b) to the above and using independence, we obtain
\begin{align}
E^X_{X_0} (L_t(1)^2 \, | \, N) \leq C N(t/2) t^{1-2\lambda_0} + C(N^2 - N) (t/2)^2 t^{-2\lambda_0}.
\end{align}
As in the proof of Theorem~\ref{thm_Ltdim}, we take the expectation with respect to $N$, which has a $\text{Poisson}(2X_0(1)/t)$ distribution. This proves part (b). \qed \\

Finally we consider the atomless property of $L_t$ (see Theorem~\ref{thm_Ltprop}(b)). Once again we carry out the necessary moment calculations under canonical measure. $L_t$ has an atom of mass $c>0$ at $x$ if $L_t(\{x\}) = c$. We decompose $L_t$ as 
\begin{equation} \label{atomdecomp}
L_t = \tilde{L}_t + \nu_t, \nonumber
\end{equation}
where $\tilde{L}_t$ is atomless and $\nu_t$ is strictly atomic. We begin with an elementary observation which provides an upper bound for the mass of the atoms of a measure. Let $M \in \N$. Let $I^n_1 = [-M, -M+ 2^{-n}]$, and for $k=2,3 \hdots, 2M2^n$, define the dyadic interval $I^n_k = (-M + (k-1)2^{-n},-M + k2^{-n}]$. Then $\{ I^n_k : k \leq 2M2^n\}$ is a partition of $[-M,M]$ into disjoint intervals of length $2^{-n}$. The following lemma is elementary.
\begin{lemma} \label{lemma_squareatoms}
Fix $M \in \N$ and suppose that $\mu$ is a finite measure supported on $[-M,M]$ with decomposition $\mu = \rho + \nu$, where $\rho$ is atomless and $\nu = \sum_{i \in I} c_i \delta_{x_i}$ is strictly atomic. Then for every $n\geq 1$,
\[ \sum_{k=1}^{2M2^n} \mu(I^k_n)^2  \geq  \sum_{i \in I} c_i^2. \]
\end{lemma}

The next lemma gives an upper bound for the second moment of $L_t$ on a ball. We denote by $B(x,r)$ the ball of radius $r>0$ centred at $x \in \R$. We recall $s^*(\delta)$ from Theorem \ref{thm_killedOU}(c); in what follows we use $\delta = 1/8$, and $s^*$ denotes $s^*(1/8)$.
\begin{lemma} \label{lemma_ball2moment}
There is a constant $C_{\ref{lemma_ball2moment}}>0$ and $t$-dependent constant $C_{\ref{lemma_ball2moment}}(t)>0$ such that for all $x \in \R$ and $r<e^{-s^*}t$, 
\begin{align} 
\N_0 (L_t(B(x,r))^2 ) &\leq C_{\ref{lemma_ball2moment}} \left[t^{-\lambda_0} \, r^{2-2\lambda_0}\, P_0^W (W_{4t/3} \in B(x,r) )  + t^{-3\lambda_0+ 1/2} \, r\,P_0^W ( W_t \in B(x,r) ) \right] \nonumber
\\ &\leq C_{\ref{lemma_ball2moment}}(t) \left[ r^{3-2\lambda_0} + r^2 \right], \nonumber
\end{align}
where $W$ is a standard Brownian motion under $P^W_0$.
\end{lemma}
\begin{proof}
We apply \eqref{e_Ltcanon_secondmomentbd} with $h(z_1,z_2) = 1_{B(x,r)}(z_1)\,1_{B(x,r)}(z_2)$. This gives
\begin{align} \label{e_ball2moment1}
&\N_0 (L_t(B(x,r))^2 ) \nonumber
\\ &\hspace{5 mm}= C \int_0^t w^{-2\lambda_0} \bigg[ \iint E^Y_{z_1} \bigg(  \exp \bigg( - \int_{0}^{\log (t/w)} F(Y_u) \, du \bigg) \nonumber
\\ &\hspace{9 mm} \times  1_{B(x,r)}(\sqrt t Y_{\log (t/w)}) \, 1_{B(x,r)}(\sqrt t Y_{\log (t/w)}+ \sqrt{w}(z_2 - z_1))\bigg)  \psi_0(z_1) \, \psi_0(z_2) \, dm(z_1) \, dm(z_2)  \bigg] dw. 
\end{align}
We now divide the above into two cases depending on the size of $w$. We first consider the singular case, where $w$ is small.\\

\textbf{Case 1: $w < e^{-s^*}t$.}\\

We interpret the exponential in (\ref{e_ball2moment1}) as the probability that $Y$ survives until time $\log(t/w)$ when it is subject to Markovian killing with rate $F(Y_u)$. Because this probability is equal to the integral of the transition density over all of $\R$, the portion of the integral corresponding to $w \in [0,e^{-s^*}t]$ equals
\begin{align} \label{e_ball2moment2}
&C \int_0^{e^{-s^*}t} w^{-2\lambda_0} \bigg[ \iiint q_{\log(t/w)}(z_1,y) \, 1_{B(x,r)}(\sqrt t  y) \, 1_{B(x,r)}(\sqrt t y+ \sqrt{w}(z_2 - z_1))  \nonumber
\\ &\times  \psi_0(z_1) \, \psi_0(z_2) \, dm(z_1) \, dm(z_2) \,dm(y) \bigg] dw.  \nonumber
\\ &\hspace{6 mm}\leq C \int_0^{e^{-s^*}t} w^{-2\lambda_0} \bigg[ \iiint e^{-\lambda_0 \log(t/w)} \, e^{z_1^2/8} e^{y^2/8} \, 1_{B(x,r)}(\sqrt t  y) \, 1_{B(x,r)}(\sqrt t y+ \sqrt{w}(z_2 - z_1))  \nonumber
\\ & \hspace{10 mm}\times  \, \psi_0(z_1) \, \psi_0(z_2) \, dm(z_1)  dm(z_2)\,dm(y)  \bigg] dw.  \nonumber
\\ &\hspace{6 mm}\leq Ct^{-\lambda_0}\int_0^{e^{-s^*}t} w^{-\lambda_0} \int e^{y^2/8} 1_{B(x,r)}(\sqrt t y)\,  \nonumber
\\ & \hspace{10 mm} \bigg[\iint e^{z_1^2/4} e^{z_2^2/8} \, 1_{B(x,r)}(\sqrt t y+ \sqrt{w}(z_2 - z_1)) \,dm(z_1) \, dm(z_2)  \bigg] dm(y) \, dw.  
\end{align}
The first inequality uses \eqref{OU_qbound} with $\delta = 1/8$, which applies because $\log(t/w) > s^*$ for all $w$ in the above integral, and the second uses \eqref{OU_psi0bd}, both with $\delta = 1/8$. In the integral in the last line we collect all the Gaussian terms. The square-bracketed term is equal to
\begin{align}
 &C\iint 1_{B(x,r)}(\sqrt t y+ \sqrt{w}(z_2 - z_1))\, e^{-z_1^2/4}  e^{-3z_2^2/8} \, dz_1 \, dz_2 \nonumber
 \\ &\hspace{5 mm}=C'\int 1_{B(x,r)}(\sqrt t y+ \sqrt{w} z)\, e^{-3z^2/20} \, dz. \nonumber
\end{align}
We have used the convolution property for independent Gaussians. We define Gaussian random variables $g_1 \sim \cN(0,4t/3)$ and $g_2 \sim \cN(0,10/3)$. Substituting the last expression into (\ref{e_ball2moment2}), we obtain
\begin{align} \label{e_ball2moment22}
&Ct^{-\lambda_0}\int_0^{e^{-s^*}t} w^{-\lambda_0}\bigg[\iint 1_{B(x,r)}(\sqrt t y) \,1_{B(x,r)}(\sqrt t y+ \sqrt{w} z)\, e^{-3z^2/20} \, e^{-3y^2/8}  dz\, dy \bigg] dw \nonumber
\\ &\hspace{5 mm}=C' t^{-\lambda_0}\int_0^{e^{-s^*}t} w^{-\lambda_0}\bigg[ P\big(g_1 \in B(x,r), g_1 + \sqrt w g_2 \in B(x,r)\big) \bigg] dw \nonumber
\\ &\hspace{5 mm}\leq Ct^{-\lambda_0}\int_0^{e^{-s^*}t} w^{-\lambda_0}\bigg[ P\big(g_1 \in  B(x,r)\big) \, P\big(\sqrt w g_2 \in B(0,2r)\big) \bigg] dw \nonumber
\\ &\hspace{5 mm}= Ct^{-\lambda_0}  P \big(g_1 \in  B(x,r) \big) \int_0^{e^{-s^*}t} w^{-\lambda_0} P\big(g_2 \in B(0,2r w^{-1/2})\big) dw.
\end{align}
Suppose that $4r^2  < e^{-s^*}t$. If $2rw^{-1/2} > 1$, we bound the probability in the integral above by $1$. If $2rw^{-1/2} \leq 1$, the probability is simply bounded by the diameter of the ball, $4rw^{-1/2}$. Thus \eqref{e_ball2moment22}, and hence \eqref{e_ball2moment2}, is bounded above by 
\begin{align} \label{e_ball2moment3}
&Ct^{-\lambda_0}  P (g_1 \in  B(x,r) ) \bigg[  \int_0^{4r^2} w^{-\lambda_0}  dw  + 4r \int_{4r^2}^{e^{-s^*}t} w^{-\lambda_0-1/2}   dw  \bigg] \nonumber
\\ &\hspace{5 mm}= Ct^{-\lambda_0}  P (g_1 \in  B(x,r) ) \bigg[4^{1-\lambda_0} \frac{r^{2-2\lambda_0}}{1-\lambda_0} + \frac{4r}{\lambda_0 - 1/2}\left((4r^2)^{-(\lambda_0 - 1/2)} - (te^{-s^*})^{-(\lambda_0 - 1/2)} \bigg)  \right] \nonumber
\\ &\hspace{5 mm}\leq Ct^{-\lambda_0}  P (g_1 \in  B(x,r) ) \, r^{2-2\lambda_0}. 
\end{align}
Finally, note that if $4r^2 \geq e^{-s^*}t$, then \eqref{e_ball2moment22} is bounded above by
\[ Ct^{-\lambda_0}  P (g_1 \in  B(x,r) )  \int_0^{e^{-s^*}t} w^{-\lambda_0}  dw \leq Ct^{-\lambda_0}P (g_1 \in  B(x,r) )  (e^{-s^*}t)^{1-\lambda_0}  \leq Ct^{-\lambda_0} P (g_1 \in  B(x,r) ) r^{2-2\lambda_0}, \]
so the upper bound for \eqref{e_ball2moment2} obtained in \eqref{e_ball2moment3} holds in this case as well.\\
 
\textbf{Case 2: $w \in (e^{-s^*}t,t]$.}\\

In this case we simply bound the exponential term in (\ref{e_ball2moment1}) above by $1$, effectively ignoring the killing, in which case $Y_{\log(t/w)} \sim m$. We also use (\ref{OU_psi0bd}) with $\delta = \frac 1 4$. Hence the contribution to \eqref{e_ball2moment1} from the $w \in (e^{-s^*}t,t]$ case is bounded above by
\begin{align}
&C \int_{e^{-s^*}t}^t w^{-2\lambda_0} \bigg[ \iiint  1_{B(x,r)}(\sqrt ty) \, 1_{B(x,r)}(\sqrt t y+ \sqrt{w}(z_2 - z_1))\, e^{z_1^2/4}\,e^{z_2^2 /4}\, dm(z_1) \, dm(z_2)\, dm(y)  \bigg] dw \nonumber
\\&\hspace{5 mm}\leq C \int_{e^{-s^*}t}^t w^{-2\lambda_0} \bigg[ \iiint  1_{B(x,r)}(\sqrt ty) \, 1_{B(x,r)}(\sqrt t y+ \sqrt{w}(z_2 - z_1))\, e^{-z_1^2/4}\,e^{-z_2^2 /4}\, dz_1 \, dz_2 \, dm(y) \bigg] dw \nonumber
\\&\hspace{5 mm}= C \int_{e^{-s^*}t}^t w^{-2\lambda_0} \bigg[ \iint  1_{B(x,r)}(\sqrt ty) \, 1_{B(x,r)}(\sqrt t y+ \sqrt{w} z)\, e^{-z^2/2}\, dz\, dm(y) \bigg] dw \nonumber
\\ &\hspace{5 mm}\leq Ct^{-\lambda_0}  P (g_3 \in  B(x,r) ) \int_{e^{-s^*}t}^t w^{-2\lambda_0} P (g_4 \in B(0,2r w^{-1/2})) \, dw. \nonumber
\end{align}
In the above, $g_3 \sim \cN(0,t)$ and $g_4 \sim \cN(0,1)$. The third line follows by the convolution property of Gaussians. We again bound the probability in the integral by the size diameter of the ball, which gives the following upper bound for the above:
\begin{align} \label{e_ball2moment4}
&Ct^{-\lambda_0}  P (g_3 \in  B(x,r) )\, 4r  \int_{e^{-s^*}t}^t w^{-2\lambda_0 -1/2}  \, dw. \nonumber
\\ &\hspace{5 mm}\leq Ct^{-\lambda_0} P (g_3 \in  B(x,r) )\, 4r  \,(e^{-s^*}t)^{-2\lambda_0 +1/2} \nonumber
\\ &\hspace{5 mm}= Ct^{-3\lambda_0 + 1/2} P (g_3 \in  B(x,r) )\, r.
\end{align}
By combining (\ref{e_ball2moment3}) and (\ref{e_ball2moment4}) and interpreting the Gaussian probabilities in terms of Brownian motion, we obtain the first inequality of the result. The second bound is obtained by bounding the Brownian density above by its maximum value.
\end{proof}

\emph{Proof of Theorem~\ref{thm_Ltprop}(b).} First consider $L_t$ under $\N_0$ and recall the decomposition \eqref{atomdecomp}, ie. $L_t = \tilde{L}_t + \nu_t$, the latter strictly atomic. Fix $M \in \N$ and consider the restriction of $L_t$ to $[-M,M]$, ie. $dL_t^{(M)}(x) := 1_{[-M,M]}(x)\,dL_t(x)$, with decomposition $L_t^{(N)} = \tilde{L}_t^{(M)} + \nu_t^{(M)}$. Note that the radius of the dyadic intervals is $r(I^k_n) = r = 2^{-(n+1)}$. By Lemma \ref{lemma_ball2moment}, we have
\begin{align}
\N_0 \left(\sum_{k=1}^{2M2^n} L^{(M)}_t(I^k_n)^2 \right) &= \sum_{k=1}^{2M2^n} \N_0\left( L^{(M)}_t(I^k_n)^2\right) \nonumber
\\ &\leq C(t)\, 2M2^n \left[(2^{-(n+1)})^{3-2\lambda_0} + (2^{-(n+1)})^{2} \right] \nonumber
\\ &\leq \,C(t)\, 2M \left[ (2^{-n})^{2 - 2\lambda_0} + 2^{-n} \right] \nonumber
\\ &\to 0 \, \text{ as } \, n \to \infty \nonumber
\end{align} 
because $2-2\lambda_0 > 0$. Moreover, by Lemma \ref{lemma_squareatoms}, the first expression is greater than or equal to the expectation (under $\N_0$) of the sum of the squares of the atoms of $L^{(M)}_t$. The above implies that this expectation must in fact be zero, so $\nu_t^{(M)} = 0$ $\N_0$-a.s. As this holds for all $M$, $\nu_t = 0$ and $L_t$ is atomless under $\N_0$. To obtain the result under $P^X_{X_0}$, we note from the cluster decomposition and \eqref{e_clusterdecomp} that (conditionally) $L_t$ is a sum of $N$ measures which are atomless by the above, and hence is atomless. \qed 

\section{Proof of Theorem \ref{thm_l2limit}} \label{s_momentsconvergence}
We begin by obtaining an expression for second moments of $L^\lambda_t$ under the canonical measure. In particular, we study $\N_0 (L^\lambda_t(\phi_1) \,L^{\lambda'}_t(\phi_2) )$ for $\lambda, \lambda' > 0$. The moment representation formula is naturally suggested by a branching particle heuristic. Its proof uses PDE methods and the Laplace functional. Let $E^B_x$ denote the expectation of a Brownian motion started at $x$. $E^{B^1, B^2}_{(x,y)}$ denotes the law of two independent Brownian motions $B^1$ and $B^2$ started from points $x$ and $y$ respectively. We recall that $p_\delta(\cdot)$ denotes the Gaussian density with variance $\delta$.
\begin{proposition} \label{prop_pde_rep1}
Let $h:\R^2 \to \R$ be a bounded Borel function and $\lambda, \lambda', t > 0$. Then
\begin{align*}
\N_0 ((L_t^\lambda \times L_t^{\lambda'})(h)) =  \,&(\lambda \lambda')^{2\lambda_0}\int_0^t E_0^B \bigg( E_{(0,0)}^{B^1, \,B^2}\bigg[ h( B_s + B^1_{t-s}, B_s + B^2_{t-s})  \nonumber
\\ &\times \exp \bigg( - \int_0^s V^{\lambda, \lambda'}_{t-u} (B^1_{t-s} + B_s - B_u, B^2_{t-s} + B_s - B_u) \, du \bigg) \nonumber
\\ &\times \exp \bigg( - \int_0^{t-s} V^{\lambda, \lambda'}_{r}(B^1_{r}, B^1_{r} + B^2_{t-s}- B^1_{t-s}) \, dr \bigg) 
\\ &\times \exp \bigg( - \int_0^{t-s} V^{\lambda, \lambda'}_{r}(B^2_{r}   + B^1_{t-s}- B^2_{t-s},B^2_{r} ) \, dr \bigg) \bigg] \bigg) ds.
\end{align*}
\end{proposition}
The proof of Proposition \ref{prop_pde_rep1} requires the following lemma.
\begin{lemma} \label{lemma_pde_canon_func}
Let $\varphi \in \cM_F(\R)$ and $\varphi_1, \varphi_2 \in \cL^1(\R)$ be non-negative and continuous. Then
\begin{align} \label{e_pde_canon_func}
\N_0 \left( X_t(\varphi_1) X_t(\varphi_2) e^{-X_t(\varphi)} \right) 
&=  \int_0^t E_0^B \bigg( \exp \left( - \int_0^s V_{t-u}^\varphi (B_u) \, ds \right) \nonumber
\\ &\,\,\,\,\,\,\times \prod_{i=1,2} E_0^{B^i} \bigg[ \exp \left( - \int_0^{t-s} V_{t-s-r}^\varphi(B_s + B^i_r) \, dr \right) \varphi_i(B_s + B^i_{t-s}) \bigg] \bigg) ds . \nonumber
\end{align}  
\end{lemma}

\begin{proof}
Let $\epsilon_1, \epsilon_2 > 0$ and $\varphi, \varphi_1$ and $\varphi_2$ be as in the statement. Viewing $\varphi_1$ and $ \varphi_2$ as the density functions of the finite measures they induce (ie. $\varphi_i(A) = \int_A \varphi_i(x)\,dx$), let $ V_t^{\varphi,\eps_1, \eps_2}$ denote the solution to \eqref{e_dualPDE} when $\phi = \varphi + \epsilon_1 \varphi_1 + \epsilon_2 \varphi_2 \in \cM_F(\R)$. By \eqref{e_LapFunCanon},
\[\N_0(1 - e^{-X_t(\varphi + \epsilon_1 \varphi_1 + \epsilon_2 \varphi_2)}) = V_t^{\varphi,\eps_1, \eps_2}(0).\]
We differentiate this expression once with respect to $\epsilon_1$ and once with respect to $\eps_2$. The derivatives of the inner expression of the left hand side are bounded above by integrable quantities (i.e. $X_t(\varphi_1)$ and $X_t(\varphi_1) X_t(\varphi_2)$) so we can take the differentiation inside the expectation in the probabilistic representation, and the derivatives of the right hand side exist. The resulting equation is the following:
\begin{equation} \label{e_2ptcanon_pde}
\N_0 \left(X_t(\varphi_1) \, X_t(\varphi_2)\, e^{-X_t(\varphi + \epsilon_1 \varphi_1 + \epsilon_2 \varphi_2)} \right) = - \frac{\partial^2}{\partial\epsilon_1 \partial \epsilon_2 }V_t^{\varphi,\eps_1, \eps_2}(0).
\end{equation}
We note that the limit of the left hand side as $\epsilon_1, \epsilon_2 \downarrow 0$ is the desired expression. We now obtain an expression for the first derivatives of $V_t^{\varphi,\eps_1, \eps_2}(0)$ with respect to $\epsilon_1$ and $\epsilon_2$. Consider the following partial differential equation:
\begin{equation} \label{e_Veps_firstderiv_PDE}
\frac{\partial u_t}{\partial t} = \frac{\Delta}{2} u_t - V_t^{\varphi, \eps_1, \eps_2} u_t \,\,\,\, \text{ for} \,(t,x) \in (0,\infty)\times \R, \hspace{6 mm} u_t \to \varphi_1 \,\, \text{ as } t \downarrow 0,
\end{equation}
where the $u_t \to \varphi_1$ in the sense of weak convergence of measures. The above can be obtained heuristically by formally differentiating (\ref{e_dualPDE}) with respect to $\eps_1$ when the initial conditions are $\varphi + \epsilon_1 \varphi_1 + \epsilon_2 \varphi_2$. By Lemmas 2.3 and 2.5 of \cite{M2002}, \eqref{e_Veps_firstderiv_PDE} has a unique solution, which we denote by $U^{1,\eps_1, \eps_2}_t$, which satisfies
\begin{equation} 
V^{\varphi,\eps_1, \eps_2}_t(x) = V_t^{\varphi,0,\eps_2}(x) +\int_0^{\eps_1} U_t^{1, \eps,\eps_2}(x) \, d\eps. \nonumber
\end{equation}
Thus $U_t^{1,\epsilon_1,\epsilon_2} = \frac{\partial}{\partial \epsilon_1} V_t^{\phi, \eps_1, \eps_2}$. We can apply the same argument to obtain a similar representation for $\frac{\partial}{\partial \epsilon_2} V_t^{\phi, \eps_1, \eps_2}$, which we denote by $U_t^{2,\epsilon_1,\epsilon_2}$. Both $U_t^{1,\epsilon_1,\epsilon_2}$ and $U_t^{2,\epsilon_1,\epsilon_2}$ have Feynman-Kac representations; for example, see Theorem 7.6 of Karatzas and Shreve \cite{KS} (on p. 366). For $i=1,2$ we have
\begin{equation} \label{e_Veps_firstderiv_FK}
U_t^{i, \eps_1, \eps_2}(x) = E_x^B \left(\varphi_i(B_t) \exp \left( - \int_0^t V^{\varphi,\eps_1,\eps_2}_{t-s} (B_s) \, ds \right)\right).
\end{equation}
We take the expression for $i=1$ and differentiate it with respect to $\eps_2$. We obtain
\begin{align}
&-\frac{\partial^2}{\partial \eps_2 \partial \eps_1} V^{\varphi,\eps_1, \eps_2}_t(x) \nonumber
\\ &\hspace{5 mm }= E^B_x \bigg( \varphi_1(B_t) \exp \left( -\int_0^t V^{\varphi,\eps_1,\eps_2}_{t-s} (B_s) \, ds \right)  \int_0^t U^{2,\eps_1, \eps_2}_{t-s}(B_s) \, ds \bigg)  \nonumber
\\ &\hspace{5 mm }= E^B_x \bigg( \varphi_1(B_t) \exp \left( -\int_0^t V^{\varphi,\eps_1,\eps_2}_{t-s} (B_s) \, ds \right) \nonumber
\\ &\hspace{9 mm }\times \int_0^t E^{B^2}_0 \bigg( \varphi_2(B_s + B^2_{t-s}) \exp \left( -\int_0^{t-s} V_{t-s-r}^{\varphi,\eps_1, \eps_2}(B_s + B^2_r) \,dr\right)  ds \bigg) \bigg), \nonumber
\end{align}
where the final line follows from another application of (\ref{e_Veps_firstderiv_FK}), this time with $i=2$. First we note that all the terms are non-negative, so we can take the internal integral over time outside the expectation. For $s<t$, the integrand then describes one Brownian motion started at $0$ and run to time $t$, and a second which branches from the first at time $s$ and evolves independently. By applying the Markov property at time $s$ we equivalently view it as a Brownian path that branches at time $s$ into two independent Brownian motions $B^1$ and $B^2$ which themselves run for a duration of $t-s$. This formulation combined with the independence of the Brownian motions gives us
\begin{align}
&\frac{\partial^2}{\partial \eps_2 \partial \eps_1} V^{\varphi,\eps_1, \eps_2}_t(x) = - \int_0^t E^B_x \bigg( \exp \left( -\int_0^s V^{\varphi,\eps_1,\eps_2}_{t-u} (B_u) \, ds \right) \nonumber
\\ &\times \prod_{i=1,2} E_0^{B^i} \bigg[ \varphi_i(B_s + B^i_{t-s}) \exp \left( -\int_0^{t-s} V_{t-s-r}^{\varphi,\eps_1, \eps_2}(B_s + B^i_r) \,dr\right) \bigg]  ds \bigg). \nonumber
\end{align}
The derivatives in $\eps_1$ and $\eps_2$ are one-sided at $0$ so we cannot exactly evaluate at $\eps_1 = \eps_2 = 0$. However, $V^{\phi,\eps_1,\eps_2}_t(x)$ is continuous in $\eps_1$ and $\eps_2$ and the integrand is bounded above by $\|\varphi_1\|_\infty \|\varphi_2\|_\infty$ so we can apply bounded convergence. As $\eps_1, \eps_2 \downarrow 0$, $V^{\varphi, \eps_1, \eps_2}_t \to V_t^\phi$ by Lemma 2.1(d) of \cite{M2002}. We also take $\eps_1, \eps_2 \downarrow 0$ in the left hand side of \eqref{e_2ptcanon_pde} and apply Dominated Convergence. Evaluating at $x = 0$ gives the result.
\end{proof}

\emph{Proof of Proposition \ref{prop_pde_rep1}.} We will prove the result for functions of product form, ie. $h(x,y) = \phi_1(x)\,\phi_2(y)$, and then use a monotone class theorem. Let $x_1,x_2 \in \R$ and $\lambda, \lambda' >0$. Consider the expression from Lemma \ref{lemma_pde_canon_func} with  $\varphi =\lambda \delta_{x_1} + \lambda' \delta_{x_2}$. For now we simply let $\varphi_1$ and $\varphi_2$ be functions satisfying the assumptions of Lemma~\ref{lemma_pde_canon_func}, but we will shortly choose them to be approximate identities at $x_1$ and $x_2$. Applying Lemma~\ref{lemma_pde_canon_func}, we have
\begin{align} \label{e_pde_intermed1}
&\N_0 \left( X_t(\varphi_1) X_t(\varphi_2) e^{-\lambda X(t,x_1) - \lambda' X(t,x_2)} \right) = \int_0^t E_0^B \bigg( \exp \left( - \int_0^s V^{\lambda, \lambda'}_{t-u} (B_u - x_1, B_u - x_2) \, du \right) \nonumber
\\ &\hspace{10 mm}\times \prod_{i=1,2} E_0^{B^i} \bigg[ \exp \left( - \int_0^{t-s} V^{\lambda, \lambda'}_{t-s-r}(B_s + B^i_r - x_1, B_s + B^i_r - x_2) \, dr \right) \varphi_i(B_s + B^i_{t-s}) \,\bigg] \bigg)\, ds,
\end{align}
where we have used translation invariance of $V^{\lambda,\lambda'}(x_1,x_2)$. Now let $\varphi_i = p_\delta(\cdot - x_i)$, and let $\phi_1, \phi_2$ be bounded, continuous functions and integrate $\phi_1(x_1) \phi_1(x_2)$  multiplied by the above over $x_1$ and $x_2$. The left hand side is then
\begin{align}  \label{e_pde_intermed2}
&\iint \phi_1(x_1) \, \phi_2(x_2)\, \N_0 \left( X_t(p_\delta(\cdot - x_1)) X_t(p_\delta(\cdot - x_2)) e^{-\lambda X(t,x_1) - \lambda' X(t,x_2)} \right)  dx_1\, dx_2.
\end{align} 
The above is absolutely bounded by 
\begin{align} \label{e_pde_bound1}
\|\phi_1\|_\infty \|\phi_2\|_\infty   \N_0 \left( \int X_t(p_\delta(\cdot - x_1)) dx_1 \int X_t(p_\delta(\cdot - x_2)) \, dx_2 \right) ,
\end{align}
where the change of order of integration follows because all the terms are non-negative once we bound $|\phi_i(x_i)|$ by $\|\phi_i\|_\infty$. Now we note that
\begin{align} \label{e_XapproxID}
\int X_t(p_\delta(\cdot - x_i)) dx_i = \iint X(t,y) p_\delta(x_i - y) \, dy \,dx _i &= \int X(t,y) \int \bigg( p_\delta(x_i - y) \, dx_i \bigg) \,dy = X_t(1). 
\end{align} 
Combined with (\ref{e_pde_bound1}), this implies that (\ref{e_pde_intermed2}) is absolutely integrable and absolutely bounded above by \\ $\|\phi_1\|_\infty \|\phi_2\|_\infty \N_0 ( X_t(1)^2 )$. Thus we can apply Fubini and rewrite (\ref{e_pde_intermed2}) as
\begin{align}  \label{e_pde_intermed3}
&\N_0 \left( \iint \phi_1(x_1) \, \phi_2(x_2)\, X_t(p_\delta(\cdot - x_1)) X_t(p_\delta(\cdot - x_2)) e^{-\lambda X(t,x_1) - \lambda' X(t,x_2)}  dx_1\, dx_2 \right).
\end{align} 
As noted, the expression inside $\N_0$ is absolutely bounded above by $\|\phi_1\|_\infty \|\phi_2\|_\infty X_t(1)^2$, which is integrable under $\N_0$, for all $\delta$. We take $\delta \downarrow 0$ and apply Dominated Convergence to obtain that the limit of \eqref{e_pde_intermed3} as $\delta \downarrow 0$ is equal to
\begin{align} \label{e_pde_intermed4}
&\N_0 \left( \lim_{\delta \to 0^+} \iint \phi_1(x_1) \, \phi_2(x_2)\, \N_0 \left( X_t(p_\delta(\cdot - x_1)) X_t(p_\delta(\cdot - x_2)) e^{-\lambda X(t,x_1) - \lambda' X(t,x_2)} \right)  dx_1\, dx_2 \right) \nonumber
\\ &= \N_0 \left( \lim_{\delta \to 0^+} \left( \int \phi_1(x_1) \, X_t(p_\delta(\cdot - x_1)) e^{-\lambda X(t,x_1)} dx_1 \right)  \left( \int \phi_2(x_2) \,X_t(p_\delta(\cdot - x_2)) e^{-\lambda X(t,x_2)} dx_2 \right)  \right).
\end{align}
We know that
\[X_t(p_\delta(\cdot - x_i)) = \int X(t,y) p_\delta(y - x_i) \, dy = X_t * p_\delta(x_i).\]
Moreover, $X(t,\cdot) \in C_c(\R)$ (ie. $X(t,\cdot)$ is continuous with compact support) $\N_0$-a.s. and $\{p_\delta\}_{\delta >0}$ are an approximate identity family, which together with the above imply that $X_t(p_\delta(\cdot - x_i)) \to X(t,x_i)$ as $\delta \downarrow 0$. Applying (\ref{e_XapproxID}) shows that the integrals in (\ref{e_pde_intermed4}) are absolutely bounded by $\|\phi_i\|_\infty X_t(1)$ for $i=1,2$ uniformly in $\delta>0$, so by another application of Dominated Convergence in (\ref{e_pde_intermed4}), the limit of \eqref{e_pde_intermed2} as $\delta \downarrow 0$ equals
\begin{align}
\N_0 \bigg( \bigg( \int \phi_1(x_1) \, X(t,x)\, e^{-\lambda X(t,x_1)} dx_1 \bigg)  \bigg( \int \phi_2(x_2) \,X(t,x_2) e^{-\lambda X(t,x_2)} dx_2 \bigg)  \bigg).  \nonumber
\end{align}
When rescaled by $(\lambda \lambda')^{2\lambda_0}$ this is equal to $\N_0 ( L^\lambda_t(\phi_1) \, L^{\lambda'}_t(\phi_2) )$. We now turn our attention to the right hand side of (\ref{e_pde_intermed1}). With $\varphi_i = p_\delta(\cdot - x_i)$, integrating against $\phi(x_1) \phi(x_2)  dx_1  dx_2$, we have
\begin{align*}
&\iint \phi_1(x_1) \, \phi_2(x_2) \bigg( \int_0^t E_0^B \bigg( E_{(0,0)}^{B^1, B^2}\bigg[ \exp \left( - \int_0^s V^{\lambda, \lambda'}_{t-u} (B_u - x_1, B_u - x_2) \, du \right) \nonumber
\\ &\times \exp \left( - \int_0^{t-s} V^{\lambda, \lambda'}_{t-s-r}(B_s + B^1_r - x_1, B_s + B^1_r - x_2) \, dr \right) p_\delta(B_s + B^1_{t-s} - x_1) 
\\ &\times \exp \left( - \int_0^{t-s} V^{\lambda, \lambda'}_{t-s-r}(B_s + B^2_r - x_1, B_s + B^2_r - x_2) \, dr \right) p_\delta(B_s + B^2_{t-s} - x_2) \bigg] \bigg)\, ds  \bigg) dx_1 \,dx_2.
\end{align*}
Since the above is equal to (\ref{e_pde_intermed2}), which we have shown is absolutely integrable, we can take the spatial integrals inside the expectations. At this point we note that we are integrating a bounded function of $x_1$ and $x_2$ with respect to the densities $p_\delta(B_s + B^1_{t-s} - x_i)$, which, because $p_\delta$ is the kernel of the Brownian semigroup, is the same as viewing $x_i$ as $B_s + B^i_{t-s + \delta}$. Hence the above is equal to
\begin{align} \label{e_pde_intermed5}
&\int_0^t E_0^B \bigg( E_{(0,0)}^{B^1, B^2}\bigg[ \phi_1( B_s + B^1_{t-s + \delta})\, \phi_2( B_s + B^2_{t-s + \delta})  \nonumber
\\ &\times \exp \left( - \int_0^s V^{\lambda, \lambda'}_{t-u} (B_u - B_s - B^1_{t-s + \delta}, B_u - B_s - B^2_{t-s + \delta}) \, du \right) \nonumber
\\ &\times \exp \left( - \int_0^{t-s} V^{\lambda, \lambda'}_{t-s-r}(B^1_r - B^1_{t-s+\delta}, B^1_r - B^2_{t-s+\delta}) \, dr \right) \nonumber
\\ &\times \exp \left( - \int_0^{t-s} V^{\lambda, \lambda'}_{t-s-r}(B^2_r - B^1_{t-s+\delta}, B^2_r - B^2_{t-s+\delta}) \, dr \right)\bigg] \bigg) ds. 
\end{align}
Taking $\delta \downarrow 0$ and applying Dominated Convergence, we note that because $B^i_{t-s+\delta} \to B^i_{t-s}$ and $\phi_1, \phi_2$ and $V^{\lambda,\lambda'}_s$ are continuous, the limit is equal to the above with $\delta = 0$. To obtain the desired form we make a time reversal of the Brownian motions. Let $\hat{B}^i_u = B^i_{t-s} - B^i_{t-s-u}$. We note that the $\hat{B}^i$ are standard Brownian motions and that $\hat{B}^i_{t-s} = B^i_{t-s}$, $\hat{B}^i_0 = 0$ and $B^i_r - B^i_{t-s} = -\hat{B}^i_{t-s-r}$. Making this substitution shows that \eqref{e_pde_intermed5} with $\delta = 0$ is equal to
\begin{align*}
&\int_0^t E_0^B \bigg( E_{(0,0)}^{\hat{B}^1, \hat{B}^2}\bigg[ \phi_1( B_s + \hat{B}^1_{t-s})\, \phi_2( B_s + \hat{B}^2_{t-s})  \nonumber
\\ &\times \exp \left( - \int_0^s V^{\lambda, \lambda'}_{t-u} (B_u - B_s - \hat{B}^1_{t-s}, B_u - B_s - \hat{B}^2_{t-s}) \, du \right) \nonumber
\\ &\times \exp \left( - \int_0^{t-s} V^{\lambda, \lambda'}_{t-s-r}(-\hat{B}^1_{t-s-r}, -\hat{B}^1_{t-s-r} + \hat{B}^1_{t-s} - B^2_{t-s}) \, dr \right) 
\\ &\times \exp \left( - \int_0^{t-s} V^{\lambda, \lambda'}_{t-s-r}(-\hat{B}^2_{t-s-r} + \hat{B}^2_{t-s} - B^1_{t-s}, -\hat{B}^2_{t-s-r}) \, dr \right) \bigg] \bigg) ds.
\end{align*}
The time index of the Brownian motions now matches the time index of the function $V^{\lambda, \lambda'}$ in the last two lines, allowing us to reverse the time of the integrals for a simpler expression. To obtain the desired expression we now apply the fact that $V^{\lambda,\lambda'}_t(a,b) = V^{\lambda,\lambda'}_t(-a,-b)$ and relabel $\hat{B}^i$ to be simply $B^i$. This proves the result for $h(x_1,x_2) = \phi_1(x_1)\,\phi_2(x_2)$ when $\phi_1, \phi_2$ are bounded and continuous. The result for general bounded measurable $h:\R^2 \to \R$ now follows from a standard monotone class argument such as Corollary 4.4 in the Appendix of Ethier and Kurtz \cite{EK}. \qed \\

\textbf{Definition.} Let $\Gamma^{\lambda,\lambda'}(s)$ denote the integrand in Proposition \ref{prop_pde_rep1}, so that the proposition states that 
\begin{equation} \label{e_Gammaint}
 \N_0 ((L^\lambda_t \times L^{\lambda'}_t)(h) ) = (\lambda \lambda')^{2\lambda_0}\int_0^t \Gamma^{\lambda, \lambda'}(s) ds.
\end{equation}\\

$\Gamma^{\lambda,\lambda'}(s)$ also depends on $h$, but we omit this. The next lemma changes variables to obtain an expression involving Ornstein-Uhlenbeck processes. We first introduce some notation. For bounded and measurable $h:\R^2 \to \R$ and a (continuous) path $(B_u : u \in [0,s])$, define $\Psi_{B,s}^{\lambda, \lambda'}(\cdot,\cdot)$ by
\begin{equation} \label{e_psi}
\Psi_{B,s}^{\lambda, \lambda'}(x,y) = h(x + B_s, y + B_s) \exp \left( - \int_0^s V^{\lambda, \lambda'}_{t-u} (x + B_s - B_u, y + B_s - B_u) \, du \right).
\end{equation} 
We define $H^c_u$ as a scaling of $V^{\lambda, \lambda'}_t$:
\begin{equation} \label{Hdef}
H_u^c(x,y) = u V_u^{1,c}(\sqrt u x, \sqrt u y) = V_1^{\sqrt u, \sqrt u c}(x,y).
\end{equation}
The scaling in the following lemma cannot be done uniformly for all $s\in[0,t]$ because it requires  $\lambda^2> (t-s)^{-1}$ and $\lambda'^2 > (t-s)^{-1}$. We derive an expression for $\Gamma^{\lambda, \lambda'}(s)$ in terms of two independent Ornstein-Uhlenbeck processes which we denote $Y^1$ and $Y^2$, for which we denote the joint (independent) expectation $E_{(x,y)}^{Y^1, Y^2}$.
\begin{lemma} \label{lemma_pde_rep2}
Let $0<s<t$, $T_1 = T_1(s) =  \log(\lambda^2 (t-s))$, $T_2 = T_2(s) = \log(\lambda'^2(t-s))$. Then for all $\lambda> (t-s)^{-1/2}$ and $\lambda' > (t-s)^{-1/2}$, we have
\begin{align*}
 \Gamma^{\lambda, \lambda'}(s)  =\,&E_0^B \bigg( E_{(0,0)}^{B^1, \, B^2}\bigg[ E_{(B_1^1, B^2_1)}^{Y^1, Y^2} \bigg( \Psi^{\lambda, \lambda'}_{B,s}(\sqrt{t-s}\, Y^1_{T_1}, \sqrt{t-s} \,Y^2_{T_2})
\\&  \times \exp \bigg(- \int_0^1 V^{1, \lambda'/ \lambda}_{u}(B^1_{u},B^1_{u} +  e^{T_1/2}( Y^2_{T_2} -   Y^1_{T_1})) + V^{1, \lambda / \lambda'}_{u} (B^2_{u} ,B^2_{u}    + e^{T_2/2}( Y^1_{T_1} -  Y^2_{T_2}) )\, du \bigg)
\\ & \times \exp \bigg( - \int_0^{T_1} H_{e^u}^{\lambda' / \lambda} (Y^1_u, Y^1_u +  e^{(T_1-u)/2}( Y^2_{T_2} - Y^1_{T_1} )) \, du \bigg) 
\\ &   \times \exp \bigg( - \int_0^{T_2} H_{e^u}^{\lambda / \lambda'} (Y^2_u, Y^2_u +  e^{(T_2-u)/2}( Y^1_{T_1} - Y^2_{T_2} )) \, du \bigg) \bigg) \bigg] \bigg).
\end{align*}
\end{lemma}
\begin{proof}
We begin with the expression from Proposition \ref{prop_pde_rep1}. We observe that $\Psi^{\lambda, \lambda'}_{B,s}$ appears and we may write the quantities in the first two lines as $\Psi^{\lambda, \lambda'}_{B,s}(B^1_{t-s}, B^2_{t-s})$. In the third and fourth lines we apply (\ref{e_Vscale_2pt}) to obtain
\begin{align*}
\Gamma^{\lambda, \lambda'}(s)  =\,\,& E_0^B \bigg( E_{(0,0)}^{B^1, \,B^2}\bigg[ \Psi^{\lambda, \lambda'}_{B,s}(B^1_{t-s}, B^2_{t-s})
\\ &\times \exp \left( - \int_0^{t-s}  \lambda^2 V^{1, \lambda'/ \lambda}_{\lambda^2 r}(\lambda B^1_{r}, \lambda(B^1_{r} + B^2_{t-s}- B^1_{t-s})) \, dr \right) 
\\ &\times \exp \left( - \int_0^{t-s} \lambda'^2 V^{\lambda / \lambda', 1}_{\lambda'^2 r}(\lambda' (B^2_{r}   + B^1_{t-s}- B^2_{t-s}) , \lambda' B^2_{r} ) \, dr \right) \bigg] \bigg).
\end{align*}
We define $\hat{B}_u^1 = \lambda B^1_{\lambda^{-2} u}$ and $\hat{B}_u^2 = \lambda' B^2_{\lambda'^{-2} u}$, which are both standard Brownian motions. Making a time change in the integrals (ie. letting $u = \lambda^2 r$ or $\lambda'^2r$) gives
\begin{align*}
\Gamma^{\lambda, \lambda'}(s)  =\,\,& E_0^B \bigg( E_{(0,0)}^{\hat{B}^1, \, \hat{B}^2}\bigg[ \Psi^{\lambda, \lambda'}_{B,s}(\lambda^{-1} \hat{B}^1_{\lambda^{2}(t-s)}, \lambda'^{-1} \hat{B}^2_{\lambda'^{2}(t-s)})
\\ &\times \exp \left( - \int_0^{\lambda^2(t-s)} V^{1, \lambda'/ \lambda}_{u}(\hat{B}^1_{u}, \hat{B}^1_{u} +  \frac{\lambda}{\lambda'} \hat{B}^2_{\lambda'^2(t-s)} -  \hat{B}^1_{\lambda^2(t-s)}) \, du \right) 
\\ &\times \exp \left( - \int_0^{\lambda'^2(t-s)} V^{\lambda / \lambda', 1}_{u} (\hat{B}^2_{u}    + \frac{\lambda'}{\lambda} \hat{B}^1_{\lambda^2(t-s)}-  \hat{B}^2_{\lambda'^{2}(t-s)} , \hat{B}^2_{u} ) \, du \right) \bigg] \bigg).
\end{align*}
Because we have assumed $\lambda, \lambda'> (t-s)^{-1/2}$, the upper bounds of integration in the integrals are greater than $1$. We now apply the Markov property for $\hat{B}^i$ at time $u=1$. We collect the portions of the integrals from the second and third lines on the interval $[0,1]$, leaving the integrals from $1$ to $\lambda^2(t-s)$ and $ \lambda'^2(t-s)$. Conditional on $\hat{B}^i_1$, the Brownian motions in the integrands' arguments are Brownian motions with initial position $\hat{B}^i_1$. If we denote these by $\tilde{B}^i_u$ (in which case, essentially, $\tilde{B}^i_u = \hat{B}^i_{u+1}$), we obtain
\begin{align} \label{scaleaux1}
\Gamma^{\lambda, \lambda'}(s)  =&\,E_0^B \bigg( E_{(0,0)}^{\hat{B}^1, \, \hat{B}^2}\bigg[ E_{(\hat{B}^1_1, \hat{B}^2_1)}^{\tilde{B}^1, \tilde{B}^2} \bigg( \Psi^{\lambda, \lambda'}_{B,s}(\lambda^{-1} \tilde{B}^1_{\lambda^{2}(t-s)-1}, \lambda'^{-1} \tilde{B}^2_{\lambda'^{2}(t-s)-1}) \nonumber
\\& \times \exp \bigg(- \int_0^1 V^{1, \lambda'/ \lambda}_{u}(\hat{B}^1_{u}, \hat{B}^1_{u} +  \frac{\lambda}{\lambda'} \tilde{B}^2_{\lambda'^{2}(t-s)-1} -  \tilde{B}^1_{\lambda^{2}(t-s)-1})  \nonumber
\\ &\hspace{ 19 mm }+ V^{\lambda / \lambda', 1}_{u} (\hat{B}^2_{u}    + \frac{\lambda'}{\lambda} \tilde{B}^1_{\lambda^2(t-s)-1}-  \tilde{B}^2_{\lambda'^{2}(t-s)-1} , \hat{B}^2_{u} )\, du \bigg) \nonumber
\\ &\times \exp \bigg( - \int_0^{\lambda^2(t-s)-1} V^{1, \lambda'/ \lambda}_{u+1}(\tilde{B}^1_{u}, \tilde{B}^1_{u} +  \frac{\lambda}{\lambda'} \tilde{B}^2_{\lambda'^2(t-s)-1} -  \tilde{B}^1_{\lambda^2(t-s)-1}) \, du \bigg)  \nonumber
\\ &\times \exp \bigg( - \int_0^{\lambda'^2(t-s)-1} V^{\lambda / \lambda', 1}_{u+1} (\tilde{B}^2_{u}    + \frac{\lambda'}{\lambda} \tilde{B}^1_{\lambda^2(t-s)-1}-  \tilde{B}^2_{\lambda'^{2}(t-s)-1} , \tilde{B}^2_{u} ) \, du \bigg) \bigg) \bigg] \bigg). 
\end{align}
Recall that if a process $Y$ is defined by
\[ Y_r = e^{-r/2} B_{e^r - 1} \]
where $B$ is a standard Brownian motion, then $Y$ is a standard one-dimensional Ornstein-Uhlenbeck process with $Y_0 = B_0$. For $i=1,2$ we let $Y^i_r = e^{-r/2} \tilde{B}^i_{e^r-1}$. Recall that $T_1 = \log(\lambda^2 (t-s))$ and $T_2= \log(\lambda'^2 (t-s))$. We therefore have that 
\[ \tilde{B}^1_{\lambda^2(t-s) - 1} = e^{T_1/2} Y^1_{T_1}, \,\,\,\,  \tilde{B}^2_{\lambda'^2(t-s) - 1} = e^{T_2/2} Y^2_{T_2}.\] Expressing $\lambda$ and $\lambda'$ in terms of $T_1,T_2$ shows that 
\[ \frac{\lambda}{\lambda'} \tilde{B}^2_{\lambda^2(t-s)-1} = e^{T_1/2} Y^2_{T_2}, \,\,\,\,\, \frac{\lambda'}{\lambda} \tilde{B}^1_{\lambda'^2(t-s)-1} = e^{T_2/2} Y^1_{T_1}.\]
Likewise, we express the argument of $\Psi^{\lambda, \lambda'}_{B,s}$ in terms of $Y^i$ and $T_i$. We substitute $u = e^r - 1$ and apply the above in \eqref{scaleaux1} to obtain
\begin{align*}
\Gamma^{\lambda, \lambda'}(s)  =\,& E_0^B \bigg( E_{(0,0)}^{\hat{B}^1, \, \hat{B}^2}\bigg[ E_{(\hat{B}^1_1, \hat{B}^2_1)}^{Y^1, Y^2} \bigg( \Psi^{\lambda, \lambda'}_{B,s}(\sqrt{t-s}\, Y^1_{T_1}, \sqrt{t-s} \,Y^2_{T_2})
\\& \times \exp \bigg(- \int_0^1 V^{1, \lambda'/ \lambda}_{u}(\hat{B}^1_{u}, \hat{B}^1_{u} +  e^{T_1/2}( Y^2_{T_2} -   Y^1_{T_1})) + V^{\lambda / \lambda', 1}_{u} (\hat{B}^2_{u}    + e^{T_2/2}( Y^1_{T_1} -  Y^2_{T_2}) , \hat{B}^2_{u} )\, du \bigg)
\\ &\times \exp \bigg( - \int_0^{T_1} e^r V^{1, \lambda'/ \lambda}_{e^r}(e^{r/2}Y^1_r, e^{r/2}Y^1_r +  e^{T_1/2}( Y^2_{T_2} - Y^1_{T_1} )) \, dr \bigg) 
\\ &\times \exp \bigg( - \int_0^{T_2} e^r V^{\lambda / \lambda', 1}_{e^r} (e^{r/2}Y^2_r  +  e^{T_2/2}( Y^1_{T_1} - Y^2_{T_2} )) , e^{r/2}Y^2_r ) \, dr \bigg) \bigg) \bigg] \bigg).
\end{align*}
We now apply (\ref{e_Vscale_2pt}) and \eqref{Hdef} in the third and fourth lines. In the third line this gives 
\begin{align*}
&e^r V^{1, \lambda'/ \lambda}_{e^r}(e^{r/2}Y^1_r, e^{r/2}Y^1_r +  e^{T_1/2}( Y^2_{T_2} - Y^1_{T_1} ))
\\ &\hspace{5 mm}= V^{e^{r/2}, e^{r/2} \lambda'/ \lambda}_1(Y^1_r, Y^1_r +  e^{(T_1-r)/2}( Y^2_{T_2} - Y^1_{T_1} ))
\\&\hspace{5 mm}= H_{e^r}^{\lambda' / \lambda} (Y^1_r, Y^1_r +  e^{(T_1-r)/2}( Y^2_{T_2} - Y^1_{T_1} )),
\end{align*} 
and similar in the fourth. Noting that $V^{c,d}_t(a,b) = V^{d,c}_t(b,a)$, we have obtained the desired expression.
\end{proof}

We now obtain an upper bound for $\Gamma^{\lambda,\lambda'}(s)$ and show that the contribution to $\N_0 ((L^\lambda_t \times L^{\lambda'}_t)(h))$ from the integral over $[t-\epsilon,t]$ vanishes as $(\epsilon, \lambda') \to (0,\infty)$.

\begin{lemma} \label{lemma_tepsprelimit} Suppose $\lambda^2 t \geq 1$, and let $h:\R^2\to \R$ be bounded. There is a constant $C_{ \ref{lemma_tepsprelimit}}>0$ such that the following hold.\\
(a) For all $\lambda' > (t-s)^{-1/2}$, 
\begin{equation}
(\lambda \lambda')^{2\lambda_0}\big| \Gamma^{\lambda, \lambda'}(s) \big| \leq C_{ \ref{lemma_tepsprelimit}} \|h\|_\infty t^{-\lambda_0} (t-s)^{-\lambda_0}. \nonumber
\end{equation}
(b) For $0<\epsilon <t$,
\begin{align} 
(\lambda \lambda')^{2\lambda_0}  \int_{t-\epsilon}^t \big| \Gamma^{\lambda, \lambda'}(s) \big| \, ds & \leq C_{ \ref{lemma_tepsprelimit}} \|h\|_\infty t^{-\lambda_0} (\epsilon^{1-\lambda_0} + \lambda'^{-2(1-\lambda_0)}). \nonumber
\end{align}

\end{lemma} 
\begin{proof} To begin we use $|h|\leq \|h\|_\infty$ and apply monotonicity (Proposition~\ref{prop_mono_subadd}(a)), ie. $V^\lambda(x), V^{\lambda'}(y) \leq V^{\lambda, \lambda'}(x,y)$, to obtain
\begin{align}
(\lambda \lambda'&)^{2\lambda_0}  \big| \Gamma^{\lambda, \lambda'}(s) \big| \nonumber
\\ &\leq \|h\|_\infty (\lambda \lambda')^{2\lambda_0}  E_0^B \bigg( E_{(0,0)}^{B^1, \,B^2}\bigg[ \exp \left( - \int_0^s V^{\lambda}_{t-u} (B^1_{t-s} + B_s - B_u) \, du \right) \nonumber
\\ &\hspace{4 mm} \times \exp \bigg( - \int_0^{t-s} V^{\lambda}_{r}(B^1_{r}) \, dr \bigg) \exp \bigg( - \int_0^{t-s} V^{\lambda'}_{r}(B^2_{r} ) \, dr \bigg) \bigg] \bigg)  \nonumber
\\ &= \|h\|_\infty (\lambda \lambda')^{2\lambda_0}  E_0^{B^1} \bigg(  \exp \bigg( - \int_0^t V^{\lambda}_{u} (B^1_u) \, du \bigg) \bigg)E_0^{B^2} \bigg( \exp \bigg( - \int_0^{t-s} V^{\lambda'}_{r}(B^2_{r} ) \, dr \bigg) \bigg) ,  \nonumber
\end{align}
where the final line follows from a time reversal of $B$ and concatenating the time-reversed $B$ with $B^1$. Applying (\ref{e_Vscale}) twice and changing the time variable, the above is equal to
\begin{align} 
\|h\|_\infty (\lambda \lambda')^{2\lambda_0}  E_0^{B^1} \bigg(  \exp \bigg( - \int_0^{\lambda^2t} V^{1}_{u} (\lambda B^1_{\lambda^{-2}u}) \, du \bigg) \bigg) E_0^{B^2} \bigg( \exp \bigg( - \int_0^{\lambda'^2(t-s)} V^{1}_{u}(\lambda ' B^2_{\lambda'^{-2}u} ) \, du \bigg) \bigg). \nonumber
\end{align}
The rescaled Brownian motions in the above are themselves standard Brownian motions which we will denote by $\hat{B}^1,\hat{B}^2$. We next let $e^r = u$ in both integrals and apply (\ref{e_Vscale}) to see that the above equals
\begin{align} \label{e_prelimeps4}
\|h\|_\infty (\lambda \lambda')^{2\lambda_0}  &E_0^{\hat{B}^1} \bigg(  \exp \bigg( - \int_{-\infty}^{\log(\lambda^2t)} V^{e^{r/2}}_{1} (e^{-r/2} \hat{B}^1_{e^r}) \, dr \bigg) \bigg)  E_0^{\hat{B}^2} \bigg( \exp \bigg( - \int_{-\infty}^{\log(\lambda'^2(t-s))} V^{e^{r/2}}_{1}(e^{-r/2}\hat{B}^2_{e^r} ) \, dr \bigg) \bigg) ds \nonumber
\\ &\leq \|h\|_\infty(\lambda \lambda')^{2\lambda_0}  E_m^{Y^1} \bigg(  \exp \bigg( - \int_{0}^{\log(\lambda^2t)} V^{e^{r/2}}_{1} (Y^1_r) \, dr \bigg) \bigg)  E^{Y^2} \bigg( \exp \bigg( - \int_{-\infty}^{\log(\lambda'^2(t-s))} V^{e^{r/2}}_{1}(Y^2_r) \, dr \bigg) \bigg) ,
\end{align}
where $Y^i_r = e^{-r/2} \hat{B}^i_{e^r}$, which makes $Y^i_u$ a stationary Ornstein-Uhlenbeck process for $u \in \R$, and we recall our assumption that $\lambda^2 t > 1$. We condition on the value of $Y^1_0$, which has distribution $m$. \\

We first use the above to prove (a). Assuming that $\lambda' > (t-s)^{-1/2}$, the upper endpoint of the second integral is positive, so by \eqref{e_prelimeps4} we have
\begin{align}\label{e_prelimeps3}
&(\lambda \lambda')^{2\lambda_0}  \big| \Gamma^{\lambda, \lambda'}(s) \big|  \nonumber
\\&\hspace{6 mm}\leq \|h\|_\infty (\lambda \lambda')^{2\lambda_0}  E_m^{Y^1} \bigg(  \exp \bigg( - \int_{0}^{\log(\lambda^2t)} V^{e^{r/2}}_{1} (Y^1_r) \, dr \bigg) \bigg)  E^{Y^2}_m \bigg( \exp \bigg( - \int_0^{\log(\lambda'^2(t-s))} V^{e^{r/2}}_{1}(Y^2_r) \, dr \bigg) \bigg),
\end{align}
where we have also conditioned on $Y^2_0$. In order to approximate the expectations above with survival probabilities for killed Ornstein-Uhlenbeck processes, we add and subtract $F(Y^i_u)$ in the integrals. Recalling the definition of $Z_T(Y)$ from \eqref{e_Zdef}, we define $Z^1_T(Y^1), Z^2_T(Y^2)$ in the same way. Thus \eqref{e_prelimeps3} is equal to
\begin{align} \label{e_prelimeps33}
\|h\|_\infty(\lambda &\lambda')^{2\lambda_0}  E_m^{Y^1} \bigg(Z^1_{\log(\lambda^2t)}(Y^1)  \exp \bigg( - \int_{0}^{\log(\lambda^2t)}F (Y^1_r) \, du \bigg) \bigg) \nonumber
\\ \times  E_m^{Y^2} \bigg( &Z^2_{\log(\lambda'^2(t-s))}(Y^2) \exp \bigg( - \int_{0}^{\log(\lambda'^2(t-s))} F(Y^2_r) \, dr \bigg) \bigg) \nonumber
\\ \leq &\,\|h\|_\infty(\lambda \lambda')^{2\lambda_0}  C_Z E_m^{Y^1} \bigg( \exp \bigg( - \int_{0}^{\log(\lambda^2t)}F (Y^1_r) \, du \bigg) \bigg) \nonumber
\\  &\times C_Z E_m^{Y^2} \bigg( \exp \bigg( - \int_{0}^{\log(\lambda'^2(t-s))} F(Y^2_r) \, dr \bigg) \bigg) ds\nonumber
\\ = &\,C\|h\|_\infty(\lambda \lambda')^{2\lambda_0}  P_m^{Y^1}(\rho^F > \log(\lambda^2t)) \,  P_m^{Y^2}(\rho^F > \log(\lambda'^2(t-s)))
\end{align}
In the first inequality we have used \eqref{e_ZTbd} twice, and the second equality follows by recognizing the expectations as survival probabilities of killed Ornstein-Uhlenbeck processes killed at rate $F(Y^i_r)$. By \eqref{OU_survivalprobGaus}, we have
\begin{equation}
 P_m^{Y^1}(\rho^F > \log(\lambda^2t)) \leq C t^{-\lambda_0} \lambda^{-2\lambda_0}, \,\,\,\,\, P_m^{Y^2}(\rho^F > \log(\lambda'^2(t-s))) \leq C (t-s)^{-\lambda_0} \lambda'^{-2\lambda_0}. \nonumber
\end{equation}
Using the above in \eqref{e_prelimeps33}, which is an upper bound for $(\lambda \lambda')^{2\lambda_0}  \big| \Gamma^{\lambda, \lambda'}(s) \big|$, proves (a).\\

We now show (b). Let $0< \epsilon < t$. Using \eqref{e_prelimeps4} we obtain that
\begin{align} \label{e_prelimeps5}
(\lambda \lambda')^{2\lambda_0} & \int_{t-\epsilon}^t\big| \Gamma^{\lambda, \lambda'}(s) \big| \, ds \, \leq \, \|h\|_\infty (\lambda \lambda')^{2\lambda_0}  E_m^{Y^1} \bigg(  \exp \bigg( - \int_{0}^{\log(\lambda^2t)} V^{e^{r/2}}_{1} (Y^1_r) \, dr \bigg) \bigg) \nonumber
\\&\hspace{ 31 mm }\times \int_{t-\epsilon}^t  E^{Y^2} \bigg( \exp \bigg( - \int_{-\infty}^{\log(\lambda'^2(t-s))} V^{e^{r/2}}_{1}(Y^2_r) \, dr \bigg) \bigg) ds.
\end{align}
We can approximate the first expectation with the survival probability of $Y^1$, just as we did in the proof of (a), and bound it above by $C\lambda^{-2\lambda_0}t^{-\lambda_0}$. Furthermore, by the proof of part (a), we know that when $\lambda' > (t-s)^{-1/2}$ the expectation in the integral above is bounded above by $C (\lambda')^{-2\lambda_0} (t-s)^{-\lambda_0}$. When this is not the case we bound it above by $1$. Thus \eqref{e_prelimeps5} is bounded above by
\begin{align}
&C \|h\|_\infty t^{-\lambda_0} \left[ 1(\lambda' \geq \epsilon^{-1/2}) \int_{t-\epsilon}^{t-\lambda'^{-2}}  (t-s)^{-\lambda_0}\,ds + (\lambda')^{2\lambda_0} \int_{t-\lambda'^{-2}}^t E^{Y^2} \left( \exp \left( - \int_{-\infty}^{\log(\lambda'^2(t-s))} V^{e^{r/2}}_{1}(Y^2_r) \, dr \right) \right)ds \right]\nonumber
\\ &\hspace{4 mm} \leq C \|h\|_\infty t^{-\lambda_0} \left[ \epsilon^{1-\lambda_0} + \lambda'^{-2(1-\lambda_0)}\right]. \nonumber
\end{align}
The result now follows.
\end{proof}

\emph{Proof of Theorem~\ref{thm_l2limit}.} Let $h : \R^2 \to \R$ be bounded and measurable. Clearly we may assume without loss of generality that $h \geq 0$. We recall from \eqref{e_Gammaint} and Proposition~\ref{prop_pde_rep1} that
\[ \N_0((L^\lambda_t \times  L^{\lambda'}_t(h)) = \int_0^{t} (\lambda \lambda')^{2\lambda_0} \Gamma^{\lambda, \lambda'}(s) \,ds,\]
where $h\geq 0$ implies that $\Gamma^{\lambda,\lambda'}(s) \geq 0$. Our strategy is to compute the limit of $(\lambda \lambda')^{2\lambda_0} \Gamma^{\lambda, \lambda'}(s)$ as $\lambda,\lambda' \to \infty$ and pass the limit through the integral. However, the scaling we use cannot be done uniformly in $s$. In order to handle this and the singularity at $s = t$, we fix $\epsilon > 0$ and analyse the integral on $[t-\epsilon,t]$ separately. We have 
\begin{equation} \label{e_Gammaintsplit}
\N_0((L^\lambda_t \times  L^{\lambda'}_t(h)) = \int_0^{t-\epsilon} (\lambda \lambda')^{2\lambda_0} \Gamma^{\lambda, \lambda'}(s) \,ds + (\lambda \lambda')^{2\lambda_0}\int_{t-\epsilon}^t \Gamma^{\lambda, \lambda'}(s)\, ds. 
\end{equation}
By Lemma \ref{lemma_tepsprelimit}(b), the limit superior of the absolute value of the second term as $\lambda' \to \infty$ is bounded above by $C\|h\|_\infty t^{-\lambda_0} \epsilon^{1-\lambda_0}$. Hence, if
\[\lim_{\lambda, \lambda' \to \infty}\int_0^{t-\epsilon} (\lambda \lambda')^{2\lambda_0} \Gamma^{\lambda, \lambda'}(s) \,ds \]
exists for all $\epsilon > 0$, then by the Cauchy condition $\lim_{\lambda, \lambda' \to \infty} \N_0((L^\lambda_t \times  L^{\lambda'}_t(h))$ exists and is the limit of the above as $\epsilon \downarrow 0$. Thus it suffices to fix $\epsilon>0$ and establish the convergence of, and find the limit of, the first term of \eqref{e_Gammaintsplit}, first as $\lambda, \lambda' \to \infty$ and then as $\epsilon \downarrow 0$. By Lemma~\ref{lemma_tepsprelimit}(a), we have
\[(\lambda \lambda')^{2\lambda_0} | \Gamma^{\lambda,\lambda'}(s)| \leq g(s) \,\,\, \text{ for all } \, s\in [0,t-\epsilon]\]
for all $\lambda, \lambda' > \epsilon^{-1/2}$ for a function $g(s) \geq0$ satisfying $\int_0^{t-\epsilon} g(s) ds < \infty$. Thus if $(\lambda \lambda')^{2\lambda_0} \Gamma^{\lambda, \lambda'}(s)$ converges as $\lambda,\lambda' \to \infty$, Dominated Convergence implies
\begin{align} \label{e_Gammareduce}
\lim_{\lambda, \lambda'  \to \infty} \N_0((L^\lambda \times L^{\lambda'}_t)(h)) &= \lim_{\epsilon \to 0^+} \lim_{\lambda, \lambda' \to \infty}\int_0^{t-\epsilon} (\lambda \lambda')^{2\lambda_0} \Gamma^{\lambda, \lambda'}(s) \,ds  \nonumber
\\ &= \lim_{\epsilon \to 0^+}\int_0^{t-\epsilon} \lim_{\lambda, \lambda' \to \infty} (\lambda \lambda')^{2\lambda_0} \Gamma^{\lambda, \lambda'}(s) \,ds,
\end{align}
and so it suffices to find the limit of $(\lambda \lambda')^{2\lambda_0} \Gamma^{\lambda, \lambda'}(s)$ as $\lambda, \lambda' \to \infty$. \\

Let $s \in (0,t)$ and assume $\lambda, \lambda' > (t-s)^{-1/2}$. By Lemma \ref{lemma_pde_rep2},
\begin{align} \label{e_original}
(\lambda \lambda')^{2\lambda_0} \Gamma^{\lambda, \lambda'}(s)\nonumber =\,&(\lambda \lambda')^{2\lambda_0} E^{B}_0 \bigg( E^{B^1, B^2}_{(0,0)} \, \bigg(\, E_{B^1_1, B^2_1}^{Y^1, Y^2} \, \bigg[\,  \Psi^{\lambda, \lambda'}_{B,s} (\sqrt{t-s}\, Y^1_{T_1},\sqrt{t-s} \, Y^2_{T_2} )  \nonumber
\\ &\times \exp \bigg(- \int_0^1 V^{1, \lambda'/ \lambda}_{u}(B^1_{u},B^1_{u} +  e^{T_1/2}( Y^2_{T_2} -   Y^1_{T_1})) + V^{1, \lambda / \lambda'}_{u} (B^2_{u} ,B^2_{u}    + e^{T_2/2}( Y^1_{T_1} -  Y^2_{T_2}) )\, du \bigg) \nonumber
\\ &\times \exp \bigg(-\int_0^{T_1} H_{e^u}^{\lambda' / \lambda} (Y^1_u, Y^1_u + e^{\frac{T_1-u}{2}}(Y^2_{T_2} - Y^1_{T_1} )) \, du \bigg) \nonumber
\\ & \times \exp \bigg(-\int_0^{T_2} H_{e^u}^{\lambda / \lambda'} (Y^2_u, Y^2_u + e^{\frac{T_2-u}{2}}(Y^1_{T_1} - Y^2_{T_2} )) \, du \bigg)\bigg) \bigg) \bigg],
\end{align}
where $T_1 = T_1(s) = \log(\lambda^2(t-s))$, $T_2 = T_2(s) = \log(\lambda'^2(t-s))$. Inside the integral in the third term we add and subtract $F(Y^i_u)$ and decompose as follows
\begin{align*}
&\exp \bigg(-\int_0^{T_1} H_{e^u}^{1} (Y^1_u, Y^1_u + e^{\frac{T-u}{2}}(Y^2_{T_2} - Y^1_{T_1} )) \, du \bigg)
\\ &\hspace{ 5mm} = \exp \bigg( - \int_0^{T_1} F(Y^1_u) \, du \bigg) \exp \bigg(\int_0^{T_1} F(Y^1_u) -  H_{e^u}^{\lambda' / \lambda} (Y^1_u, Y^1_u + e^{\frac{T_1-u}{2}}(Y^2_{T_2} - Y^1_{T_1} )) \, du \bigg). 
\end{align*}
We do the same to the fourth term with the obvious changes of indices. The first term in the above is the probability that the Ornstein-Uhlenbeck process $Y^1$ with killing function $F$ survives until time $T_1$. We extract a similar term from the symmetric term corresponding to $Y^2$ and $T_2$. Weighting the expectation of a functional with this survival probability is equivalent to restricting the expectation to the event that the process survives; in our case, we restrict to the event that $Y^1$ and $Y^2$ survive until $T_1$ and $T_2$, respectively. Thus \eqref{e_original} is equal to
\begin{align} \label{e_ogmod1}
&(\lambda \lambda')^{2\lambda_0}  E^{B}_0 \bigg( E^{B^1, B^2}_{(0,0)} \, \bigg(\, E_{B^1_1, B^2_1}^{Y^1, Y^2} \, \bigg[\,  \Psi^{\lambda, \lambda'}_{B,s} (\sqrt{t-s}\, Y^1_{T_1},\sqrt{t-s} \,Y^2_{T_2} )  \nonumber
\\ &\times \exp \bigg(- \int_0^1 V^{1, \lambda'/ \lambda}_{u}(B^1_{u},B^1_{u} +  e^{T_1/2}( Y^2_{T_2} -   Y^1_{T_1})) + V^{1, \lambda / \lambda'}_{u} (B^2_{u} ,B^2_{u}    + e^{T_2/2}( Y^1_{T_1} -  Y^2_{T_2}) )\, du \bigg)  \nonumber
\\ &\times \exp \bigg(\int_0^{T_1} F(Y^1_u) -  H_{e^u}^{\lambda' / \lambda} (Y^1_u, Y^1_u + e^{\frac{T_1-u}{2}}(Y^2_{T_2} - Y^1_{T_1} )) \, du \bigg) \nonumber
\\ & \times \exp \bigg(\int_0^{T_2} F(Y^2_u) - H_{e^u}^{\lambda / \lambda'} (Y^2_u, Y^2_u + e^{\frac{T_2-u}{2}}(Y^1_{T_1} - Y^2_{T_2} )) \, du \bigg) 1(\rho_1> T_1) 1(\rho_2>T_2)\,\bigg]\, \bigg) \bigg) , 
\end{align}
where $\rho_i = \rho_i^F$ is the lifetime of the killed process $Y^i$. Recall the transition density $q_t(\cdot,\cdot)$ (with respect to $m$) of the killed diffusion . We condition on the endpoints $Y^i_{T_i} = z_i$ (recall from Lemma~\ref{lemma_endpoint_indep}(a) that the regular conditional distributions exist for all $z_i \in \R$) and integrate against $q_{T_i}(\cdot, z_i) \,dm(z_i)$ to obtain that (\ref{e_ogmod1}) is equal to
\begin{align} \label{e_ogmod2}
&(\lambda \lambda')^{2\lambda_0}  E^{B}_0 \bigg( E^{B^1, B^2}_{(0,0)} \, \bigg(\, \iint\,  \Psi^{\lambda, \lambda'}_{B,s} (\sqrt{t-s}z_1 ,\sqrt{t-s} z_2 ) \nonumber
\\ &\times \exp \bigg(- \int_0^1 V^{1, \lambda'/ \lambda}_{u}(B^1_{u},B^1_{u} +  e^{T_1/2}( z_2 -   z_1)) + V^{1, \lambda / \lambda'}_{u} (B^2_{u} ,B^2_{u}    + e^{T_2/2}( z_1 -  z_2) )\, du \bigg) \nonumber
\\ &\times E_{B^1_1}^{Y^1} \bigg( \exp \bigg(\int_0^{T_1} F(Y^1_u) -  H_{e^u}^{\lambda' / \lambda} (Y^1_u, Y^1_u + e^{\frac{T_1-u}{2}}(z_2 - z_1)) \, du \bigg)\bigg|\, \rho_1 > T_1, Y^1_{T_1} = z_1 \bigg) \nonumber
\\ &  \times E_{B^2_1}^{Y^2} \bigg( \exp \bigg(\int_0^{T_2} F(Y^2_u) - H_{e^u}^{\lambda / \lambda'} (Y^2_u, Y^2_u + e^{\frac{T_2-u}{2}}(z_1 - z_2)) \, du \bigg) \bigg| \, \rho_2 > T_2, Y^2_{T_2} = z_2 \bigg) \nonumber  
\\ & \hspace{10 mm} \times q_{T_1}(B^1_1, z_1) \, q_{T_2}(B^2_1, z_2) \, dm(z_1)\, dm(z_2)\, \bigg) \bigg) 
\\ &=:  \iint (\lambda \lambda')^{2\lambda_0} \bigg[ E^{B}_0 \bigg( E^{B^1, B^2}_{(0,0)} \, \bigg(\,  G(\lambda, \lambda', s, B, B^1, B^2, z_1, z_2)  \,q_{T_1}(B^1_1,z_1) \,q_{T_2}(B^2_1, z_2) \bigg) \bigg)\, \bigg] dm(z_1) \, dm(z_2) . \nonumber 
\end{align}
The conditional probabilities that appear are the same that are defined in Section~\ref{s_OU}, in particular Lemma~\ref{lemma_endpoint_indep}. We have used that the terms in the third and fourth lines are independent conditional on the endpoints. Hereafter, $Y^1$ and $Y^2$, and their respective laws, refer to killed Ornstein-Uhlenbeck processes with killing function $F$. Furthermore, after this point we will suppress the conditioning on $\rho_i > T_i$, as it is implicit in the conditioning $Y^i_{T_i} = z_i$ that $\rho_i > T_i$.\\

We introduce notation for the terms appearing in $G(\lambda, \lambda', s, B, B^1, B^2, z_1, z_2)$. We define
\begin{align} 
&Q(\lambda, \lambda', B^1, B^2,z_1,z_2)  \label{e_Qdef} 
\\ &\hspace{ 10 mm }:= \exp \bigg(- \int_0^1 V^{1, \lambda'/ \lambda}_{u}(B^1_{u},B^1_{u} +  e^{T_1/2}( z_2 -   z_1)) + V^{1, \lambda / \lambda'}_{u} (B^2_{u} ,B^2_{u}    + e^{T_2/2}( z_1 -  z_2) )\, du \bigg),\nonumber \\
&\tilde{Z}^1_{T_1} = \tilde{Z}^1_{T_1}(Y^1, z_1,z_2,\lambda'/\lambda):=\exp \bigg(\int_0^{T_1} F(Y^1_u) -  H_{e^u}^{\lambda' / \lambda} (Y^1_u, Y^1_u + e^{\frac{T_1-u}{2}}(z_2 - z_1)) \, du \bigg), \label{e_Ztilde1def} \\
&\tilde{Z}^2_{T_2} = \tilde{Z}^2_{T_2}(Y^2,z_2,z_1,\lambda/\lambda') := \exp \bigg(\int_0^{T_2} F(Y^2_u) - H_{e^u}^{\lambda / \lambda'} (Y^2_u, Y^2_u + e^{\frac{T_2-u}{2}}(z_1 - z_2)) \, du \bigg).\label{e_Ztilde2def}
\end{align}
We recall that $\Psi^{\lambda, \lambda'}_{B,s} (\sqrt{t-s}z_1 ,\sqrt{t-s} z_2 )$ was defined in \eqref{e_psi}. From \eqref{e_ogmod2} we have
\begin{align} \label{e_Gdef}
&G(\lambda, \lambda',s,B, B^1, B^2, z_1, z_2) \nonumber
\\ &\hspace{4 mm}\, = \Psi^{\lambda, \lambda'}_{B,s} (\sqrt{t-s}z_1 ,\sqrt{t-s} z_2 )\, Q(\lambda,\lambda',B^1,B^2,z_1,z_2)\,E_{B^1_1}^{Y^1} \left( \tilde{Z}^1_{T_1} \big| \,Y^1_{T_1} = z_1 \right)  E_{B^2_1}^{Y^2} \left( \tilde{Z}^2_{T_2} \big| \,Y^2_{T_2} = z_2 \right).
\end{align}
We note that $\tilde{Z}^1_{T_1}$ and $\tilde{Z}^2_{T_2}$ are perturbations of the corresponding $Z^i_{T_i}$ terms. In particular, we defined $Z^i_{T_i}$ by 
\begin{equation} \label{e_Zidef}
Z^i_{T_i}(Y^i) = \exp \bigg( \int_0^{T_i} F(Y^i_u) - V_1^{e^{u/2}} (Y^i_u)\, du \bigg).
\end{equation}
By Proposition \ref{prop_mono_subadd}(a) and \eqref{Hdef}, we have that $H_{e^u}^{c}(x,y) \geq V_1^{e^{u/2}}(x)$, and hence
\begin{equation} \label{e_Ztildebd}
\tilde{Z}^i_{T_i} \leq Z^i_{T_i}(Y^i) \leq C_Z,
\end{equation}
where the second inequality is by \eqref{e_ZTbd}. Using $Q(\lambda,\lambda',B^1,B^2,z_1,z_2) \leq 1$ and $|\Psi^{\lambda, \lambda'}_{B,s}| \leq \|h \|_\infty$, both of which are obvious from these terms' definitions, we therefore obtain that for a constant $C_1>0$,  
\begin{equation} \label{e_Gbound}
|G(\lambda, \lambda',s,B, B^1, B^2, z_1, z_2)| \leq C_1
\end{equation}
uniformly in its arguments. We now define $\Theta(\lambda,\lambda',s,B,B^1,B^2, z_1,z_2)$ as the function in the square-bracketed term in \eqref{e_ogmod2} multiplied by the scaling factor $(\lambda \lambda')^{2\lambda_0}$. That is,
\begin{equation}\label{e_thetadef}
\Theta(\lambda,\lambda',s,B,B^1,B^2,z_1,z_2) :=  G(\lambda, \lambda',s, B, B^1_1, B^2_1, z_1, z_2) \,(\lambda \lambda')^{2\lambda_0} \,q_{T_1}(B^1_1,z_1) \,q_{T_2}(B^2_1, z_2).
\end{equation}
Note from \eqref{e_ogmod2} that
\begin{equation} \label{gammatheta}
\Gamma^{\lambda,\lambda'}(s) = \iint E^{B}_0 ( E^{B^1, B^2}_{(0,0)} \, (\,\Theta(\lambda,\lambda',s,B,B^1,B^2,z_1,z_2) ) ) \,dm(z_1) \,dm(z_2).
\end{equation}
Recall that $T_1 = \log(\lambda^2(t-s))$ and $T_2 = \log(\lambda'^2(t-s))$. Taking $s^*(1/8)$ as in Theorem~\ref{thm_killedOU}(c), we note that if $\lambda,\lambda' > e^{s^*/2}(t-s)^{-1/2}$, then $T_1, T_2 \geq s^*(1/8)$. We define $\bar{\lambda}(s)$ as
\begin{equation} \label{lambdabardef}
\bar{\lambda}(s) := \big[e^{s^*(1/8)/2}(t-s)^{-1/2}\big] \vee 1
\end{equation}
and $\tau(s)$ by
\begin{equation} \label{taudef}
\tau(s) = \log( \bar{\lambda}(s)^2 (t-s)).
\end{equation}
Applying \eqref{OU_qbound} with $\delta = 1/8$, we obtain
\[q_{T_1}(b_1,z_1) \,q_{T_2}(b_2, z_2) \leq C (t-s)^{-2\lambda_0} (\lambda \lambda')^{-2\lambda_0} e^{1/8(b_1^2 + b_2^2 + z_1^2 + z_2^2)}\]
for all $T_1, T_2 > \tau(s)$ (equivalently, $\lambda,\lambda' > \bar{\lambda}(s)$). Using the above and \eqref{e_Gbound}, we obtain
\begin{align} \label{e_thetabd}
& | \Theta(\lambda,\lambda',s,B,B^1,B^2, z_1,z_2) |  \nonumber
\\ &\leq  C \, (t-s)^{-2\lambda_0} \exp \left(\left[(B^1_1)^2 + (B_1^2)^2 + z_1^2 + z_2^2 \right]/8 \right)  \,\,\,\text{  for all }\, \lambda,\lambda' > \bar{\lambda}(s).
\end{align}
Since $B_1^i \sim m$, \eqref{e_thetabd} implies that $\Theta$ has a (uniform in $\lambda, \lambda' > \bar{\lambda}(s)$) upper bound which is integrable with respect to $dP^B_0 dP^{B^1}_0 dP^{B^2}_0 dm(z_1) \, dm(z_2)$. From \eqref{gammatheta}, this implies that $(\lambda \lambda')^{2\lambda_0} \Gamma^{\lambda,\lambda'}(s)$ is bounded for $\lambda,\lambda' > \bar{\lambda}(s)$ (for fixed $s<t$). Moreover, if $\lim_{\lambda, \lambda' \to \infty} \Theta(\lambda,\lambda',s,B,B^1,B^2,z_1,z_2)$ exists for $P^B_0 \otimes P^{B^1,B^2}_{(0,0)}$-a.a.~$\omega$ and Lebesgue-a.a. $z_1,z_2 \in \R$, then by \eqref{gammatheta} and Dominated Convergence (using \eqref{e_thetabd}), we have
\begin{align} \label{gammathetalimit}
\lim_{\lambda,\lambda'\to \infty} (\lambda \lambda')^{2\lambda_0} \Gamma^{\lambda,\lambda'}(s) = \iint  E^B_0 ( E^{B^1,B^2}_{(0,0)} [ \,\lim_{\lambda,\lambda'\to \infty} \Theta(\lambda,\lambda',s,B,B^1,B^2,z_1,z_2) ])\, dm(z_1)\,dm(z_2).
\end{align}
In view of \eqref{e_Gammareduce}, the above implies the following:
\begin{align} \label{lastreduction}
\text{If } \lim_{\lambda,\lambda' \to \infty} &\Theta(\lambda,\lambda',s,B,B^1,B^2,z_2,z_2) \, \text{ exists }P^B_0 \otimes P^{B^1,B^2}_{(0,0)}\text{-a.s. for a.e. } z_1, z_2 \in \R, \text{ then} \nonumber
\\ &\lim_{\lambda,\lambda'\to \infty} \N_0((L_t^\lambda \times L_t^{\lambda'})(h)) = \int_0^t \bigg[\iint  E^B_0 ( E^{B^1,B^2}_{(0,0)} [ \,\lim_{\lambda,\lambda'\to \infty} \Theta(\lambda,\lambda',s,B,B^1,B^2,z_1,z_2) ])\, dm(z_1)\,dm(z_2) \bigg]ds.
\end{align}
As $h\geq 0$, and hence $\Gamma^{\lambda,\lambda'}(s) \geq 0$, the right hand side of the above is equal to the last expression of \eqref{e_Gammareduce} (provided $\Theta$ converges) by Monotone Convergence. Thus it suffices to compute the limit of $\Theta(\lambda,\lambda',s,B,B^1,B^2,z_2,z_2)$ as $\lambda,\lambda' \to \infty$. As we only need to find the limit a.e. in $(z_1, z_2)$, we will hereafter assume that $z_1 \neq z_2$. We also take this opportunity to reiterate our assumptions about $\lambda$ and $\lambda'$. Originally we assumed $\lambda,\lambda' > (t-s)^{-1/2}$; in view of the above, we augment the assumption to $\lambda,\lambda' > \bar{\lambda}(s)$, or equivalently, $T_1, T_2 > \tau(s)$. This implies that $\lambda, \lambda' >1$ and $T_1, T_2 > s^*(1/8)$.\\

$\Theta$ is the product of the function $G$ and the rescaled transition densities, ie. $\lambda^{2\lambda_0} q_{T_1}(B^1_1,z_1)$ and $\lambda'^{2\lambda_0}  q_{T_2}(B^2_1,z_2)$. We will show that both of these approach finite limits as $\lambda, \lambda' \to \infty$. First, let us handle the transition densities. By Lemma~\ref{lemma_densitylimit},
\[\lim_{T_i \to \infty} e^{\lambda_0 T_i} q_{T_i}(B_1^i,z_i) = \psi_0(B_1^i)\, \psi_0(z_i)\]
for $i=1,2$. Using the definitions of $T_1$ and $T_2$ (e.g. $T_1 = \log(\lambda^2(t-s))$), we readily obtain from the above that
\begin{align} \label{densitylimit}
&\lambda^{2\lambda_0}  q_{T_1}(B^1_1,z_1) \to (t-s)^{-\lambda_0} \psi_0(B^1_1) \, \psi_0(z_1) \,\, \text{ as } \, \lambda \to \infty, \,\,\,\, \text{ and}  \nonumber
\\ &\lambda'^{2\lambda_0}  q_{T_2}(B^2_1,z_2) \to (t-s)^{-\lambda_0} \psi_0(B^2_1) \, \psi_0(z_2) \,\, \text{ as } \, \lambda' \to \infty 
\end{align}
for all $B^1_1, B^2_1, z_1,
 z_2 \in \R$.\\

We now compute the limit of $G$. We begin by focussing on the components of $G$ for which the analysis is most technical, which are the conditional expectations of $\tilde{Z}^i_{T_i}$. We will focus on $i=1$, but the analysis carries over to the $i=2$ case. For now, we replace $B^1_1$ with a generic point $x \in \R$. We will show that
\begin{align} \label{e_tildeZlim}
\lim_{\lambda,\lambda' \to \infty} E_{x}^{Y^1} \big( \tilde{Z}^1_{T_1}(Y^1,z_1,z_2,\lambda'/\lambda) \big| \,Y^1_{T_1} = z_1 \big) = E_x^{Y,\infty} (Z_\infty(Y)) \, E_{z_1}^{Y,\infty} (W_\infty(Y,z_2)),
\end{align} 
where we recall that $E_x^{Y,\infty}$ is the expectation under the law of the killed process $Y$ with $Y_0 = x$ conditioned to survive for all time, as defined in Theorem~\ref{thm_killedOU}(e). $Z_T(Y)$ is as defined in \eqref{e_Zdef} and we recall from \eqref{e_Zinf} that $Z_\infty(Y) = \lim_{T\to \infty} Z_T(Y)$ exists and is bounded by $C_Z$. $W_S(Y,z)$ is defined as
\begin{equation} \label{e_Wdef}
W_S(Y,z) = \exp \bigg( \int_0^S F(Y_u) - F_2(Y_u, Y_u - e^{u/2}(z - Y_0)) \, du \bigg).
\end{equation}
The integrand in $W_S(Y,z)$ is negative, which implies that $0< W_S(Y,z) \leq 1$ for all $S>0$, and so $W_\infty(Y,z)$ exists and is bounded by $1$. Heuristically, the $Z_\infty$ term comes from the early part of the integral in $\tilde{Z}^i_{T_i}$, and the $W$ term comes from the tail part, and these two contributions are asymptotically independent. Since the time at which we condition is $T$ and goes to infinity, in the limit the expectations are computed under the measure of the process conditioned to survive forever. Because $z_1$ and $z_2$ are fixed, we will hereafter suppress the dependence of $\tilde{Z}_{T_1}^1$ on them and simply write $\tilde{Z}^1_{T_1}(Y^1,\lambda'/\lambda)$. Moreover, we will only be analysing $Y^1, T_1$ and $\tilde{Z}^1_{T_1}(Y^1,\lambda'/\lambda)$ for the time being, so we simply denote these by $Y,T$ and $\tilde{Z}_T(Y,\lambda'/\lambda)$.\\

Let us now proceed more carefully. Let $0<K< T / 2$. We apply the Markov property to $E_{x}^{Y} ( \tilde{Z}_{T}(Y,\lambda'/\lambda) \big| \,Y_{T} = z_1 )$ at times $K$ and $T - K$ and expand in terms of the joint density of $(Y_K,Y_{T-K})$. As in \eqref{e_bridgetransden1}, the joint density of $(Y_K,Y_{T-K})$ at $(w,y)$ with respect to $m\times m$ under $P^Y_x(Y \in \cdot \,| Y_T = z_1)$ is
\[ \frac{q_K(x,w) q_{T - 2K}(w,y) q_K(y,z_1)}{q_{T}(x,z_1)}. \]
Thus we obtain the following:
\begin{align} \label{e_tildeZbreakdown}
&E_{x}^{Y} \left( \tilde{Z}_T(Y,\lambda'/\lambda) \,\big|\, Y_{T} = z_1 \right) \nonumber  \hspace{ 5mm}
\\ &\hspace{ 5mm}=  E_{x}^{Y} \bigg( \exp \bigg(\int_0^{T} F(Y_u) -  H_{e^u}^{\lambda' / \lambda} (Y_u, Y_u + e^{\frac{T-u}{2}}(z_2 - z_1)) \, du \bigg)\bigg| \,Y_{T} = z_1 \bigg) \nonumber 
\\ &\hspace{ 5mm}=\iint E_x^{Y} \bigg( \exp \bigg(\int_0^K F(Y_u) -  H_{e^u}^{\lambda' / \lambda} (Y_u, Y_u + e^{\frac{T-u}{2}}(z_2 - z_1)) \, du \bigg)\bigg|\, Y_K = w \bigg) \nonumber
\\ &\hspace{ 9mm}\times E_w^{Y} \bigg( \exp \bigg(\int_0^{T-2K} F(Y_u) -  H_{e^{K+u}}^{\lambda' / \lambda} (Y_u, Y_u + e^{\frac{T-K-u}{2}}(z_2 - z_1)) \, du \bigg)\bigg| \,Y_{T-2K} = y \bigg) \nonumber
\\ &\hspace{ 9mm}\times E_y^{Y} \bigg( \exp \bigg( \int_0^K F(Y_u) -  H_{e^{T-K+u}}^{\lambda' / \lambda} (Y_u, Y_u + e^{\frac{K-u}{2}}(z_2 - z_1)) \, du \bigg) \bigg|\, Y_{K} = z_1 \bigg) \nonumber
\\ &\hspace{9 mm} \times \frac{q_K(x,w) q_{T - 2K}(w,y) q_K(y,z_1)}{q_{T}(x,z_1)} \, dm(w)\, dm(y). 
\end{align}
Denote the three conditional expectations by $A_1(x,w, \lambda,\lambda',K), A_2(w,y,\lambda,\lambda',K)$ and $A_3(y,z_1,\lambda,\lambda',K)$. That is,
\begin{align}
A_1(x, w, \lambda,\lambda',K) &=  E_x^{Y} \bigg( \exp \bigg(\int_0^K F(Y_u) -  H_{e^u}^{\lambda' / \lambda} (Y_u, Y_u + e^{\frac{T-u}{2}}(z_2 - z_1)) \, du \bigg)\bigg|\, Y_K = w \bigg) \label{A1def}
\\ A_2(w,y,\lambda,\lambda',K) &= E_w^{Y} \bigg( \exp \bigg(\int_0^{T-2K} F(Y_u) -  H_{e^{K+u}}^{\lambda' / \lambda} (Y_u, Y_u + e^{\frac{T-K-u}{2}}(z_2 - z_1)) \, du \bigg)\bigg| \,Y_{T-2K} = y \bigg) \label{A2def}
\\ A_3(y,z_1,\lambda,\lambda',K) &= E_y^{Y} \bigg( \exp \bigg(\int_0^K F(Y_u) -  H_{e^{T-K+u}}^{\lambda' / \lambda} (Y_u, Y_u + e^{\frac{K-u}{2}}(z_2 - z_1)) \, du \bigg)\bigg| \,Y_{K} = z_1 \bigg). \label{A3def}
\end{align}
We observe that $A_1, A_2$ and $A_3$ all depend on $z_1$ and $z_2$ in addition to their listed arguments, as these values appear in their integrands. As for the time being we are viewing $z_1$ and $z_2$ as fixed, we omit this additional dependence. Noting that the integrand is bounded above by $F(Y_u) - V_1^{e^{u/2}}(Y_u)$ in each case, from \eqref{e_ZTbd} we have $A_i \leq C_Z$ for $i=1,2,3$. In terms of the $A_i$, \eqref{e_tildeZbreakdown} can be rewritten as
\begin{align} \label{e_tildeZbreakdownA}
&E_{x}^{Y} \left( \tilde{Z}_T(Y,\lambda'/\lambda) \, \big|\,  Y_{T} = z_1 \right) \nonumber
\\&= \iint A_1(x,w, \lambda,\lambda',K) \, A_2(w,y,\lambda,\lambda',K)\, A_3(y,z_1,\lambda,\lambda',K) \,  \frac{q_K(x,w)\, q_{T - 2K}(w,y)\, q_K(y,z_1)}{q_{T}(x,z_1)} \, dm(w) \, dm(y).
\end{align}
There are two main contributions in the $A_i$. The first comes from $F$ and the first argument of the $H$ function, and is approximately equal to $F(Y_u) - V_1^{e^{u/2}}(Y_u)$; the second comes from the second argument of the $H$ function. We will see that, asymptotically, $A_1$ is only affected by the first contribution and gives the $Z_\infty(Y)$ term in (\ref{e_tildeZlim}); $A_3$ is only affected by the second contribution and gives the $W_\infty(z_2,\infty)$ term in (\ref{e_tildeZlim}). The contribution of $A_2$ is will be seen to be negligible. We first show that $A_2$ is arbitrarily close to $1$ as $K$ is made large, uniformly in $T$ sufficiently large depending on $K$. Define $Z^a_T(Y,\lambda'/\lambda,K)$ as $\tilde{Z}_T(Y, \lambda'/\lambda,K)$ with $A_2$ replaced by $1$; that is,
\begin{align} \label{Zadef}
Z^a_T(Y,\lambda'/\lambda,K) =\, & \exp \bigg( \int_0^K F(Y_u) - H_{e^u}^{\lambda' / \lambda}(Y_u, Y_u + e^{\frac{T-u}{2}}(z_2-z_1)) \,du \bigg)  \nonumber
\\ &\times \exp \bigg( \int_{T-K}^T F(Y_u) -  H_{e^u}^{\lambda' / \lambda} (Y_u, Y_u + e^{\frac{T-u}{2}}(z_2 - z_1)) \, du \bigg).
\end{align}
As in \eqref{e_tildeZbreakdown} and \eqref{e_tildeZbreakdownA}, we therefore have
\begin{align} \label{Zaex}
&E^Y_{x} \big( Z^a_T(Y,\lambda'/\lambda,K) \, \big|  \,Y_T = z_1 \big) \nonumber
\\ &\hspace{5 mm} = \iint A_1(x,w,\lambda,\lambda',K) A_3(y,z_1,\lambda,\lambda',K)  \frac{q_K(x,w)\, q_{T - 2K}(w,y)\, q_K(y,z_1)}{q_{T}(x,z_1)} \, dm(w) \, dm(y). 
\end{align}
By monotonicity (Proposition~\ref{prop_mono_subadd}(a)) and \eqref{e_VlambdaConvergence} we have
\begin{equation} \label{e_FHbound1}
F(Y_u) -  H_{e^{K+u}}^{\lambda' / \lambda} (Y_u, Y_u + e^{\frac{T-K-u}{2}}(z_2 - z_1)) \leq F(Y_u) - V_1^{e^{(K+u)/2}}(Y_u) \leq C e^{-(K+u)(2\lambda_0 -1)/2}
\end{equation}
uniformly in $T > 2K$. Integrating this over $u$ shows that the exponent in $A_2$ is bounded above $C' e^{-(2\lambda_0-1)K/2}$ for a constant $C'$, uniformly in $T>2K$. We choose $K$ large enough so that exponent in $A_2$ is smaller than $2$. Then by \eqref{e_tildeZbreakdownA} and \eqref{Zadef}, applying the mean value theorem, we have
\begin{align} \label{exp_A2error}
&\big| E^Y_{x} \big( \tilde{Z}_T(Y,\lambda'/\lambda) - Z^a_T(Y,\lambda'/\lambda,K) \, \big|  \,Y_T = z_1 \big) \big| \nonumber \hspace{4 mm}
\\ & \hspace{5 mm}\leq \frac{1}{q_{T}(x,z_1)}  \iint A_1(x,w, \lambda,\lambda',K) \, \bigg| A_2(w,y,\lambda,\lambda',K) - 1 \bigg| A_3(y,z_1,\lambda,\lambda',K) \nonumber
\\ &\hspace{26 mm }\times q_K(x,w) q_{T - 2K}(w,y) q_K(y,z_1) \,  dm(w) \, dm(y) \nonumber
\\ & \hspace{5 mm}\leq  \frac{e^2 C_Z^2}{q_{T}(x,z_1)}  \iint  E_w^{Y} \bigg(  \int_0^{T-2K} |F(Y_u) -  H_{e^{K+u}}^{\lambda' / \lambda} (Y_u, Y_u + e^{\frac{T-K-u}{2}}(z_2 - z_1))| \, du\, \bigg| \,Y_{T-2K} = y \,\bigg)  \nonumber
\\ &\hspace{26 mm} \times  q_K(x,w) q_{T - 2K}(w,y) q_K(y,z_1)  \, dm(w) \, dm(y),
\end{align}
uniformly for all $T > 2K$, where we have also used $A_1 A_3 \leq C_Z^2$. The term in the absolute value inside the integral can be positive or negative; (\ref{e_FHbound1}) provides an upper bound for $F- H_{e^{K+u}}^{\lambda' / \lambda}$. To obtain a lower bound, we note that $H_{e^{K+u}}^{\lambda' / \lambda}(a,b) \leq F_2(a,b) \leq F(a) + F(b) $ by Proposition~\ref{prop_mono_subadd} (using part (a) and then part (b)). Using this bound implies that
\begin{equation} \label{e_FHbound2}
F(Y_u) -  H_{e^{K+u}}^{\lambda' / \lambda} (Y_u, Y_u + e^{\frac{T-K-u}{2}}(z_2 - z_1)) \geq -F(Y_u + e^{\frac{T-K-u}{2}}(z_2 - z_1)).
\end{equation}
Together, (\ref{e_FHbound1}) and (\ref{e_FHbound2}) imply that the absolute value appearing in the integral in (\ref{exp_A2error}) is bounded above by 
\[C e^{-(K+u)(2\lambda_0 -1)/2} + F(Y_u + e^{\frac{T-K-u}{2}}(z_2 - z_1)). \]
We have already noted that when integrated over $u$, the first term is bounded by $C'e^{-K(2\lambda_0 - 1)/2}$ (uniformly in $T$). The first term has no dependence on the spatial parameters $w$ and $y$, so in \eqref{exp_A2error} the transition densities and can be integrated and cancelled with the denominator. We get that for all $T>2K$, (\ref{exp_A2error}) is bounded above by
\begin{align} 
&C'e^{-K(2\lambda_0 - 1)/2} + \frac{C}{q_{T}(x,z_1)}  \iint  E_w^Y \bigg(  \int_0^{T-2K} F( Y_u + e^{\frac{T-K-u}{2}}(z_2 - z_1)) \, du \, \bigg|\, Y_{T-2K} = y \bigg) \nonumber
\\ &\hspace{30 mm} \times  q_K(x,w) q_{T - 2K}(w,y) q_K(y,z_1)  \, dm(w) \, dm(y).  \nonumber
\end{align}
We consider the time reversed process in the above and apply \eqref{reversible}, which implies that the above is equal to, and hence for all $T>2K$, \eqref{exp_A2error} is bounded above by
\begin{align}\label{exp_A2error2}
&C'e^{-K(2\lambda_0 - 1)/2} +\frac{C}{q_{T}(x,z_1)} \iint  E_y^Y \bigg(  \int_0^{T-2K} F( Y_u + e^{\frac{K+u}{2}}(z_2 - z_1)) \, du \, \bigg|\, Y_{T-2K} = w \bigg)  \nonumber
\\ &\hspace{30 mm} \times  q_K(x,w) q_{T - 2K}(w,y) q_K(y,z_1)  \, dm(w) \, dm(y).
\end{align}

We recall the asymptotic behaviour of $F$ from \eqref{e_FODE}(iii), ie. that $F(x) \sim c_1 |x| e^{-x^2/2}$ as $|x| \to \infty$. This implies there is a constant $c_2 >0$ such that 
\begin{equation} \label{Ftailbd}
F(x) \leq c_2(1+ |x|) e^{-x^2/2} \,\,\, \text{ for all } \,\, x\in\R.
\end{equation}
In order for this to give a useful upper bound in \eqref{exp_A2error2}, we'll need to show that the argument of $F$ is large in absolute value. It is enough to show that $|Y_u| \ll e^{\frac{K+u}{2}} |z_2 - z_1|$ with high probability when conditioned on its endpoint. Recall that we have assumed $z_1 \neq z_2$. We bound the integrand over the two cases mentioned above and exchange the integral and expectation, which is justifiable since $F$ is positive.  We have
\begin{align} \label{exp_A2error3}
&E_y^{Y} \bigg(  \int_0^{T-2K} F( Y_u + e^{\frac{K+u}{2}}(z_2 - z_1)) \, du \, \bigg|\, Y_{T-2K} = w \bigg) \nonumber
\\ &\hspace{5 mm}\leq E_y^{Y} \bigg( \int_0^{T-2K} F(e^{\frac{K+u}{4}}|z_2 - z_1|) + F(0) 1(|Y_u| \geq e^{\frac{K+u}{4}}|z_2-z_1|) \, du \, \bigg|\, Y_{T-2K} = w \bigg) \nonumber
\\ &\hspace{5 mm}\leq c_2 \int_0^\infty (1+ e^{\frac{K+u}{4}}|z_2-z_1|) e^{-e^{\frac{K+u}{2}}|z_2-z_1|^2/2} du + F(0) \int_0^{T-2K} P_y^{Y}(|Y_u| \geq e^{\frac{K+u}{4}}|z_2-z_1| \, \big| \, Y_{T-2K} = w) \, du,
\end{align}
where we have used \eqref{Ftailbd} and the fact that $F$ is radially decreasing. A simple substitution shows that
\begin{align} \label{exp_Gtail}
c_2 \int_0^\infty (1+ e^{\frac{K+u}{4}}|z_2-z_1|) e^{-e^{\frac{K+u}{2}}|z_2-z_1|^2/2} du &\leq 4c_1 \int_{e^{K/4}|z_2-z_1|}^\infty (1 + a^{-1})e^{-a^2/2} da \nonumber
\\ &\leq C \int_{e^{K/4}|z_2-z_1|}^\infty e^{-a^2/2} da + C 1(e^{K/4}|z_2 - z_1| < 1) \int_{e^{K/4}|z_2-z_1|}^1 a^{-1}e^{-a^2/2} da \nonumber
\\ &\leq C \int_{e^{K/4}|z_2-z_1|}^\infty e^{-a^2/2} da - C  \left[ \log(e^{K/4}|z_2-z_1|)  \wedge 0 \right].
\end{align}
To bound the second term in (\ref{exp_A2error3}) we expand the probability of the large excursion in terms of the transition densities. There are two cases, which we handle in the following lemma. In what follows, $s^* = s^*(1/8)$ from Theorem~\ref{thm_killedOU}(c).
\begin{lemma} \label{lemma_bridgetail}
Let $M>0$ andand $w,y \in \R$. \\
(a) There is a constant $C>0$ such that for $S,u>0$
 satisfying $u,S-u \geq s^*$,
\begin{equation} 
P_y^Y(|Y_u| \geq M \, | \, Y_{S} = w) \leq \frac{C}{q_S(y,w)} e^{-\lambda_0 S} e^{(y^2 + w^2)/8} \bigg[ \frac{e^{-M^2 / 4}}{M} \wedge 1 \bigg]. \nonumber
\end{equation}
(b) For fixed $u_0>0$ the families
\[ \{ P_y^Y(Y_u \in \cdot \, | \, Y_{S} = w) : S\geq u_0, 0\leq u\leq u_0 \} \, \text{ and } \,\{ P_y^Y(Y_{S-u} \in \cdot \, | \, Y_{S} = w) : S \geq u_0, 0 \leq u \leq u_0 \} \]
are tight. 
\end{lemma}
\begin{proof}
To see (a), we use \eqref{e_bridgetransden1} and \eqref{OU_qbound} with $\delta = 1/8$ to obtain that for $u, S-u \geq s^*$,
\begin{align}  
P_y^Y(Y_u \geq M \, \big| \, Y_{S} = w) &= \int_M^\infty \frac{q_u(y,a) q_{S-u}(a,w)}{q_S(y,w)} dm(a) \nonumber
\\ &\leq \frac{c_{1/8}^2}{q_S(y,w)} e^{-\lambda_0 S}\int_M^\infty e^{(y^2 + 2a^2 + w^2)/8} \, dm(a) \nonumber
\\ &\leq \frac{C}{q_S(y,w)} e^{-\lambda_0 S} e^{(y^2 + w^2)/8} \bigg[\frac{e^{-M^2/4}}{M} \wedge 1 \bigg], \nonumber
\end{align}
where the last line uses a standard upper bound on Gaussian tails and bounds the integral above by a constant when $M$ is small. The bound for $Y_u < -M$ is the same. The first family in part (b) is tight as a consequence of Lemma \ref{lemma_endpoint_indep}(c). To see that the second family is tight we consider the time reversal of $Y$ and use \eqref{reversible}, from which tightness now also follows from Lemma \ref{lemma_endpoint_indep}(c).
\end{proof}

Applying Lemma \ref{lemma_bridgetail}(a), using (\ref{exp_Gtail}) and separating the integrals depending if $u, S-u \geq s^*$ or not, we have that (\ref{exp_A2error3}) is bounded above by
\begin{align}  \label{exp_A2error8}
&C \int_{e^{K/4}|z_2-z_1|}^\infty e^{-a^2/2} da - C  \left[ \log(e^{K/4}|z_2-z_1|)  \wedge 0 \right] \nonumber
\\  &+ \frac{C}{q_{T - 2K}(y,w)} e^{-\lambda_0 (T- 2K)} e^{y^2/8}e^{w^2/8} \int_{s^*}^{T - 2K - s^*} \left[ \frac{e^{-e^{\frac{K+u}{2}}|z_2-z_1|^2 / 4}}{e^{\frac{K+u}{4}}|z_2-z_1|} \wedge 1 \right] du \nonumber
\\ & + C \bigg( \int_0^{s^*} + \int_{T - 2K - s^*}^{T - 2K} \bigg) P_y^{Y}(|Y_u| \geq e^{\frac{K+u}{4}}|z_2-z_1| \, \big| \, Y_{T-2K} = w)\, du. 
\end{align}
As the above is an upper bound for the expectation appearing in the second term of \eqref{exp_A2error2}, and \eqref{exp_A2error2} is an upper bound for \eqref{exp_A2error}, we have
\begin{align} \label{exp_A2error33}
\big| E^Y_{x} \big( \tilde{Z}_T(&Y,\lambda'/\lambda) - Z^a_T(Y,\lambda'/\lambda,K) \, \big|  \,Y_T = z_1 \big) \big| \nonumber
\\ &\leq C e^{-K(2\lambda_0 - 1)/2} + \frac{C}{q_{T}(x,z_1)}  \iint \bigg[  \int_{e^{K/4}|z_2-z_1|}^\infty e^{-a^2/2} da - C  \left[ \log(e^{K/4}|z_2-z_1|)  \wedge 0 \right] \nonumber
\\  &+ \frac{1}{q_{T - 2K}(y,w)} e^{-\lambda_0 (T- 2K)} e^{y^2/8}e^{w^2/8} \int_{s^*}^{T - 2K - s^*} \left[ \frac{e^{-e^{\frac{K+u}{2}}|z_2-z_1|^2 / 4}}{e^{\frac{K+u}{4}}|z_2-z_1|} \wedge 1 \right] du \nonumber
\\ & + \bigg( \int_0^{s^*} + \int_{T - 2K - s^*}^{T - 2K} \bigg) P_y^{Y}(|Y_u| \geq e^{\frac{K+u}{4}}|z_2-z_1| \, \big| \, Y_{T-2K} = w)\, du \bigg] \nonumber 
\\ &\times  q_K(x,w) q_{T - 2K}(w,y) q_K(y,z_1)  \, dm(w) \, dm(y). 
\end{align}
Note that the first two terms in the integral with respect to $y$ and $w$ are independent of these variables. We can therefore integrate them out; using the fact that
\[ \iint q_K(x,w) q_{T - 2K}(w,y) q_K(y,z_1)  \, dm(w) \, dm(y) = q_T(x,z_1) \]
(and an obvious cancellation) we obtain that
\begin{align} \label{exp_A2error333}
\big| E^Y_{x} \big( &\tilde{Z}_T(Y,\lambda'/\lambda) - Z^a_T(Y,\lambda'/\lambda,K) \, \big|  \,Y_T = z_1 \big) \big| \nonumber
\\ &\leq Ce^{-K(2\lambda_0 - 1)/2} + C\int_{e^{K/4}|z_2-z_1|}^\infty e^{-a^2/2} da - C  \left[ \log(e^{K/4}|z_2-z_1|)  \wedge 0 \right] \nonumber
\\ &+  \frac{C}{q_{T}(x,z_1)} e^{-\lambda_0 (T- 2K)} \iint e^{y^2/8}e^{w^2/8}  q_K(x,w) q_K(y,z_1)  \, dm(w) \, dm(y) \bigg( \int_{s^*}^{T - 2K - s^*} \left[ \frac{e^{-e^{\frac{K+u}{2}}|z_2-z_1|^2 / 4}}{e^{\frac{K+u}{4}}|z_2-z_1|} \wedge 1 \right] du \bigg) \nonumber
\\&+ \frac{C}{q_{T}(x,z_1)} \iint \bigg[ \bigg( \int_0^{s^*} + \int_{T - 2K - s^*}^{T - 2K} \bigg) P_y^{Y}(|Y_u| \geq e^{\frac{K+u}{4}}|z_2-z_1| \, \big| \, Y_{T-2K} = w)\, du \bigg] \nonumber
\\&\hspace{10 mm} \times q_K(x,w) q_{T - 2K}(w,y) q_K(y,z_1)  \, dm(w) \, dm(y) \nonumber 
\\ &=: \delta_1 + \delta_2 + \delta_3 + \delta_4 + \delta_5,
\end{align}
where $\delta_i = \delta_i(T,K,z_1,z_2)$. We first note that \begin{equation}
\delta_i(T,K,z_1,z_2) \to 0 \text{  as  } K \to \infty \text{ (uniformly in }T\geq 2K\text{) for  } i=1,2,3. \nonumber
\end{equation}
Turning to $\delta_4$ and $\delta_5$, we observe that by Lemma~\ref{lemma_densitylimit}, $e^{\lambda_0 T} q_T(x,z_1) \to \psi_0(x) \psi_0(z_1)$ as $\lambda \to \infty$. Since $T \to q_T(x,z_1)$ is continuous, $q_T(x,z_1)>0$ for all $T \geq \tau(s)$ and $\psi_0(x) \psi_0(z_1) >0$, this implies that there exists $\beta(x,z_1) = \beta > 0$ such that
\begin{equation} \label{e_qdenombd}
q_T(x,z_1) \geq \beta e^{-\lambda_0 T} \psi_0(x) \psi_0(z_1) \hspace{ 6 mm} \forall \,\,T \geq \tau(s).
\end{equation}
Applying \eqref{OU_qbound} twice with $\delta = 1/8$ and using \eqref{e_qdenombd}, we have
\begin{align}
|\delta_4 (T,&K,z_1,z_2)|  \nonumber
\\ &\leq \frac{C \beta^{-1}}{\psi_0(x)\psi_0(z_1)} e^{2\lambda_0 K} e^{-2\lambda_0 K} e^{(x^2+z_1^2)/8} \iint e^{y^2/4} e^{w^2/4} dm(w) \,dm(y) \bigg( \int_{s^*}^{T - 2K - s^*} \frac{e^{-e^{\frac{K+u}{2}}|z_2-z_1|^2 / 4}}{e^{\frac{K+u}{4}}|z_2-z_1|} \, du \bigg)  \nonumber
\\ &\leq \frac{C e^{(x^2+z_1^2)/8}}{\psi_0(x)\psi_0(z_1)} \bigg( \int_{s^*}^{T - 2K - s^*} \frac{e^{-e^{\frac{K+u}{2}}|z_2-z_1|^2 / 4}}{e^{\frac{K+u}{4}}|z_2-z_1|} \, du \bigg) \nonumber
\\ &= \frac{4C e^{(x^2+z_1^2)/8}}{ \psi_0(x)\psi_0(z_1)} \int_{e^{s^*/4}e^{K/4}|z_2-z_1|}^{\infty} \frac{e^{-a^2 /4}}{a^2} \, da,
\end{align} 
where the last line follows from a simple substitution. Thus we have $\delta_4 (T,K,z_1,z_2) \to 0$ as $K \to \infty$, and again we note that convergence is uniform in $T > 2K$. It remains to handle $\delta_5$. By three applications of \eqref{OU_qbound} with $\delta = 1/8$ and \eqref{e_qdenombd}, we have
\begin{align}
|\delta_5 (T,&K,z_1,z_2)|  \nonumber
\\ \leq \, &\frac{C}{\psi_0(x)\psi_0(z_1)} \iint \bigg[ \bigg( \int_0^{s^*} + \int_{T - 2K - s^*}^{T - 2K} \bigg) P_y^{Y}(|Y_u| \geq e^{\frac{K+u}{4}}|z_2-z_1| \, \big| \, Y_{T-2K} = w)\, du \bigg] \nonumber
\\&\times e^{x^2/8} e^{w^2/4} e^{y^2/4} e^{z_1^2/8} \, dm(w) \, dm(y) \nonumber.
\end{align}
The square bracketed term vanishes as $K \to \infty$ uniformly in $T \geq 2K + s^*(1/8)$ by Lemma~\ref{lemma_bridgetail}(b). The probabilities are bounded so the integrand obviously has a uniformly integrable upper bound. By Dominated Convergence, we have that $\delta_5 (T,K,z_1,z_2) \to 0$ as $K \to \infty$, uniformly in $T \geq 2K + s^*(1/8)$. We have therefore shown that $\sum_{i=1}^5 \delta_i(T,K,z_1,z_2)$ is arbitrarily small as $K \to \infty$, uniformly in $T\geq 2K + s^*(1/8)$ and in $\lambda'>\bar{\lambda}(s)$, where we recall that we have assumed $\lambda,\lambda' > \bar{\lambda}(s)$. From \eqref{taudef}, $\lambda > \bar{\lambda}(s)$ is equivalent to $T > \tau(s)$. As $\tau(s) \geq s^*(1/8)$, $T\geq 2K + \tau(s)$ implies that $T \geq 2K + s^*(1/8)$. Thus by \eqref{exp_A2error333} we have proved the following. Recall that $\tilde{Z}_T(Y,\lambda'/\lambda) = \tilde{Z}_T(Y,z_1,z_2,\lambda'/\lambda)$.
\begin{lemma} \label{lemma_Zalimit} For all $x, z_1,z_2 \in \R$ such that $z_1 \neq z_2$, for all $K>0$,
\[\delta^{\sim}_a(K) = \, \sup_{T\geq 2K + \tau(s), \,\lambda'>\bar{\lambda}(s)} \big|E^Y_{x} \big(\tilde{Z}_T(Y,\lambda'/\lambda) - Z^a_T(Y,\lambda'/\lambda,K) \, \big|  \,Y_T = z_1 \big)\big| \to 0 \,\text{ as } \,K \to \infty.  \]
\end{lemma}

Given this Lemma, it suffices to find the limit of (the conditional expectation of) $Z^a_T(Y,\lambda'/\lambda,K)$, and so $A_2$ has been replaced by $1$. \\

Next we consider $A_3(y,z_1,\lambda,\lambda',K)$, which we recall from \eqref{A3def} is defined as
\[  E_y^{Y} \bigg( \exp \bigg( \int_0^{K} F(Y_u) -  H_{e^{T-K+u}}^{\lambda' / \lambda} (Y_u, Y_u + e^{\frac{K-u}{2}}(z_2 - z_1)) \, du \bigg) \bigg| \,Y_{K} = z_1 \bigg). \]
We will show that in the limit as $\lambda, \lambda' \to \infty$, the integrand will be $F-F_2$. Define $A^*_3(y,z_1,K)$ by
\begin{equation} \label{A3stardef}
A^*_3(y,z_1,K) = E_y^{Y} \bigg( \exp \bigg( \int_0^K F(Y_u) -  F_2(Y_u, Y_u + e^{\frac{K-u}{2}}(z_2 - z_1)) \, du \bigg) \bigg| \,Y_{K} = z_1 \bigg).
\end{equation} 
The difference between the integrands of $A_3$ and $A_3^*$ is equal to $(F_2 - H_{e^{T-K+u}}^{\lambda' / \lambda})(Y_u, Y_u + e^{\frac{K-u}{2}}(z_2 - z_1)) $, which we bound below by monotonicity and above via Lemma \ref{lemma_V2pt_ROC}. We have
\begin{align} \label{exp_A3error1}
0 \leq (F_2 - H_{e^{T-K+u}}^{\lambda' / \lambda})(Y_u, Y_u + e^{\frac{K-u}{2}}(z_2 - z_1))
& \leq C \bigg[e^{-(T-K+u)(2\lambda_0 - 1)/2} + \left(\frac{\lambda'}{\lambda}\right)^{-(2\lambda_0 - 1)}e^{-(T-K+u)(2\lambda_0 - 1)/2} \bigg] \nonumber
\\&\leq Ce^{(K-u)(2\lambda_0-1)/2}  (t-s)^{-(2\lambda_0-1)/2} \left[\lambda^{-(2\lambda_0 - 1)} + \lambda'^{-(2\lambda_0-1)}\right].
\end{align}
The first line uses the definition of $H^c$, which we recall from \eqref{Hdef}, and in the second line we have used that $T = \log(\lambda^2 (t-s))$. Since $\lambda,\lambda' > \bar{\lambda}(s) \geq 1$, the last expression in (\ref{exp_A3error1}) is bounded by $Ce^{K/2} (t-s)^{-(2\lambda_0-1)/2}$ for all $u \in [0,K]$. Thus by using $|e^x - e^y| \leq (e^x \vee e^y)|x-y|$ and \eqref{exp_A3error1}, we have
\begin{align}\label{exp_A3error2}
|A^*_3(y,&z_1,K) - A_3(y,z_1,\lambda,\lambda',K)| \nonumber
\\ &\leq  \exp \left( CK e^{K/2} (t-s)^{-(2\lambda_0 - 1)/2}\right) E^Y_{y} \bigg( \int_0^K (F_2- H_{e^{T-K+u}}^{\lambda'/\lambda})(Y_u, Y_u + e^{K/2}(z_2 - z_1)) \, du \, \bigg| \,Y_{K} = z_1 \bigg)  \nonumber
\\ &\leq \exp \left(CK e^{K/2} (t-s)^{-(2\lambda_0 - 1)/2} \right) (t-s)^{-(2\lambda_0 - 1)/2} \left[\lambda^{-(2\lambda_0 - 1)} + \lambda'^{-(2\lambda_0-1)}\right] \int_0^K  Ce^{(K-u)(2\lambda_0-1)/2}  du \nonumber
\\ &\leq C(K,t-s) \left[\lambda^{-(2\lambda_0 - 1)} + \lambda'^{-(2\lambda_0-1)}\right]
\end{align}
for some constant $C(K,t-s)>0$. Define $Z^{b}_T(Y,\lambda'/\lambda,K)$ as we defined $Z^a_T(Y,\lambda'/\lambda,K)$ in \eqref{Zadef} but with $F-F_2$ replacing the integrand in the second term. That is,
\begin{align} \label{Zbdef}
Z^b_T(Y,\lambda'/\lambda,K) =\, & \exp \bigg( \int_0^K F(Y_u) - H_{e^u}^{\lambda' / \lambda}(Y_u, Y_u + e^{\frac{T-u}{2}}(z_2-z_1)) \,du \bigg)  \nonumber
\\ &\times \exp \bigg( \int_{T-K}^T F(Y_u) -  F_2 (Y_u, Y_u + e^{\frac{T-u}{2}}(z_2 - z_1)) \, du \bigg).
\end{align}
In particular, we have
\begin{align} \label{Zbex}
&E_x^Y \big(Z^{b}_T(Y,\lambda'/\lambda,K) \, \big| \, Y_T = z_1 \big) \nonumber  
\\ &\hspace{ 5mm}= \iint A_1(x,w, \lambda,\lambda',K) A^*_3(y,z_1,K) \frac{q_K(x,w)\, q_{T - 2K}(w,y)\, q_K(y,z_1)}{q_{T}(x,z_1)} \, dm(w) \, dm(y). 
\end{align}
Because \eqref{exp_A3error2} is uniform in $y$ and $z_1$ and $|A_1| \leq C_Z$, we can integrate out the transition densities to obtain the following.
\begin{lemma} \label{lemma_Zblimit} For $K>0$ and $s \in [0,t)$, there is a constant $C(K,t-s)$ such that
\begin{equation} 
\delta^a_b(K,\lambda,\lambda') = \big| E^Y_x \big( Z^{a}_T(Y,\lambda'/\lambda,K) - Z^{b}_T(Y,\lambda'/\lambda,K) \, \big| \, Y_T = z_1 \big)   \big| \leq C(K,t-s)\left[\lambda^{-(2\lambda_0 - 1)} + \lambda'^{-(2\lambda_0-1)}\right] \nonumber
\end{equation}
for all $\lambda, \lambda' > \bar{\lambda}(s)$.
\end{lemma}
We now analyse $A^*_3$ in greater detail. In particular, we perform a time reversal on the process $Y$. By \eqref{reversible}, we have
\begin{equation}
A^*_3(y,z_1,K) = E_{z_1}^{Y} \bigg( \exp \bigg( \int_0^K F(Y_u) -  F_2 (Y_u, Y_u + e^{\frac{u}{2}}(z_2 - z_1)) \, du \bigg) \bigg| \, Y_{K} = y \bigg). \nonumber
\end{equation}
This is the term that in \eqref{e_tildeZlim} we claimed converges to $E_{Z_1}^{Y,\infty}(W_\infty(Y,z_2))$, defined in \eqref{e_Wdef}, in the limit. However, the above expectation is still conditional on the endpoint. We now show that the contribution from the tail of the integral is vanishing, making the quantity asymptotically independent of the endpoint $y$. Let $0<M<K$. Define $A^{*}_3(y,z_1,M,K)$ by truncating the integral in \eqref{A3stardef} at time $M$. That is,
\begin{equation} \label{A3stardef}
A^*_3(y,z_1,M,K) = E_{z_1}^{Y} \bigg( \exp \bigg( \int_0^M F(Y_u) -  F_2 (Y_u, Y_u + e^{\frac{u}{2}}(z_2 - z_1)) \, du \bigg) \bigg| \, Y_{K} = y \bigg). \nonumber
\end{equation}
We now define $Z^c(Y,\lambda'/\lambda,M,K)$ by truncating the corresponding integral in $Z^b(Y,\lambda'/\lambda,K)$ (the integral over $[T-K,T]$ in \eqref{Zbdef} becomes the integral over $[T-M,T]$) so that $A_3^*(y,z_1,M,K)$ replaces $A_3^*(y,z_1,K)$ in the conditional expectation.
\begin{lemma} \label{lemma_Zclimit} For all $x,  z_1, z_2 \in \R$ such that $z_1 \neq z_2$, 
\begin{equation}
\delta^b_c(M) = \,\sup_{K \geq  M + s^*(1/8)} \, \sup_{T \geq 2K + \tau(s), \,\lambda' > \bar{\lambda}(s)} \big| E^Y_{x} \big(Z^{b}_T(Y,\lambda'/\lambda,K) - Z^{c}_T(Y,\lambda'/\lambda,M,K)\, \big|  \,Y_T = z_1 \big) \big| \to 0 \, \text{ as } M \to \infty. \nonumber
\end{equation}
\end{lemma}
\begin{proof}
Using the inequality $|e^{-x} - e^{-y}| \leq |x-y|$ for $x,y \geq 0$, we have
\begin{equation} \label{exp_A3errorX}
|A^*_3(y,z_1,K) - A^*_3(y,z_1,M,K)| \leq E_{z_1}^{Y} \bigg( \int_M^K | F(Y_u) -  F_2 (Y_u, Y_u + e^{u/2}(z_2 - z_1))|\,du \, \bigg|\, Y_{K} = y \bigg).
\end{equation}
By Proposition~\ref{prop_mono_subadd}(b), the absolute value of the above integrand is at most $F(Y_u + e^{\frac{u}{2}}(z_2 - z_1))$ (for a similar argument see \eqref{e_FHbound2}). Exchanging expectation and integration, we proceed as in (\ref{exp_A2error3}), (\ref{exp_Gtail}), and \eqref{exp_A2error8}, and apply Lemma \ref{lemma_bridgetail} to bound (\ref{exp_A3errorX}) above by
\begin{align} \label{exp_A3errorY}
&c_1 \int_0^{K-M} (1+e^{\frac{M+u}{4}}|z_2-z_1|) e^{-e^{\frac{M+u}{2}}|z_2-z_1|^2/2} + F(0) \int_M^{K} P_{z_1}^{Y}(|Y_u| \geq e^{\frac{u}{4}}|z_2-z_1| \, \big| \, Y_{K} = y) \, du \nonumber
\\ \leq \,&4c_1 \int_{e^{M/4}|z_2-z_1|}^{\infty}(1+ a^{-1}) e^{-a^2/2}\,da + \frac{C e^{-\lambda_0 K}}{q_{K}(z_1,y)} e^{(z_1^2 + y^2)/8} \int_0^{K-M-s^*} \left[ \frac{e^{-e^{\frac{M+u}{2}}|z_2-z_1|^2}}{e^{\frac{M+u}{4}}|z_2-z_1|} \wedge 1 \right] du \nonumber
\\ &\hspace{10 mm} + C \int_{K-s^*}^K P_{z_1}^{Y}(|Y_u| \geq e^{\frac{M+u}{4}}|z_2-z_1| \, \big| \, Y_{K} = y)\, du.
\end{align}
Expanding in terms of transition densities and using $|A_1| \leq C_Z$, we have
\begin{align} 
\big| E^Y_x \big( Z^{b}_T(&Y,\lambda'/\lambda,K) - Z^{c}_T(Y,\lambda'/\lambda,M,K) \, \big| \, Y_T= z_1 \big) \big| \nonumber
\\  \leq \,&\frac{C}{q_T(x,z_1)} \int |A_3(y,z_1,\lambda,\lambda',K) - A^*_3(y,z_1,M,K)| q_{T-K}(x,y)\, q_K(y,z_1) \,dm(y).  \nonumber
\end{align}
Using \eqref{exp_A3errorY} as an upper bound for the integrand, we obtain an expression which closely resembles \eqref{exp_A2error333}; in particular, four terms appear, directly corresponding to $\delta_2,\delta_3, \delta_4$ and $\delta_5$ of that expression. Moreover, they can be handled using the exact same arguments, as in the proof of Lemma \ref{lemma_Zalimit}, but with $(M,K)$ playing the roles of $(K,T)$. Because the arguments are the same, we omit them.
\end{proof}

We now comment on the limit of $A^*_3(w,M,K)$ as $K \to \infty$. Recalling \eqref{A3stardef} and the definition of $W_M$ in \eqref{e_Wdef}, we have
\begin{equation*}
A^*_3(y,z_1,M,K) = E_{z_1}^{Y} \bigg( \exp \bigg( \int_0^M F(Y_u) -  F_2 (Y_u, Y_u + e^{u/2}(z_2 - z_1)) \, du \bigg) \bigg| Y_{K} = y \bigg) = E_{z_1}^{Y} ( W_M(Y,z_2) \, \big| \, Y_K = y ).
\end{equation*} 
The functional $W_M(Y,z_2)$ is a bounded continuous function of $Y\restrict{[0,M]}$. By Lemma \ref{lemma_endpoint_indep}(b), we have
\begin{equation} \label{e_endpointindep1}
\forall \, M>0, \,\, \lim_{K \to \infty} A^*_3(y,z_1,M,K) = E_{z_1}^{Y,\infty} \left( W_M(Y,z_2)\right). 
\end{equation}
We define $Z^d_T(Y,\lambda'/\lambda, M,K)$ by
\begin{equation} \label{Zddef}
Z^d_T(Y,\lambda'/\lambda,M,K) := \exp \bigg(\int_0^{K} F(Y_u) -  H_{e^{K+u}}^{\lambda' / \lambda} (Y_u, Y_u + e^{\frac{T-u}{2}}(z_2 - z_1)) \, du \bigg) \times E_{z_1}^{Y,\infty} \left( W_M(Y,z_2)\right).
\end{equation}
Note that the second term is now deterministic; it no longer depends on the original Ornstein-Uhlenbeck process $Y$ or the spatial variable $y$. We then have
\begin{equation} 
E^Y_x \big( Z^d_T(Y,\lambda'/\lambda,M,K)\, \big| \, Z_T = z_1 \big) = \int A_1(x,w, \lambda,\lambda',K) \frac{q_K(x,w)\, q_{T-K}(w,z_1)}{q_T(x,z_1)} \, dm(w) \times E_{z_1}^{Y,\infty} \left( W_M(Y,z_2)\right). \nonumber
\end{equation}
Bounding $A_1 \leq C_Z$ and integrating out the transition densities, by \eqref{e_endpointindep1} we obtain the following.
\begin{lemma} \label{lemma_Zdlimit} For all $x,z_1,z_2 \in \R$,
\begin{equation}
\delta^c_d(M,K) = \, \sup_{T \geq 2K + \tau(s),\, \lambda' > \bar{\lambda}(s)} \big| E^Y_{x} \big( Z^{c}_T(Y,\lambda'/\lambda,M,K) - Z^d_T(Y,\lambda'/\lambda,M,K) \, \big|  \,Y_T = z_1 \big) \big| \to 0 \, \text{ as } \,K \to \infty \nonumber
\end{equation}
for each fixed $M>0$.
\end{lemma}
From our starting expression for $\tilde{Z}_T(Y,\lambda'/\lambda)$ in \eqref{e_tildeZbreakdownA}, all that remains to be handled in $Z^d_T(Y,\lambda'/\lambda,M,K)$ is the $A_1$ term, whose definition we recall from \eqref{A1def} is
\begin{equation}
A_1(x,w, \lambda,\lambda',K) =  E_x^{Y} \bigg( \exp \bigg(\int_0^K F(Y_u) -  H_{e^u}^{\lambda' / \lambda} (Y_u, Y_u + e^{\frac{T-u}{2}}(z_2 - z_1)) \, du \bigg)\bigg| \,Y_K = w \bigg). \nonumber
\end{equation}

The dominant contribution to the integral in $A_1$ resembles $F(Y_u) - V_1^{e^{u/2}}(Y_u)$. By Proposition~\ref{prop_mono_subadd} we have the following upper and lower bounds for the difference of the integrand and this term:
\begin{equation} \label{exp_A1error1}
-F( Y_u + e^{\frac{T-u}{2}}(z_2 - z_1)) \leq \left[F(Y_u) -  H_{e^u}^{\lambda' / \lambda} (Y_u, Y_u + e^{\frac{T-u}{2}}(z_2 - z_1))\right] - \left[F(Y_u) - V_1^{e^{u/2}}(Y_u)\right] \leq 0.
\end{equation}
Recall from \eqref{e_Zdef} that $Z_K(Y)$ is defined as
\begin{equation}
Z_K(Y) = \exp \bigg(\int_0^K F(Y_u) - V_1^{e^{u/2}}(Y_u) \,du \bigg). \nonumber
\end{equation} 
Because both the exponential in $A_1$ and $Z_K(Y)$ are bounded above by $C_Z$, by \eqref{exp_A1error1} we have
\begin{align} \label{Zebd}
\big|A_1(x,w, \lambda,\lambda',K) - E^Y_x ( Z_K(Y)\, | \, Y_K = w )\big| \leq C_Z\, E_x^{Y} \bigg(\int_0^K F( Y_u + e^{\frac{T-u}{2}}(z_2 - z_1))\, du \, \bigg| \,Y_K = w \bigg).
\end{align}
We define $Z^e(Y,M,K)$ by
\begin{equation} \label{Zedef}
Z^e(Y,M,K) = Z_K(Y) \times E_{z_1}^{Y,\infty} \left( W_M(Y,z_2)\right).
\end{equation}
Using \eqref{Zebd} and proceeding as in the proofs of Lemmas~\ref{lemma_Zalimit} and \ref{lemma_Zclimit}, we obtain the following.
\begin{lemma} \label{lemma_Zelimit} For all $x, z_1, z_2 \in \R$ such that $z_1 \neq z_2$, we have
\begin{equation}
\delta^d_e(M,K) = \sup_{\lambda'> \bar{\lambda}(s)} \big| E^Y_x \big( Z^d_T(Y,\lambda'/\lambda,M,K) - Z^e(Y,M,K) \, \big| \, Y_T = z_1 \big) \big|  \to 0 \, \text{ as } \, T \to \infty \nonumber
\end{equation}
for all fixed $M$ and $K$ satisfying $0<M<K$.
\end{lemma}
From \eqref{Zedef}, we have
\begin{equation}
E^Y_x (Z^e(Y,M,K) \, | \, Y_T=z_1) = E^Y_x (Z_K(Y)\, | \, Y_T = z_1) \times E_{z_1}^{Y,\infty} \left( W_M(Y,z_2)\right). \nonumber
\end{equation}
Thus by Lemma~\ref{lemma_endpoint_indep}(b) and the fact that $Z_K(Y) \leq C_Z$ (and is a continuous function of $Y$) we have the following.
\begin{lemma}\label{lemma_Zelimit2}
For all $x,z_1,z_2 \in \R$,
\begin{equation} 
\delta^e_f(T,M,K) =  \big|E^Y_x (Z^e(Y,M,K) \, | \, Y_T=z_1) - E^{Y,\infty}_x (Z_K(Y)) E_{z_1}^{Y,\infty} \left( W_M(Y,z_2)\right) \big| \to 0 \, \text{ as } \,T \to \infty \nonumber
\end{equation}
for each fixed $0<M<K$.
\end{lemma}
We are now ready to establish the limiting form of $E^Y_x (\tilde{Z}_T(Y,\lambda'/\lambda)\, | \,Y_T = z_1)$ (provided $z_1 \neq z_2)$. Let $M>0$, $K>M$, $T \geq 2K + \tau(s)$ and $\lambda' > \bar{\lambda}(s)$. Bounding above by the sum of the $\delta$ terms in Lemmas \ref{lemma_Zalimit}-\ref{lemma_Zelimit2}, we have that
\begin{align} \label{bigdiff}
\big| E^Y_x (\tilde{Z}_{T}(Y,\lambda'/\lambda) \, &| \, Y_T=z_1) - E^{Y,\infty}_x (Z_K(Y)) E_{z_1}^{Y,\infty} \left( W_M(Y,z_2)\right) \big|  \nonumber
\\ &\leq \delta^{\sim}_a(K) + \delta^a_b(K,\lambda,\lambda') + \delta^b_c(M) + \delta^c_d(M,K) + \delta^d_e(M,K) + \delta^e_f(T,M,K).
\end{align}
Let $\epsilon > 0$. By Lemma~\ref{lemma_Zclimit}, we can choose $M_0>0$ to be sufficiently large such that $\delta^b_c(M) < \epsilon / 4$ for all $M \geq M_0$, and choose some $M \geq M_0$. By Lemma~\ref{lemma_Zalimit} and Lemma~\ref{lemma_Zdlimit}, we can then choose $K_0$ to be large enough such that $\delta^\sim_a(K) + \delta^c_d(M,K) < \epsilon / 4$ for all $K\geq K_0$. Fix $K > K_0$. Next, by Lemmas~\ref{lemma_Zelimit} and \ref{lemma_Zelimit2} we can choose $T_0 > 2K + \tau(s)$ such that for all $T \geq T_0$, $\delta^d_e(T,M,K) + \delta^e_f(T,M,K) < \epsilon / 4$. Finally, Lemma~\ref{lemma_Zblimit} allows us to choose $\lambda(\epsilon)>\bar{\lambda}(s)$ such that $T = \log(\lambda^2 (t-s)) \geq T_0$ and $\delta^a_b(K,\lambda,\lambda') < \epsilon / 4$ for all $\lambda, \lambda' \geq \lambda(\epsilon)$. We therefore obtain from \eqref{bigdiff} that 
\begin{equation}
\limsup_{\lambda,\lambda' \to \infty} \big|E^Y_x (\tilde{Z}_T(Y,\lambda'/\lambda)\, | \,Y_T = z_1) - E^{Y,\infty}_x (Z_K(Y)) E_{z_1}^{Y,\infty} \left( W_M(Y,z_2)\right)\big| < \epsilon \nonumber
\end{equation} 
for the $M$ and $K$ chosen above. This holds for all $\epsilon > 0$ for sufficiently large $M$ and $K$ (with $M<K$). It therefore holds that if $\lim_{M,K \to \infty, K > M} E^{Y,\infty}_x (Z_K(Y)) E_{z_1}^{Y,\infty} \left( W_M(Y,z_2)\right)$ exists, then $\lim_{\lambda,\lambda' \to \infty} E^Y_x (\tilde{Z}_T(Y,\lambda'/\lambda)\, | \,Y_T = z_1)$ exists and is equal to it. Thus it suffices to find the limit of $E^{Y,\infty}_x (Z_K(Y)) E_{z_1}^{Y,\infty} \left( W_M(Y,z_2)\right)$ as $M,K \to \infty$ with $M<K$. As the first term depends only on $K$ and the second depends only on $M$, we can consider the limits independently. First consider $E^{Y,\infty}_x (Z_K(Y))$. By \eqref{e_ZTbd}, $Z_K \uparrow Z_\infty \leq C_Z$, so the limit of the first term as $K \to \infty$ is $E^{Y,\infty}_x(Z_\infty(Y))$ by Monotone Convergence. We recall the definition of $W_M$ from \eqref{e_Wdef}. The integral in $W_M$ is monotone in $M$ and hence converges to the integral on $[0,\infty]$ as $M\to \infty$. Using the fact that $|W_M(Y,z_2)| \leq 1$ for all $M$ and continuity of the exponential, we can bring the limit inside, and  $E_{z_1}^{Y,\infty}(W_M(Y,z_2)) \to E_{z_1}^{Y,\infty}(W_\infty(Y,z_2))$ as $M \to \infty$. Thus we have shown that 
\begin{equation} \label{Ztildelim1}
\lim_{\lambda ,\lambda' \to \infty} E^Y_x (\tilde{Z}_T(Y,\lambda'/\lambda)\, | \,Y_T = z_1) = E^{Y,\infty}_x (Z_\infty(Y)) E_{z_1}^{Y,\infty} \left( W_\infty(Y,z_2)\right).
\end{equation}
Finally, recall that $\tilde{Z}_T(Y,\lambda'/\lambda)$ was in fact $\tilde{Z}^1_{T_1}(Y^1,\lambda'/\lambda)$. The analysis for $\tilde{Z}^2_{T_2}(Y^2,\lambda/\lambda')$ (under its respective conditional expectation) carries through unchanged. Since \eqref{Ztildelim1} holds for all $x \in \R$, we have therefore shown the following:
\begin{align} \label{Ztildelim2}
\text{For all } B^1_1, B^2_1 &\in \R \text{ and all } z_1,z_2 \in \R \text{ such that } z_1 \neq z_2 \nonumber
\\ & \lim_{\lambda, \lambda' \to \infty} E^{Y^1}_{B^1_1} \big( \tilde{Z}^1_{T_1}(Y^1,\lambda'/\lambda) \, \big| \, Y^1_{T_1} = z_1 \big) = E^{Y,\infty}_{B^1_1}(Z_\infty(Y))\, E^{Y,\infty}_{z_1}(W_\infty(Y,z_2)), \text{ and}
\\& \lim_{\lambda, \lambda' \to \infty} E^{Y^2}_{B^2_1} \big( \tilde{Z}^2_{T_2}(Y^2,\lambda/\lambda') \, \big| \, Y^2_{T_2} = z_2 \big) = E^{Y,\infty}_{B^2_1}(Z_\infty(Y))\, E^{Y,\infty}_{z_2}(W_\infty(Y,z_1)). \nonumber
\end{align}
To find the limit of $G$ it remains to identify the limits of $\Psi^{\lambda, \lambda'}_{B,s} (\sqrt{t-s}z_1 ,\sqrt{t-s} z_2 )$ and $Q(\lambda_1,\lambda_2,B^1,B^2,z_1,z_2)$. From \eqref{e_psi} we recall that the former prelimit is defined as
\begin{align}
&\Psi_{B,s}^{\lambda, \lambda'}(\sqrt{t-s}z_1 ,\sqrt{t-s} z_2 ) \nonumber
\\ &\hspace{10 mm}= h(\sqrt{t-s}z_1 + B_s, \sqrt{t-s}z_2 + B_s) \exp \left( - \int_0^s V^{\lambda, \lambda'}_{t-u} (\sqrt{t-s}z_1 + B_s - B_u, \sqrt{t-s}z_2 + B_s - B_u) \, du \right). \nonumber
\end{align}
For all $z_1, z_2 \in \R$ and all Brownian paths $(B_u, u \in[0,s])$, the obvious limit of the above as $\lambda,\lambda' \to \infty$ is obtained by replacing $V^{\lambda, \lambda'}_{t-u}$ with $V^{\infty,\infty}_{t-u}$. By monotonicity (in $\lambda,\lambda'$) of the integral and continuity of the exponential we can take the limit inside. Denoting the limit by $\Psi^{\infty, \infty}_{B,s} (\sqrt{t-s}z_1 ,\sqrt{t-s} z_2 )$, we have
\begin{equation} \label{psilimit}
\lim_{\lambda, \lambda' \to \infty} \Psi_{B,s}^{\lambda, \lambda'}(\sqrt{t-s}z_1 ,\sqrt{t-s} z_2 ) = \Psi_{B,s}^{\infty, \infty}(\sqrt{t-s}z_1 ,\sqrt{t-s} z_2 ).
\end{equation}
This leaves $Q(\lambda, \lambda',B^1,B^2,z_1,z_2)$, which we recall from (\ref{e_Qdef}) is defined by
\begin{equation} 
Q(\lambda, \lambda',B^1,B^2,z_1,z_2) = \exp \bigg(- \int_0^1 V^{1, \lambda'/ \lambda}_{u}(B^1_{u},B^1_{u} +  e^{T_1/2}( z_2 -   z_1)) + V^{1,\lambda/\lambda'} (B^2_u, B^2_u + e^{T_2/2}(z_1-z_2))\, du \bigg). \nonumber
\end{equation}
The integrand is the sum of two terms that are very similar; for now we restrict our attention to the first. In particular, we will show that
\begin{equation} \label{Qaux}
\lim_{T_1 \to \infty} \,\sup_{\lambda'> (t-s)^{-1/2}} \,\bigg| \exp \bigg(- \int_0^1 V^{1, \lambda'/ \lambda}_{u}(B^1_{u},B^1_{u} +  e^{T_1/2}( z_2 -   z_1)) \, du \bigg) - \exp \bigg( -\int_0^1 V^1_u(B^1_u) \,du \bigg) \bigg| = 0.
\end{equation}
We claim that since the second argument of the integrand goes to infinity, asymptotically the function resembles $V^1_u(B^1_u)$. To see this use both parts of Proposition~\ref{prop_mono_subadd} to conclude that
\begin{equation} 
0 \leq \left[ V_u^{1, c}(B_u^1, B_u^1 + e^{T_1/2}(z_2 - z_1)) - V_u^1(B_u^1) \right] \leq V^\infty_u(B_u^1 + e^{T_1/2}(z_2 - z_1)) \nonumber
\end{equation}
for all $c > 0$. $P^B_0$-a.s., there is a constant $R(\omega)>0$ such that $|B_u^1(\omega)| \leq R(\omega)$ for all $u \in [0,s]$. Provided $z_1\neq z_2$, for $\lambda$ sufficiently large, $e^{T_1/2}|z_2-z_1| \geq 2R$. Then for $\lambda$ sufficiently large and $\lambda,\lambda' > \bar{\lambda}(s)$,
\begin{align}
&\bigg| \,\exp \bigg(- \int_0^1 V^{1, \lambda'/ \lambda}_{u}(B^1_{u},B^1_{u} +  e^{T_1/2}( z_2 -   z_1)) \, du \bigg) - \exp \bigg(- \int_0^1 V^{1}_{u}(B^1_{u}) \, du \bigg) \bigg| \nonumber
\\ &\hspace{10 mm}\leq \int_0^1 V^\infty_u(e^{T_1/2}(z_2 - z_1)-R)\, du. \nonumber
\end{align}
The integrand is bounded above by $V^\infty_u(R)$, which is integrable on $[0,1]$. We take $\lambda \to \infty$ and apply Dominated Convergence; since $V^\infty_u(y) = u^{-1}F(u^{-1/2} y)$ and by \eqref{e_FODE}(iii), we have $\lim_{|y| \to \infty}V^\infty_u(y) = 0$, and hence limit of the above as $\lambda \to \infty$ (ie. as $T_1 \to \infty$) is zero. Thus \eqref{Qaux} holds. We handle the second term the integral in $Q(\lambda,\lambda',B^1,B^2,z_1,z_2)$ in an identical fashion, now with the roles of $\lambda$ and $\lambda'$ reversed, thereby establishing that 
\begin{align} \label{Qlimit}
&\text{For all } z_1,z_2 \in \R \text{ such that } z_1 \neq z_2, \,dP^{B^1}_0 dP^{B^2}_0\text{-a.s.},
\\ & \hspace{ 10 mm } \lim_{\lambda,\lambda' \to \infty} Q(\lambda,\lambda',B^1,B^2,z_1,z_2) = \exp \bigg( - \int_0^1 V^1_u(B_u^1) \,du \bigg) \exp \bigg( - \int_0^1 V^1_u(B_u^2) \, du \bigg). \nonumber
\end{align}
We have therefore found the limit of $G(\lambda,\lambda',s,B,B^1,B^2,z_1,z_2)$ and hence of $\Theta(\lambda,\lambda',s,B,B^1,B^2,z_1,z_2)$. In particular, recall the definitions \eqref{e_Gdef} and \eqref{e_thetadef}. From \eqref{densitylimit}, \eqref{Ztildelim2}, \eqref{psilimit} and \eqref{Qlimit}, we have shown that $dP^B_0 dP^{B^1}_0 dP^{B^2}_0$-a.s., for all $z_1, z_2\in \R$ such that $z_1 \neq z_2$, 
\begin{align} \label{endgame1}
 &\lim_{\lambda,\lambda' \to \infty} \Theta(\lambda,\lambda',s,B,B^1,B^2,z_1,z_2) = 
(t-s)^{-2\lambda_0} \Psi^{\infty, \infty}_{B,s} (\sqrt{t-s}z_1 ,\sqrt{t-s} z_2 )\exp \left( \int_0^1 -V_u^1(B_u^1) - V^1_u (B^2_u) \, du \right) \nonumber
\\ &\hspace{ 14 mm} \,\times \,  E_{z_1}^{Y,\infty}(W_\infty(Y,z_2)) E_{z_2}^{Y,\infty}(W_\infty(Y,z_1))   \,E_{B^1_1}^{Y,\infty}(Z_\infty(Y)) E_{B^2_1}^{Y,\infty} (Z_\infty(Y)) \psi_0(B^1_1)\, \psi_0(B^2_1)\, \psi_0(z_1)\,\psi_0(z_2). 
\end{align}
Thus by \eqref{lastreduction}, $\lim_{\lambda,\lambda' \to \infty} \N_0( (L^\lambda_t \times L^{\lambda'}_t)(h))$ exists and satisfies
\begin{equation} \label{endgame2}
\lim_{\lambda,\lambda' \to \infty} \N_0( (L^\lambda_t \times L^{\lambda'}_t)(h)) = \int_0^t \bigg[ \iint E^B_0( E^{B^1, B^2}_{(0,0)} [\lim_{\lambda,\lambda' \to \infty} \Theta(\lambda,\lambda',s,B,B^1,B^2,z_1,z_2)]) \, dm(z_1) \, dm(z_2) \bigg] ds.
\end{equation} 
To obtain the desired expression, we note that the terms in \eqref{endgame1} that depend on $B^1$ and $B^2$ can be collected in a constant. In particular, we define a constant $C_{\ref{thm_l2limit}}>0$ by
\begin{align} \label{e_Cdef}
C_{\ref{thm_l2limit}}^2 &= E_{(0,0)}^{B^1, B^2} \bigg( \exp \bigg( -\int_0^1 V_u^1(B^1_u) + V_u^1(B^2_u)\, du \bigg) E_{B^1_1}^{Y,\infty}(Z(\infty))E_{B^2_1}^{Y,\infty}(Z(\infty)) \psi_0(B^1_1) \,\psi_0(B^2_1) \bigg) \nonumber
\\ &=\bigg[E_0^B  \left(\exp \bigg(- \int_0^1 V_u^1(B_u)\, du \bigg) E_{B^1}^{Y,\infty}(Z_\infty(Y)) \, \psi_0(B_1)\right) \bigg]^2.
\end{align}
We also define a function $\rho(\cdot,\cdot)$ by
\begin{equation} \label{e_rhodef}
\rho(z_1, z_2) = E_{z_1}^{Y,\infty}(W_\infty(Y,z_2)) E_{z_2}^{Y,\infty}(W_\infty(Y,z_1)).
\end{equation}
It is clear that $\rho(\cdot,\cdot)$ is jointly continuous and bounded by $1$ from the definition of $W_\infty(Y,z)$. Thus by \eqref{endgame1},
\begin{align}
&E_{(0,0)}^{B^1, B^2} \left[ \lim_{\lambda,\lambda' \to \infty} \Theta(\lambda,\lambda',s,B,B^1,B^2,z_1,z_2) \right] = C_{\ref{thm_l2limit}}^2 (t-s)^{-2\lambda_0} \Psi^{\infty, \infty}_{B,s} (\sqrt{t-s}z_1 ,\sqrt{t-s} z_2 )  \rho(z_1,z_2) \psi_0(z_1)\,\psi_0(z_2). \nonumber
\end{align}
Substituting the above into \eqref{endgame2} completes the proof. \qed


\end{document}